\documentclass[a4paper,11pt]{amsart}
\usepackage[a4paper, left=1.8cm, right=1.8cm, top=2.5cm, bottom=2.5cm]{geometry}

\usepackage[english]{babel}
\usepackage{amsmath,amsfonts,amssymb,amsthm,relsize,setspace,nicefrac,yhmath,amscd,eucal}
\usepackage{amsrefs,mathrsfs}
\usepackage{mathtools,tkz-euclide,tikz-3dplot,tkz-tab}

\usepackage{tikz}
\usetikzlibrary{calc}

\usepackage{xcolor}
\usepackage{soul} % \st{}, \hl{}, etc.

\usepackage{enumitem}

\usepackage{comment}
\usepackage{anyfontsize}

\usepackage[colorlinks,linkcolor=red,citecolor=blue,urlcolor=blue,hypertexnames=true]{hyperref}

\usepackage[active]{srcltx}

\usepackage[toc,page]{appendix}

\newtheorem{theorem}{Theorem}

\newtheorem{lemma}[theorem]{Lemma}
\newtheorem{proposition}[theorem]{Proposition}

\newtheorem{remark}{Remark}

\theoremstyle{definition}

\newcommand{\be}{\begin{equation}}
\newcommand{\bel}[1]{\begin{equation}\label{#1}}
\newcommand{\ee}{\end{equation}}

\newcommand{\barr}{\begin{eqnarray}}
\newcommand{\earr}{\end{eqnarray}}
\newcommand{\bars}{\begin{eqnarray*}}
\newcommand{\ears}{\end{eqnarray*}}

\newtheorem{subn}{\name}

\newcommand{\bsn}[1]{\def\name{#1}\begin{subn}}
\newcommand{\esn}{\end{subn}}

\newtheorem{sub}{\name}[section]

\newcommand{\bs}{\begin{sub}}
\newcommand{\es}{\end{sub}}

\newcommand{\bth}[1]{\def\name{Theorem}\begin{sub}\label{t:#1}}
\newcommand{\blemma}[1]{\def\name{Lemma}\begin{sub}\label{l:#1}}
\newcommand{\bcor}[1]{\def\name{Corollary}\begin{sub}\label{c:#1}}
\newcommand{\bdef}[1]{\def\name{Definition}\begin{sub}\label{d:#1}}
\newcommand{\bprop}[1]{\def\name{Proposition}\begin{sub}\label{p:#1}}

\newcommand{\BA}{\begin{array}}
\newcommand{\EA}{\end{array}}
\newcommand{\BAN}{\renewcommand{\arraystretch}{1.2}
\setlength{\arraycolsep}{2pt}\begin{array}}
\newcommand{\BAV}[2]{\renewcommand{\arraystretch}{#1}
\setlength{\arraycolsep}{#2}\begin{array}}
\newcommand{\BSA}{\begin{subarray}}
\newcommand{\ESA}{\end{subarray}}

\newcommand{\BAL}{\begin{aligned}}
\newcommand{\EAL}{\end{aligned}}
\newcommand{\BALG}{\begin{alignat}}
\newcommand{\EALG}{\end{alignat}}
\newcommand{\BALGN}{\begin{alignat*}}
\newcommand{\EALGN}{\end{alignat*}}
\def\angb<#1>{\langle #1 \rangle}%% angle bracket
\newcommand {\rd}{\color{red}}
\let\.=\cdot
\let\0=\emptyset

\numberwithin{equation}{section}

\theoremstyle{definition}

\let\.=\cdot
\let\0=\emptyset

\newenvironment{formula}[1]{\begin{equation}\label{eq:#1}}
                       {\end{equation}\noindent}

\def\Fi#1{\begin{formula}{#1}}
\def\Ff{\end{formula}\noindent}

\setstcolor{red}

\usepackage{authblk}
\usepackage{changepage}

\title[]{The eigentheory for nonlocal  cooperative-advective system  and its role in the study of free boundary system for directional epidemic models}

\makeatletter
\author[1]{Soufiane Bentout} \let\Author\@author

\affil[1]{Department of Mathematics and Informatics, University of  Ain Temouchent,\\ Belhadj Bouchaib, BP 284 RP, 46000, Algeria\\
Engineering and Sustainable Development Laboratory, Faculty of Science and Technology,  University of Ain Temouchent, Ain Temouchent, 46000, Algeria}
\author[$2,*$]{Hoang-Hung Vo}
\affil[$2$]{Faculty of Mathematics and Applications, Saigon University, 273 An Duong Vuong st., Ward Choquan, Ho Chi Minh City, Viet Nam}
\email{\rd vhhung@sgu.edu.vn}
\email{soufiane.bentout@univ-temouchent.edu.dz}
\begin{document}

%%%%%%%%%%%%
%%Setting up the TITLE and AUTHOR
\date{\today}


%  \textheight=8 true in
%   \textwidth=6 true in
%   \oddsidemargin=-0.8 cm
%   \evensidemargin=-0.8 cm
\maketitle
\begin{adjustwidth}{1.5cm}{1.5cm}
\begin{center}
\let\thefootnote\relax\footnotetext{\textit{}}

\let\thefootnote\relax\footnotetext{\textit{$^{2,*}$Corresponding author. Email address: vhhung@sgu.edu.vn}}

\Author
\end{center}
\end{adjustwidth}

%  \textheight=9 true in
%   \textwidth=7 true in
%   \oddsidemargin=-0.8 cm
%   \evensidemargin=-0.8 cm
%\maketitle

\begin{abstract}
In this paper, in the absence of a variational structure but under sharper regularity requirements on the solutions, we propose and analyze a nonlocal cooperative reaction--diffusion system with free boundaries and drift terms,  motivated by the models of directional epidemic spread. We first establish the well-posedness of the local problem and the global existence and uniqueness of classical solutions in $C^1$ space, which present substantial analytical challenges compared with the earlier works~\cite{Du,Berestycki2016a,Berestycki2016b,Cao2019,NguyenVo2022,Tang2024a,Tang2024b}.

Next, we investigate the associated nonlocal eigenvalue problem and establish the existence, simplicity, qualitative properties, and asymptotic behavior of the principal eigenvalue.
The analysis relies on a combination of Fredholm theory, the Crandall–Rabinowitz bifurcation theorem, and Hadamard-type derivative formulae to describe its dependence on the model parameters and to elucidate its connection with the basic reproduction number~$R_0$.

Based on the spectral characterization, we demonstrate that the system exhibits a \emph{sharp vanishing--spreading dichotomy} in its long-term dynamics.
In particular, when $R_0 \le 1$, all solutions vanish, whereas for $R_0 > 1$, the outcome depends on the initial domain size~$h_0$ and the free-boundary expansion rate~$\mu$.
Moreover, there exists a critical habitat length~$\mathcal L^\ast$ such that if $h_0 < \mathcal L^\ast$, one can identify a critical threshold $\widehat{\mu}>0$  that vanishing occurs for $\mu \in (0,\widehat{\mu}]$, while spreading happens for $\mu > \widehat{\mu}$. In the spreading regime, solutions converge to the unique positive steady state, whereas in the vanishing regime they decay uniformly to zero. These results establish a \emph{rigorous theoretical framework} for understanding the threshold dynamics of cooperative--advective nonlocal systems with free boundaries arising in directional epidemic models, and potentially provide a mathematical foundation for further studies in \emph{epidemic modeling}, \emph{ecological invasion}, and \emph{population dynamics}.

.

\end{abstract}

%utilizing

%\subjclass[2010]{Primary 35B50, 47G20; secondary 35J60}

%\date{January 1, 2001 and, in revised form, June 22, 2001.}

%\dedicatory{This paper is dedicated to Professor E. Zuazua on his 60th birthday with our deep admiration}

\tableofcontents

\section{ Introduction}

Infectious diseases and ecological invasions have posed persistent threats to human society and natural ecosystems for centuries. Their spatial spread is shaped not only by local random movement but also by long-range dispersal and directed drift driven by environmental flows or human activities. Capturing these mechanisms in mathematical models is essential for elucidating threshold dynamics and predicting long-term outcomes. Classical reaction–diffusion frameworks based on the Laplacian provide a baseline, but they can be inadequate when movements are inherently nonlocal or exhibit preferential directions. Waterborne diseases, including cholera, botulism, giardiasis, diarrhoea, dysentery,
typhoid, campylobacteriosis and salmonellosis, remain a major threat to global public health,
particularly in developing countries. These infections are transmitted when pathogenic
microorganisms contained in contaminated freshwater, surface water, lakes or aquatic reservoirs
enter human populations. Recent estimates indicate that waterborne diseases contribute approximately
3.6\% of the total global burden of disease measured in DALYs (disability-adjusted life years) and are
responsible for about 1.5 million deaths annually \cite{Weiss2016}. In the case of cholera, the impact
is especially severe, with an estimated 2.9 million cases and 95,000 deaths each year across
69 endemic countries \cite{Ali2015}. To analyze and predict the spread of waterborne infections, numerous mathematical
models have been developed. Early approaches, typically formulated as systems of
ordinary differential equations, employed community--water compartmental structures
to represent pathogen transmission during cholera outbreaks
\cite{Wang2015, Andrews2011, Codeco2001, Hartley2006, Mukandavire2011,
Tien2010, Bertuzzo2011, Eisenberg2013, Wang2015, Pascual2017,
Mukandavire2013, Chao2011, Bertuzzo2017, Tien2011, Eisenberg2016, Tuite2011, King_2018}. To capture more realistic dynamics, subsequent studies incorporated
environmental and spatial heterogeneity, thereby extending the modeling framework
beyond purely temporal dynamics \cite{Bertuzzo2008, Bertuzzo2010, Righetto2013,
Mari2012}. In particular, spatially explicit models have been proposed using patches,
networks, or directed graphs \cite{Bertuzzo2007, Righetto2011, Casagrandi2006,
Bertuzzo2016}, where nodes represent localities such as towns or villages and edges
describe transport-mediated connections. Within these frameworks, the dynamics of
waterborne pathogens are governed by systems of differential equations that explicitly
account for dispersal and directed transport processes. Comprehensive reviews of
mathematical modeling for cholera outbreaks are given in \cite{Feng2016}, while
practical implications for public health interventions are summarized in \cite{Codeco2009}.

The mathematical modeling of man--environment--man epidemic systems originally traces back
to the pioneering works of Capasso and his collaborators
\cite{Capasso1978,Capasso1979,Capasso1982,Capasso1988,Capasso1997}.
Their central idea was to distinguish explicitly between the dynamics of the
infected human population and the pathogen concentration in the environment.
This framework provided one of the first rigorous mathematical models for
waterborne diseases such as cholera. In its simplest form, the core model can be written as a system of two ordinary
differential equations:
\begin{equation}\label{eq:capasso-ode}
\begin{cases}
\dot z_1(t) = -\,a_{11}\, z_1(t) + a_{12}\, z_2(t),\\[0.4em]
\dot z_2(t) = -\,a_{22}\, z_2(t) + g\!\big(z_1(t)\big),
\end{cases}\quad\quad t>0
\end{equation}
where $z_1(t)$ denotes the concentration of pathogens in the environment and
$z_2(t)$ the density of infected individuals. The parameter $1/a_{11}$
represents the mean lifetime of the pathogen in the aquatic medium,
$1/a_{22}$ the average duration of infection, and $a_{12}$ the rate at which
infected individuals contribute to the environmental contamination. The
function $g(\cdot)$ denotes the force of infection, typically assumed to be
increasing, with $g(0)=g'(0)=0$, and saturating at high contamination levels.

To account for spatial spread, Capasso and Kunisch \cite{Capasso1988} extended
this model to a reaction--diffusion framework, in which the environmental
pathogen diffuses in space while the infected population remains localized:
\begin{equation}\label{eq:capasso-pde}
\begin{cases}
\partial_t u_1(x,t) = d\,\Delta u_1(x,t) - a_{11}\,u_1(x,t) + a_{12}\,u_2(x,t),\\[0.4em]
\partial_t u_2(x,t) = -\,a_{22}\,u_2(x,t) + a_{21}\,g\!\big(u_1(x,t)\big),
\end{cases}
\qquad x \in \Omega, \; t>0,
\end{equation}
where $u_1(x,t)$ denotes the environmental pathogen density and $u_2(x,t)$
the density of infected individuals. Appropriate boundary conditions
(Dirichlet or Neumann) are imposed depending on the epidemiological
context. Later, an important models for waterborne disease dynamics was proposed by Hsu and Yang~\cite{Hsu2013}, who studied the existence, uniqueness, monotonicity and asymptotic behavior of traveling waves to the following reaction--diffusion system:

\begin{equation}
\label{eq:HYmodel}
\begin{cases}
\partial_t u = d_1 \Delta u - a u + H(v), & t>0,\; x\in \mathbb{R},\\[0.4em]
\partial_t v = d_2 \Delta v - b v + G(u), & t>0,\; x\in \mathbb{R},
\end{cases}
\end{equation}
where $u(t,x)$ denotes the pathogen density in the aquatic environment, while $v(t,x)$ represents the density of infective hosts.
Here $d_1,d_2>0$ are diffusion coefficients, and $a,b>0$ are the pathogen decay and host removal (recovery/mortality) rates, respectively.
The feedback functions $H,G:\mathbb{R}_+\to\mathbb{R}_+$ are assumed to satisfy
\begin{equation}
\label{eq:HYassump}
\left\{
\begin{aligned}
& H,G \in C^2([0,\infty)), \quad H(0)=G(0)=0,\quad H'(z),G'(z)>0,\ \forall z\geq 0, \\[0.4em]
& H''(z),G''(z)<0,\ \forall z>0; \exists\,\bar z>0:\ G\!\left(\tfrac{H(\bar z)}{a}\right)< b\,\bar z .
\end{aligned}
\right.
\tag{1.4}
\end{equation}
Here, the variable $u$ describes the concentration of pathogens (e.g.\ \emph{Vibrio cholerae}) in water sources and its dynamics combine spatial diffusion, natural clearance at rate $a$, and  $H(v)$ captures nonlinear shedding dynamics, which may saturate due to limits on pathogen output, stage-dependent infectivity, or environmental capacity.
The variable $v$ denotes the infected host population, which diffuses spatially with coefficient $d_2$, leaves the class at rate $b$, $G(u)$ represents nonlinear dose--response infection, reflecting that risk saturates at high pathogen concentrations. The concavity conditions $H''<0$ and $G''<0$ capture biological saturation: additional infectives do not increase shedding indefinitely, and infection risk does not grow linearly with pathogen concentration but levels off due to immune responses or behavioral constraints. In contrast with linear or bilinear incidence models, the use of two separate concave functions, $H(v)$ and $G(u)$, allows the shedding and infection processes to be modeled independently. This dual structure yields richer threshold behavior, backward bifurcation, and realistic wave propagation dynamics.

A modern development of waterborne epidemic modeling is due to Fitzgibbon,
Morgan, Webb and Wu \cite{Fitzgibbon2020}, who formulated a coupled
community--river system to describe the 2010 cholera outbreak in Haiti. In
their framework, the human population of each community is divided into
susceptible $S_j$, infected $I_j$, and a local pathogen reservoir $B_j$,
while $b(x,t)$ denotes the pathogen concentration transported through the
river system. The complete system reads
\begin{equation}\label{eq:fitzgibbon-haiti}
\left\{
\begin{aligned}
\dfrac{dS_j}{dt} &= -\,\frac{\beta_j S_j B_j}{k_j+B_j},
&& t>0,\ j=1,2,\dots,n, \\[0.6em]
\dfrac{dI_j}{dt} &= \frac{\beta_j S_j B_j}{k_j+B_j} - \gamma_j I_j,
&& t>0,\ j=1,2,\dots,n, \\[0.6em]
\dfrac{dB_j}{dt} &= r_j I_j
+ \mu_j\,\frac{1}{d_j}\!\int_{x_j}^{x_j+d_j}\! b(x,t)\,dx
- (\rho_j+\delta_j)B_j,
&& t>0,\ j=1,2,\dots,n, \\[0.8em]
\partial_t b(x,t) &= \varepsilon\,\partial_{xx} b - q\,\partial_x b
+ \sum_{j=1}^n \chi_j(x)\Big(\tfrac{\rho_j}{d_j}B_j - \mu_j b\Big) - \delta_r b,
&& t>0,\ x\in\mathbb{R},
\end{aligned}
\right.
\end{equation}
with initial conditions
\[
b(x,0)=b_0(x), \qquad S_j(0)=S_{j0},\ I_j(0)=I_{j0},\ B_j(0)=B_{j0}.
\]
Here, $\chi_j(x)=\mathbf{1}_{[x_j,x_j+d_j]}(x)$ denotes the indicator of the
river segment corresponding to community $j$. The parameters $\varepsilon$ and
$q$ represent the diffusion and advection coefficients in the river,
respectively. The terms $\rho_j$ and $\mu_j$ model the exchange of pathogens
between the community reservoir $B_j$ and the river compartment $b$. A central feature of this model is the explicit advection term
$-q\,\partial_x b$ in the river dynamics. While diffusion captures only random
mixing of pathogens in water, the advection term represents the directed
transport due to river currents. In practice, such directed transport is the
primary mechanism by which cholera bacteria spread from rural upstream regions
to densely populated downstream communities. Ignoring advection would
significantly underestimate the velocity and spatial extent of outbreaks.
Including this effect allows the model to reproduce the observed rapid spread
of cholera along the Haitian river network, and highlights the crucial role of
hydrological flows in shaping epidemic dynamics and designing intervention
strategies. In this model, $S_j$ and $I_j$ denote the numbers of susceptible and infected
individuals in community $j$, while $B_j$ represents the pathogen load in the
local water reservoir of that community. The variable $b(x,t)$ describes the
pathogen concentration in the river, with $x$ representing position along the
river domain. Disease transmission follows a saturated incidence rate of the
form $\beta_j S_j B_j /(k_j + B_j)$, where $\beta_j$ and $k_j$ are transmission
parameters. Infected individuals recover at rate $\gamma_j$ and shed pathogens
into the local water source at rate $r_j$. The exchange of pathogens between
the river and community reservoirs is modeled by two parameters: $\mu_j$, the
rate of river-to-community transfer over the segment $[x_j, x_j+d_j]$ (expressed
as the average $\tfrac{1}{d_j}\int b$), and $\rho_j$, the rate of
community-to-river transfer. Natural pathogen decay occurs at rates $\delta_j$
in the local reservoirs and $\delta_r$ in the river. The river dynamics further
include diffusion with coefficient $\varepsilon$ and directed transport with
velocity $q$, represented by the advection term $-q\,\partial_x b$, which
captures downstream flow in the positive $x$-direction. Finally, $\chi_j$
denotes the indicator function identifying the river segment associated with
community $j$.

To  deeper understand  the long-term behavior of the spatial spread of waterborne pathogens and the expansion of the infected region, many mathematical models are proposed to explore the different dynamics of the infectious diseases \cite{Bertuzzo2007, Bertuzzo2008, Bertuzzo2010, Bertuzzo2011, Bertuzzo2011, Bertuzzo2017, Chao2011, Codeco2001, Codeco2009, DuNi2020, Eisenberg2013, Eisenberg2016, Feng2016, Fitzgibbon2020, Grassly2006, Hartley2006, King2008, King_2018, Mari2012, Mukandavire2011, Mukandavire2013, Pascual2017, Righetto2011, Righetto2013, Tien2010, Tien2011, Wang2015, Weiss2016, Wu2016}. The free boundary problems have  been also proposed for this model \cite{DuNi2020,Zhao2020, Du, LiXuZhang2017, NguyenVo2022}. In particular, the long time dynamics of a nonlocal diffusion single population model with advection and free boundaries was recently investigated by Tang and Dai \cite{Tang2024a}, their model read as as follows :
 \begin{equation}\label{eq:model}
\begin{cases}
u_t = d \displaystyle \left[\int_{g(t)}^{h(t)} J(x-y)u(t,y)\,dy - u(t,x)\right]
      - q u_x + f(t,x,u),
      & t>0,\; x\in(g(t),h(t)), \\[2ex]
u(t,g(t)) = u(t,h(t)) = 0, & t>0, \\[1ex]
h'(t) = \mu \displaystyle \int_{g(t)}^{h(t)} \int_{h(t)}^{+\infty} J(x-y)u(t,x)\,dy\,dx, & t>0, \\[2ex]
g'(t) = -\mu \displaystyle \int_{g(t)}^{h(t)} \int_{-\infty}^{g(t)} J(x-y)u(t,x)\,dy\,dx, & t>0, \\[2ex]
h(0) = -g(0) = h_0, \quad u(0,x) = u_0(x), & x\in[-h_0,h_0],
\end{cases}
\end{equation}
where the advection term represents a directional drift caused by wind, animal migration, or the spreading of epidemics through water or river flows. Inspired by the mentioned works,  in this paper, we investigate the spectral theory and the long-time dynamics of an advection–free boundary system involving coupled nonlocal diffusions and a concave nonlinearity in heterogeneous environment, which is precisely given by :
 % requires \usepackage{changepage}
\[
\begin{cases}\label{eq:main-problem}
u_t &= d_1\!\left[\displaystyle\int_{g(t)}^{h(t)} J_1(x-y)u(t,y)\,dy - u(t,x)\right]
+ p\,u_x - a(x)u(t,x) + H\!\left(v(t,x)\right),
 t>0, g(t)<x<h(t),\\[2mm]
v_t &= d_2\!\left[\displaystyle\int_{g(t)}^{h(t)} J_2(x-y)v(t,y)\,dy - v(t,x)\right]
+ q\,v_x - b(x)v(t,x) + G\!\left(u(t,x)\right),
 t>0,~ g(t)<x<h(t),\\[2mm]
u(t,x) &= v(t,x)=0,
 t>0,~ x\in\{g(t),h(t)\},\\[2mm]
h'(t) &= \mu\!\left(
\displaystyle\int_{g(t)}^{h(t)}\!\!\int_{h(t)}^{\infty}
J_1(x-y)u(t,x)\,dy\,dx
+ \rho\!\!\int_{g(t)}^{h(t)}\!\!\int_{h(t)}^{\infty}
J_2(x-y)v(t,x)\,dy\,dx
\right),
 t>0,\\[2mm]
g'(t) &= -\mu\!\left(
\displaystyle\int_{g(t)}^{h(t)}\!\!\int_{-\infty}^{g(t)}
J_1(x-y)u(t,x)\,dy\,dx
+ \rho\!\!\int_{g(t)}^{h(t)}\!\!\int_{-\infty}^{g(t)}
J_2(x-y)v(t,x)\,dy\,dx
\right),
 t>0,\\[2mm]
-g(0) &= h(0)=h_0,\quad
u(0,x)=u_0(x),\quad
v(0,x)=v_0(x),
 x\in[-h_0,h_0].
\end{cases}
\]

From an epidemiological viewpoint, in this model, $J_1,J_2$ are symmetric dispersal kernels, which permit jumps over finite or even large distances, capturing realistic scenarios in which pathogens or infected individuals can be transported by environmental factors (e.g. contaminated water, wind, or human mobility) to non-adjacent locations. This leads to faster and wider epidemic spread compared to purely local diffusion.

The advection terms $(p u_x,\; q v_x)$ represent the \emph{directed movement or transport} of infected individuals or pathogens. In the real world, the advection corresponds to biased movement caused by environmental or social factors, such as river or water flow in the case of waterborne diseases (e.g., cholera), prevailing winds in airborne diseases, or directional human migration. Without advection, the spread is symmetric; with advection, the infection front advances faster in the downstream direction.

The heterogeneous environment coefficients $a(x), b(x)$ act as \emph{spatially dependent removal or decay rates}. From an epidemiological aspect, $a(x)$ can represent the recovery or death rate of infected individuals, which may vary according to local healthcare quality or demographic conditions while $b(x)$ describes the local decay or clearance rate of pathogens in the environment, which depends on climatic or ecological conditions (e.g., temperature, sunlight, dryness). This spatial heterogeneity models unequal control measures or environmental survival conditions across regions.

The free boundary functions $h'(t), g'(t)$ describe the \emph{expansion speed of the infected domain}. The boundary conditions link the front propagation to the gradients of $u$ and $v$ at the edges. Epidemiologically, $h'(t)$ corresponds to the rightward expansion speed of the epidemic, while $g'(t)$ corresponds to the leftward speed. If the infection intensity near the boundary remains high, the epidemic front expands indefinitely (\emph{successful invasion}); if it weakens, the boundary stabilizes and the infection may vanish (\emph{containment scenario}).

Finally, the concave nonlinear terms $H(v)$ and $G(u)$ represent the cooperative cross-interactions between the two populations, ensuring that the presence of one species enhances the growth of the other. Near the disease-free or low-density state, the linearization of these functions determines the threshold dynamics, such as the basic reproduction number and invasion conditions. At higher densities, however, the nonlinear behavior of $H$ and $G$ introduces saturation effects that prevent unbounded growth and stabilize the system at biologically meaningful equilibria. The nonlinearities are increasing and strictly concave on $(0,\infty)$, reflecting diminishing returns: as the interacting population grows, the marginal effect on growth decreases. This concavity is crucial for modeling realistic epidemic or ecological interactions, as it ensures boundedness and stabilizing effects in the system.

Throughout this paper, we assume the following conditions :

\begin{enumerate}
    \item[\textbf{(J)}] The kernels $J_i, i\in\{1,2\}$ satisfy: $J_i \in C^1(\mathbb{R}) \cap L^\infty(\mathbb{R})$, $J_i \geq 0$, $J_i(x) = J_i(-x)$ for all $x\in\mathbb{R}$, $J_i(0) > 0$ and
         \[
        \int_{\mathbb{R}} J_1(x)\, dx = \int_{\mathbb{R}} J_2(x)\, dx = 1.
        \]

    \item[\textbf{(GH)}] Given functions $G, H: \mathbb{R}^+ \to \mathbb{R}^+$ verify the following assumptions:
    \begin{itemize}
        \item $H, G \in C^2([0, \infty))$,
        \item $G(0) = H(0) = 0$ and $H(z), G(z) > 0, \forall z \geq 0$,
        \item There exist constants $a, b > 0$ such that $G\left(\frac{H(z)}{a}\right) < bz$ for all $z \geq 0$,
        \item $H'(z) < 0$ and $G'(z) < 0$ for all $z \in (0, \infty)$.
    \end{itemize}

    \item[\textbf{(IC)}] We also assume the conditions on the initial conditions:
    \begin{itemize}
        \item $u_0, v_0 \in C([-h_0, h_0])$, with $u_0(\pm h_0) = v_0(\pm h_0) = 0$,
        \item $u_0(x), v_0(x) > 0$ for all $x \in (-h_0, h_0)$.
    \end{itemize}
\end{enumerate}

\medskip
\noindent\textbf{Linearization and spectral threshold.}
Linearizing \eqref{eq:main-problem} at the disease-free state $(u,v)=(0,0)$ over a fixed domain $\Omega\subset\mathbb{R}$ (e.g., a snapshot or a frozen free boundary) yields the cooperative system
\begin{equation}\label{eq:lin}
\partial_t
\begin{pmatrix} u \\ v \end{pmatrix}
=
\underbrace{\begin{pmatrix}
d_1(N_1-I)+p\,\partial_x & 0\\
0 & d_2(N_2-I)+q\,\partial_x
\end{pmatrix}}_{\displaystyle \mathcal{T}}
\begin{pmatrix} u \\ v \end{pmatrix}
+
\underbrace{\begin{pmatrix}
-a(x) & H'(0)\\
G'(0) & -b(x)
\end{pmatrix}}_{\displaystyle \mathcal{M}(x)}
\begin{pmatrix} u \\ v \end{pmatrix},
\quad x\in\Omega.
\end{equation}
The associated stationary eigenproblem (seeking exponential solutions $e^{-\lambda t}\phi$) reads
\begin{equation}\label{eq:eig}
\Big(\mathcal{T}+\mathcal{M}(\cdot)\Big)\,\Phi + \lambda\,\Phi = 0,
\qquad \Phi=(\phi_1,\phi_2)^{\!\top}\!, \qquad \Phi\not\equiv 0.
\end{equation}
The \emph{principal eigenvalue} $\lambda_p$ (if it exists) governs linear stability and serves as a surrogate for the basic reproduction number $R_0$ via the sign relation
\begin{equation}\label{eq:R0-lambda}
R_0>1 \iff \lambda_p<0,
\qquad
R_0<1 \iff \lambda_p>0.
\end{equation}
We underline that, in previous studies \cite{NguyenVo2022,ThuVo2025},  for scalar nonlocal problems without advection, the principal eigenvalue $\lambda_p$ enjoys a variational formula and it can often be written in $L^2(\Omega)$; for cooperative systems these formulas extend to $E:=L^2(\Omega)\times L^2(\Omega)$ by
\begin{equation}\label{eq:var}
\lambda_p
=
-\,\sup_{\|\Phi\|_{E}=1}
\left\langle
\begin{pmatrix}
d_1(N_1-I) & 0\\[0.15em]
0 & d_2(N_2-I)
\end{pmatrix}\Phi
+
\begin{pmatrix}
-a(\cdot) & H'(0)\\
G'(0) & -b(\cdot)
\end{pmatrix}\Phi
\,,\,\Phi\right\rangle_{E}
\ \ \
\end{equation}
where $\langle\cdot,\cdot\rangle_E$ is the $L^2$ inner product on $E$. However, the presence of the advection terms $(p u_x,\; q v_x)$ in~\eqref{eq:lin}
makes the study of the existence, simplicity, and qualitative properties of the
principal eigenvalue significantly more challenging, mainly due to the sharper
$C^1(\Omega)$ regularity requirement of the eigenfunction and the lack of a
variational characterization. In order to overcome this, we may consider $-\mathcal{T}-\mathcal{M}$ in~\eqref{eq:lin} as an accretive perturbation of a positive compact operator (or of a compact resolvent generator), allowing one to invoke non-selfadjoint spectral theory and perturbation arguments \cite{Kato1995}. Advection terms $p\,\partial_x$, $q\,\partial_x$ typically break self-adjointness; nonetheless, one can restore a comparison structure by suitable transforms (e.g., exponential weights) or by working with the resolvent and positive cones, in the spirit of \cite{KreinRutman1948,Burger1988,Kato1995} and this represents a completely new approach compared with previous studies \cite{Berestycki2016a,Du,Tang2024a, Tang2024b, ThuVo2025, NguyenVo2022, DuNi2020, Cao2019, Bao2017, Ahn2016}.

\medskip
\noindent\textbf{Free boundaries and spreading--vanishing.}
When the habitat is allowed to expand through flux-driven free boundaries, as in \eqref{eq:main-problem}, the long-time dynamics exhibit a sharp dichotomy: either the solution vanishes uniformly (the front stalls), or it \emph{spreads} and a positive profile invades with the boundaries moving outward. Although, this paradigm has been established for several nonlocal models without advection \cite{Tang2024b, NguyenVo2022, DuNi2020, Cao2019, Du,Zhao2020}, again introducing advection requires significantly deeper techniques due to the requirement of sharper regularity requirements on time-dependent and stationary solutions, spectral thresholds and the front motion (the integrals defining $g'(t)$ and $h'(t)$), with qualitative consequences for the spreading speed and critical domain size. The main contributions of the present work are:

\begin{itemize}
\item[(i)] \emph{Well-posedness with advection.} We establish local, global existence and uniqueness of classical solutions to \eqref{eq:main-problem} under (J), (GH), (IC) in $C^1$ functional space. Key tools include comparison principle adapted to jump–advection operators and a priori sharper estimates on $C^1$ in order to apply the Banach fixed point theorem.

\item[(ii)] \emph{Eigentheory with drift.} We establish the existence and simplicity of the principal eigenvalue for the linearized problem \eqref{eq:eig} on bounded domains in the \emph{absence of a variational structure} (cf.\ \cite{Berestycki2016a,Cao2019,NguyenVo2022,ThuVo2025}).
We further analyze its qualitative dependence on parameters, employing \emph{Fredholm theory}, the \emph{Crandall--Rabinowitz} bifurcation theorem, and \emph{Hadamard-type} derivative formulae as developed in the recent work of Benguria et al. \cite{Benguria2024} to establish differentiability with respect to those parameters.
Our approach necessitates a significantly more refined eigentheoretic framework than that used in the existing literature.

\item[(iii)] \emph{Basic reproduction number.} We define $R_0$ through the spectral bound of the next-generation operator and show the threshold equivalence \eqref{eq:R0-lambda}. For constant coefficients we derive implicit/explicit formulas for $R_0$, clarifying the role of $p,q$ and of kernel features.

\item[(iv)] \emph{Spreading--vanishing criteria with advection.} We identify critical thresholds in terms of the initial half-width $h_0$ and the free-boundary parameters $(\mu,\rho)$ in \eqref{eq:main-problem}.
In particular, we prove that \emph{vanishing} occurs when $R_0 \le 1$, whereas for $R_0 > 1$, the long-time behavior depends delicately on $(h_0,\mu,\rho)$.
In this case, a \emph{sharp transition} between vanishing and spreading regimes is established, which considerably extends the known results of the long-time dynamics for free-boundary systems with coupled advection terms;
see, for instance, \cite{Tang2024a,Tang2024b,NguyenVo2022,DuNi2020,Cao2019,Du}.

\end{itemize}

The principal results for system \eqref{eq:main-problem} are stated precisely as follows.

\begin{theorem}[\textbf{Global existence and uniqueness}]\label{thm:global-existence}
Suppose that $J_1, J_2$ satisfy assumption \textbf{(J)}, and assumptions \textbf{(A1)}–\textbf{(A2)} hold. Then, for any \( h_0>0\), problem~\eqref{eq:main-problem} admits a unique positive solution \( (u, v, g, h) \) defined for all \( t > 0 \).

Moreover, for any fixed \( T > 0 \), the solution components satisfy
\[
u \in X_{u_0, g, h}, \quad v \in X_{v_0, \infty}, \quad g \in \mathcal{G}_{h_0,T}, \quad h \in \mathcal{H}_{h_0,T}.
\]
\end{theorem}

\medskip

To analyze the spreading–vanishing dichotomy, we first introduce the following eigenvalue problem:
\begin{equation}\label{eq:maineigenrvalue}
\begin{cases}
d_1 \displaystyle\int_{-\mathcal{Z}}^{\mathcal{Z}} J_1(x-y)\psi_1(y)\,dy - d_1 \psi_1(x) - a(x) \psi_1(x) +H'(0) \psi_2(x)
+ p \psi_1^\prime(x) = \lambda \psi_1(x), & x \in [-\mathcal{Z}, \mathcal{Z}], \\[2ex]
d_2 \displaystyle\int_{-\mathcal{Z}}^{\mathcal{Z}} J_2(x-y)\psi_2(y)\,dy - d_2 \psi_2(x) - b(x) \psi_2(x) +G'(0) \psi_1(x)
+ q \psi_2^\prime(x) = \lambda \psi_2(x), & x \in [-\mathcal{Z}, \mathcal{Z}].
\end{cases}
\end{equation}

\begin{theorem}
[\textbf{Principal eigenvalue}]\label{Principlaeignevalue}

The problem \eqref{eq:maineigenrvalue} admits a principal eigenvalue
$\lambda^\ast$, which is simple and associated with a strictly positive eigenfunction in $C^1([-\mathcal{Z}, \mathcal{Z}])$.

\end{theorem}

\begin{theorem}[\textbf{Qualitative structure and parameter dependence of the principal eigenvalue}]\label{thm:eigen-properties}
Let $p,q>0$, $a,b\in L^\infty(\mathbb R)$, and $H'(0),G'(0)\geq 0$.
Assume that the nonlocal kernels $J_1,J_2$ satisfy \textnormal{\textbf{(J)}} so that each integral operator
$\mathscr J_l$ is positive and compact on $X_l:=L^2([-\mathcal Z,\mathcal Z];\mathbb R)$ for all $\mathcal Z>0$.
Denote by $\lambda^\ast(\mathcal Z)$ the principal eigenvalue of \eqref{eq:maineigenrvalue} on $[-\mathcal Z,\mathcal Z]$ and define
\[
\mathcal R_0:=\rho\!\big(\mathscr J(-\mathscr T)^{-1}\big)
   =\lim_{\mathcal Z\to\infty}\rho_{\mathcal Z}(0),
\]
where $\rho_{\mathcal Z}(0)$ denotes the spectral radius of the corresponding next–generation operator on $[-\mathcal Z,\mathcal Z]$.

Then the following hold:
\begin{enumerate}[label=(\roman*),wide,labelwidth=!,labelindent=0pt]

\item \textbf{Monotonicity and continuity.}
The mapping $\mathcal Z\mapsto\lambda^\ast(\mathcal Z)$ is strictly increasing and continuous on $(0,\infty)$.

\medskip
\item \textbf{Threshold characterization.}
There exists a sharp critical domain size $\mathcal Z_\ast\in(0,\infty]$ such that
\[
\lambda^\ast(\mathcal Z)<0 \quad\text{for } \mathcal Z<\mathcal Z_\ast,
\qquad
\lambda^\ast(\mathcal Z)>0 \quad\text{for } \mathcal Z>\mathcal Z_\ast.
\]
Moreover, $\mathcal Z_\ast<\infty$ if and only if $\mathcal R_0>1$, in which case
$\lambda^\ast(\mathcal Z_\ast)=0$ (equivalently, $\rho_{\mathcal Z_\ast}(0)=1$);
if $\mathcal R_0\le1$, then $\lambda^\ast(\mathcal Z)<0$ for all $\mathcal Z>0$.

\medskip
\item \textbf{Differentiability with respect to diffusion rates.}
Fix $\mathcal Z>0$ and let $d=(d_1,d_2)\in U\subset(0,\infty)^2$.
For each $d$, define
\[
\mathcal L_d:=\mathscr J_d-\mathscr T_d
\]
acting on a Banach lattice $X$. Suppose $\mathcal L_d$ admits a simple, isolated principal eigenvalue $\lambda^\ast(d)$ with normalized positive eigenfunction $\varphi(d)\in\mathcal D\subset X$ and positive adjoint eigenfunction $w^*(d)\in X^*$ satisfying
$\langle w^*(d),\varphi(d)\rangle=1$.
If $d\mapsto\mathcal L_d$ is $C^1$ in the operator norm (in particular affine in $d$), then $\lambda^\ast$ is of class $C^1$ on $U$, and
\[
\frac{\partial\lambda^\ast}{\partial d_j}(d)
=\big\langle w^*(d),\,\partial_{d_j}\mathcal L_d[\varphi(d)]\big\rangle,
\qquad j=1,2.
\]
For the present operator, one has explicitly
\[
\mathcal L_d(\psi_1,\psi_2)(x)
=\begin{pmatrix}
d_1\!\displaystyle\int J_1(x-y)\psi_1(y)\,dy - d_1\psi_1(x) - a(x)\psi_1(x) + H'(0)\psi_2(x) + p\psi_1'(x)\\[6pt]
d_2\!\displaystyle\int J_2(x-y)\psi_2(y)\,dy - d_2\psi_2(x) - b(x)\psi_2(x) + G'(0)\psi_1(x) + q\psi_2'(x)
\end{pmatrix},
\]
and therefore
\[
\frac{\partial\lambda^\ast}{\partial d_j}(d)
= \int_{-\mathcal Z}^{\mathcal Z} w_j^*(x)
\Big(\int_{-\mathcal Z}^{\mathcal Z}J_j(x-y)\varphi_j(y)\,dy - \varphi_j(x)\Big)\,dx,
\qquad j=1,2.
\]
\end{enumerate}
\end{theorem}

\begin{theorem}[\textbf{Vanishing-Spreading}]\label{thm:threshold-mu}
Suppose $R_0>1$ and $h_0<\mathcal L^\ast$. Then there exists a threshold $\hat\mu>0$, depending on the initial data $(u_0,v_0)$, such that
\begin{itemize}
  \item vanishing occurs for every $\mu\in(0,\hat{\mu}]$ (i.e.\ $h_\infty-g_\infty<\infty$),
  \item spreading occurs for every $\mu\in(\hat{\mu},\infty)$ (i.e.\ $h_\infty-g_\infty=\infty$ and $(u,v)\to(\mathscr U^\ast,\mathscr V^\ast)$).
\end{itemize}
\end{theorem}

\medskip
\noindent\textbf{Organization of the paper. }
Section~2 presents the precise assumptions, several preliminary results, and the well-posedness of problem~\eqref{eq:main-problem}.
Section~3 develops the spectral theory and establishes the existence, simplicity, and qualitative properties of the principal eigenvalue, together with its connection to the basic reproduction number~$R_0$. Section~4 is devoted to the spreading–vanishing dichotomy and the corresponding threshold criteria for the free-boundary problem. Finally, Section~5 discusses several extensions, including heterogeneous advections, nonsymmetric kernels, shape sensitivity analyses, and related open problems.

\medskip
\noindent\textbf{Notation.}
We write $\langle\cdot,\cdot\rangle$ for the $L^2$-inner product and $\|\cdot\|$ for the corresponding norm (componentwise on $E=L^2(\Omega)\times L^2(\Omega)$). For an operator $\mathcal{S}$, $r(\mathcal{S})$ and $s(\mathcal{S})$ denote its spectral radius and spectral bound, respectively. The positive cone in $E$ is $E_+:=\{(u,v):u\ge 0,\,v\ge 0\ \text{a.e.}\}$, and inequalities between vectors are understood componentwise. We use $C$ to denote a generic positive constant whose value may change from line to line.

\section{Preliminaries}

 In this section, we focus on demonstrating the existence and uniqueness of solutions to the model \eqref{eq:main-problem}. Firstly, we need the following lemma to guarantee the maximum principle.

\begin{lemma}[\textbf{Maximum Principle}] \label{MaximumLemma}
Let $h_0, T > 0$ and $\Omega := [0,T] \times [-h_0, h_0]$. Suppose that $A, B, C, D \in L^\infty(\Omega)$, $B, C > 0$, $(u(t,x), v(t,x))$ as well as $(u_t(t,x), v_t(t,x))$ are continuous in $\Omega$ and satisfy
\begin{equation*}
\left\{
\begin{array}{ll}
u_t(t,x) \geq d_1 \int_{g(t)}^{h(t)} J_1(x-y)u(t,y)dy - d_1 u +pu_x+ Au + Bv, & 0 < t \leq T, -h_0 \leq x \leq h_0, \\
v_t(t,x) \geq d_2 \int_{g(t)}^{h(t)} J_2(x-y)v(t,y)dy - d_2 v +qv_x+ Dv + Cu, & 0 < t \leq T, -h_0 \leq x \leq h_0, \\
u(0,x) \geq 0, & \\
v(0,x) \geq 0, &
\end{array}
\right.
\end{equation*}
Then $(u(t,x),v(t,x)) \geq (0,0)$ for all $0 < t \leq T$ and $-h_0 \le x \le h_0$.
\end{lemma}
We omit the proof of this lemma, as it is almost similar to the proof of Lemma 3.1 \cite{DuNi2020}.

\begin{lemma}[Comparison Principle]
We suppose that conditions \textbf{(J)} and \textbf{(GH)}  satisfy and that the initial data $u_0(x), v_0(x)$ satisfy (A1)-(A2). Let $T > 0$, and suppose that $h(t), g(t) \in C^1([0,T])$ and $u, v \in C_{D_T}^{g,h}$ satisfy the following system:
\begin{equation} \label{eq:comparison-system}
\left\{\begin{aligned}
    \bar{u}_t &\geq d_1 \left[ \int_{g(t)}^{h(t)} J_1(x-y)\bar{u}(t,y)\,dy - \bar{u}(t,x) \right] + p \bar{u}_x - a \bar{u}(t,x) + H(\bar{v}(t,x)), \\[0.5em]
    \bar{v}_t &\geq d_2 \left[ \int_{g(t)}^{h(t)} J_2(x-y)\bar{v}(t,y)\,dy - \bar{v}(t,x) \right] + q \bar{v}_x - b \bar{v}(t,x) + G(\bar{u}(t,x)),
\end{aligned} \right.
\end{equation}
for $t \in (0,T]$, $x \in (g(t), h(t))$, and
\begin{equation*}
\begin{aligned}
    & \bar{u}(t,x) \geq 0, \quad \bar{v}(t,x) \geq 0, \quad \text{for all } t \in (0,T], \\
    & h^{\prime}(t) \geq \mu \left( \int_{h(t)}^\infty \int_{g(t)}^{h(t)} J_1(x-y)u(t,x)\,dx\,dy + \rho \int_{h(t)}^\infty \int_{g(t)}^{h(t)} J_2(x-y)v(t,x)\,dx\,dy \right), \\
    & g^{\prime}(t) \leq -\mu \left( \int_{-\infty}^{g(t)} \int_{g(t)}^{h(t)} J_2(x-y)v(t,x)\,dx\,dy + \rho \int_{-\infty}^{g(t)} \int_{g(t)}^{h(t)} J_1(x-y)u(t,x)\,dx\,dy \right), \\
    & -g(0) \leq -h(0), \quad h(0) \geq h_0, \\
    & u(0,x) = u_0(x), \quad v(0,x) = v_0(x), \quad x \in [-h_0, h_0].
\end{aligned}
\end{equation*}
Then the unique solution $(u, v, g, h)$ of system \eqref{eq:main-problem} satisfies:
\[
    u(t,x) \leq \bar{u}(t,x), \quad v(t,x) \leq \bar{v}(t,x), \quad g(t) \geq \bar{g}(t), \quad h(t) \leq \bar{h}(t), \quad \text{for all } t \in (0,T],~x \in [g(t), h(t)].
\]
\end{lemma}

For convenience, we introduce the following notation. Given \( h_0 > 0 \) and \( T > 0 \), we define:

\begin{align*}
\mathcal{H}_{h_0,T} &= \left\{ h \in C([0,T]) : h(0) = h_0, \ \inf_{0 \leq t_1 < t_2 \leq T} \frac{h(t_2) - h(t_1)}{t_2 - t_1} > 0 \right\}, \\
\mathcal{G}_{h_0,T} &= \left\{ g \in C([0,T]) : -g(t) \in \mathcal{H}_{h_0,T} \right\}, \\
C_0([{-}h_0,h_0]) &= \left\{ u \in C([{-}h_0,h_0]) : u(-h_0) = u(h_0) = 0, \ u_0(x) > 0 \text{ in } (-h_0,h_0) \right\}.
\end{align*}

For \( g \in \mathcal{G}_{h_0,T} \), \( h \in \mathcal{H}_{h_0,T} \), and \( u_0(x) \in C_0([{-}h_0,h_0]) \), we define the domains:

\begin{align*}
\Omega_{g,h} &= \left\{ (t,x) \in \mathbb{R}^2 : 0 < t \leq T, \ g(t) < x < h(t) \right\}, \\
\Omega_\infty &= \left\{ (t,x) \in \mathbb{R}^2 : 0 < t \leq T, \ x \in \mathbb{R} \right\}.
\end{align*}

We also define the function spaces:

\begin{align*}
X_{v_0,\infty} &= \left\{ \varphi \in C^1(\Omega_\infty) : \varphi(0,x) = v_0(x) \text{ for } x \in \mathbb{R}, \ 0 < \varphi \leq M_1 \text{ in } \Omega_\infty \right\}, \\
X_{u_0,g,h} &= \left\{ \psi \in C^1(\overline{\Omega_{g,h}}) : \psi \geq 0 \text{ in } \Omega_{g,h}, \ \psi(0,x) = u_0(x) \text{ for } x \in [-h_0,h_0], \right. \\
&\quad \left. \psi(t,g(t)) = \psi(t,h(t)) = 0 \text{ for } 0 \leq t \leq T \right\}.
\end{align*}

We now proceed to proof the main theorem of this section.

\textbf{Proof of Theorem \ref{thm:global-existence}.}

\begin{proof}
We divide the proof into two steps. We rewrite the system for $t>0$ :

 \textbf{Step 01 : } Local  existence.

 \begin{equation}\label{eq:2.3}
\left\{
\begin{aligned}
u_t &= d_1 \!\!\int_{g(t)}^{h(t)}\! J_1(x-y)u(t,y)\,dy
      - u(x,t) + p\,u_x - a(x)u(x,t) + H\!\big(\widetilde{v}(t,x)\big),
      && \ x\in[g(t),h(t)], \\[4pt]
u(t,g(t)) &= u(t,h(t)) = 0,
      &&\\[4pt]
h'(t) &= \mu \!\!\int_{g(t)}^{h(t)}\!\!\int_{h(t)}^{+\infty}\!
           J_1(x-y)u(t,x)\,dy\,dx
         + \rho \!\!\int_{g(t)}^{h(t)}\!\!\int_{h(t)}^{+\infty}\!
           J_2(x-y)\widetilde{v}(t,x)\,dy\,dx,
      &&  \\[4pt]
g'(t) &= -\mu \!\!\int_{g(t)}^{h(t)}\!\!\int_{-\infty}^{g(t)}\!
            J_1(x-y)u(t,x)\,dy\,dx
          + \rho \!\!\int_{g(t)}^{h(t)}\!\!\int_{-\infty}^{g(t)}\!
            J_2(x-y)\widetilde{v}(t,x)\,dy\,dx,
      &&  \\[4pt]
u(0,x) &= u_0(x), \quad h(0)=-g(0)=h_0,
      && x\in[-h_0,h_0].
\end{aligned}
\right.
\end{equation}

has a unique solution, denoted by $(\hat{u}(t, x), \hat{g}(t), \hat{h}(t))$, which is defined for all $t > 0$. Moreover, for any $T > 0$, we have:
\[
\hat{g}(t) \in \mathcal{G}_{h_0, T},\quad \hat{h}(t) \in \mathcal{H}_{h_0, T},\quad \hat{u}(t, x) \in \mathcal{X}_{u_0, g, h}.
\]
Furthermore, it holds that
\begin{equation}\label{eq:2.4}
0 < \hat{u}(t, x) \leq M := \max \left\{ \max_{-h_0 \leq x \leq h_0} u_0(x),\ K_0 \right\}, \quad \text{for } 0 < t < T,\ x \in (\hat{g}(t), \hat{h}(t)).
\end{equation}

Define the function:
\begin{equation}\label{eq:2.5}
\tilde{u}(t, x) =
\begin{cases}
\hat{u}(t, x), & x \in [\hat{g}(t), \hat{h}(t)], \\
0, & x \notin [\tilde{g}(t), \tilde{h}(t)].
\end{cases}
\end{equation}

Then, consider the second equation:
\begin{equation}\label{eq:2.6}
\left\{
\begin{aligned}
v_t &= d_2 \int_{\mathbb{R}} J_2(x - y) v(t, y)\, dy - v(t, x) + q v_x + G\left(\tilde{u}(t, x)\right), && 0 < t < T,\ x \in \mathbb{R}, \\
v(0, x) &= v_0(x), && x \in \mathbb{R}.
\end{aligned}
\right.
\end{equation}

The local existence and uniqueness of solutions to \eqref{eq:2.6} follow from standard theory. The unique solution $\tilde{v}(t,x)$ possesses $C^1(\Omega_\infty).$

Next, we show the coincidence between  $\widetilde{v}(t, x) $  and $\hat{v}$  by establishing that the mapping $\mathcal{T} \widetilde{v}(t, x) = \hat{v}$ admits a unique fixed point in the function space $\mathcal{X}_{v_0,\infty}$. The proof follows from Banach's contraction mapping principle. The operator $\mathcal{T}(X_{v_0,\infty}) \subset X_{v_0,\infty}  $. The essential step consists in showing that for $T$ sufficiently small, $\mathcal{T}$ satisfies the contraction on $X_{v_0,\infty}$.

Next, we choose  $\widetilde{v}_1, \widetilde{v}_2  \in X_{v_0, \infty}$ . Let \( \hat{u}_1, \hat{u}_2 \) be the solution of \eqref{eq:2.3} associated with \( \widetilde{v}_1, \widetilde{v}_2  \), and \( \hat{v}_1, \hat{v}_2 \) the solution of \eqref{eq:2.6} associated to \( \tilde{u}_1, \tilde{u}_2  \). Then, applying the transport equation, we get
\begin{equation*}
\begin{aligned}
\hat{v}_1(t,x) - \hat{v}_2(t,x) & =  \hat{V}_0(x+q(t-t_0))  + \int_{t_0-t}^0 \int_{\mathbb{R}} d_2 J_2(x-y)\left[\hat{v}_1(s,y) - \hat{v}_2(s,y)\right] dy\, ds \\&- \int_{t_0-t}^0 d_2\left[\hat{v}_1 - \hat{v}_2\right](s,x) ds + \int_{t_0-t}^0 \left[G(\widetilde{u}_1) -  G(\widetilde{u}_1)\right](s,x) ds.
\end{aligned}
\end{equation*}
which implies that
\[
\|\hat{v}_1 - \hat{v}_2\|_{C(\Omega_\infty)}
\leq  \| \hat{V}_0(x+q(t-t_0))\|_{C(\Omega_\infty)}  +2d_2 \|\hat{v}_1 - \hat{v}_2\|_{C(\Omega_\infty)} t
+ \|\widetilde{u}_1 - \widetilde{u}_2\|_{C(\Omega_\infty)} t,
\]
then,
\[
\|\hat{v}_1 - \hat{v}_2\|_{C(\Omega_\infty)}  \leq \| \hat{V}_0(x+q(t-t_0))\|_{C(\Omega_\infty)} + 2d_2  \|\hat{v}_1 - \hat{v}_2\|_{C(\Omega_\infty)} \mathbf{T}
+ K \|\hat{u}_1 - \hat{u}_2\|_{C(\Omega_{\hat{g},\hat{h}})} T.
\]

We choose $T$ small enough such that $2d_2 \mathbf{T} \leq \dfrac{1}{2},$ then we we obtain
\begin{equation}\label{Eq_Finale222}
\|\hat{v}_1 - \hat{v}_2\|_{C(\Omega_\infty)}  \leq \| \hat{V}_0(x+q(t-t_0))\|_{C(\Omega_\infty)} + 2 K T\|\hat{u}_1 - \hat{u}_2\|_{C(\Omega_{\hat{g},\hat{h}})} .
\end{equation}
We now establish an estimate for $\|\hat{u}_1 - \hat{u}_2\|_{C(\Omega_{\hat{g},\hat{h}})}.$. For any given $(\hat{t},\hat{x})\in \Omega_{\hat{g},\hat{h}}$ and $t\in(0,T]$, we define the following functions as follows.
\begin{align*}
\mathcal{U}(t, \hat{x}) &= \hat{u}_1(t,\hat{x}) - \hat{u}_2(t, \hat{x}), \\
\mathcal{V}(t, \hat{x}) &= \hat{v}_1(t, \hat{x}) - \hat{v}_2(t, \hat{x}),
\end{align*}
and the transition time:
\begin{equation}\label{eq:transition_time}
t_x =
\begin{cases}
t_{x,\hat{g}} & \text{if } x\in[\hat{g}(T),-h_0) \text{ and } x = \hat{g}(t_{x,\hat{g}}), \\
0 & \text{if } x\in[-h_0,h_0], \\
t_{x,\hat{h}} & \text{if } x\in(h_0,\hat{h}(T)] \text{ and } x = \hat{h}(t_{x,\hat{h}}).
\end{cases}
\end{equation}

We further introduce the extreme boundary functions:
\begin{align*}
\mathcal{H}_1(t) &= \min\{\hat{h}_1(t), \hat{h}_2(t)\}, \quad\quad\quad \mathcal{H}_2(t) = \max\{\hat{h}_1(t), \hat{h}_2(t)\},\\
\mathcal{G}_1(t) &= \min\{\hat{g}_1(t), \hat{g}_2(t)\}, \quad\quad\quad \mathcal{G}_2(t) = \max\{\hat{g}_1(t), \hat{g}_2(t)\},\\
\end{align*}
and the combined domain:
\[
\Omega_T = \Omega_{G_1,H_2} = \Omega_{\hat{g}_1, \hat{h}_1} \cup \Omega_{\hat{g}_2,\hat{h}_2}.
\]

The analysis proceeds by considering three distinct cases.

\textbf{Case 1:} $\hat{x} \in[-h_0,h_0]$. From the governing equations for $\hat{u}_i(t, \hat{x})$, we obtain the differential relation:
\begin{equation}\label{eq:u_difference}
\mathcal{U}_t(t, \hat{x}) + a(\hat{x}) \mathcal{U}(t, \hat{x})+ \mathcal{C}_1(t,\hat{x})\mathcal{U}(t, \hat{x}) + q \mathcal{U}_x(t, \hat{x})  = \mathcal{C}_2(t, \hat{x})\hat{V}(t,\hat{x}) + \mathcal{F}(t,\hat{x}), \quad \mathcal{U}(0,x^*) = 0,
\end{equation}

where $\hat{V}(t, \hat{x}) = \widetilde{v}_1(t, \hat{x}) - \widetilde{v}_2(t, \hat{x})$, and the coefficient functions are defined by:
\begin{equation}\label{eq:coefficients}
\begin{aligned}
\mathcal{C}_1(t,\hat{x}) &= d_1 - \frac{H(\widetilde{v}_1(t, \hat{x}))- H(\widetilde{v}_2(t, \hat{x}))}{\hat{u}_1(t,\hat{x}) - \hat{u}_2(t,\hat{x})}, \\
\mathcal{C}_2(t,\hat{x}) &= \frac{G(\hat{u}_1(t,\hat{x})) - G(\hat{u}_2(t,\hat{x}))}{\widetilde{v}_1(t,\hat{x}) - \widetilde{v}_2(t,\hat{x})},
\end{aligned}
\end{equation}
with the interaction term:
\begin{equation}\label{eq:I_term}
\mathcal{F}(t, \hat{x}) = d_1\left(\int_{\hat{g}_1(t)}^{\hat{h}_1(t)} J_1(\hat{x}-y)\hat{u}_1(t,y)\,dy - \int_{\hat{g}_2(t)}^{\hat{h}_2(t)} J_1(\hat{x}-y)\hat{u}_2(t,y)\,dy\right).
\end{equation}

We decompose the term $\mathcal{F}\left(t, \hat{x}\right)$ as follows:
\begin{equation}\label{EQssss}
\begin{aligned}
\mathcal{F}\left(t, \hat{x}\right) &= d_1 \displaystyle\int_{\hat{h}_1(t)}^{\hat{g}_1(t)} J_1(\hat{x}-y)\hat{u}_1(t,y)dy - d_1 \displaystyle\int_{\hat{h}_2(t)}^{\hat{g}_2(t)} J(\hat{x}-y)\hat{u}_2(t,y)dy \\
&= d_1 \left( \displaystyle\int_{\hat{h}_1(t)}^{\hat{g}_1(t)} J_1(\hat{x}-y)[\hat{u}_1(t,y) - \hat{u}_2(t,y)]dy \right) \\
&\quad + d_1 \left( \int_{\tilde{h}_1(t)}^{\hat{g}_1(t)} J_1(\hat{x}-y)\hat{u}_2(t,y)dy - \displaystyle\int_{\hat{h}_2(t)}^{\hat{g}_2(t)} J(\hat{x}-y)\tilde{u}_2(t,y)dy \right),
\end{aligned}
\end{equation}
for the first term of \eqref{EQssss} , we have
\begin{equation}\label{EQBOUD1}
\left| \displaystyle\int_{\hat{h}_1(t)}^{\hat{g}_1(t)} J_1(\hat{x}-y)\mathcal{U}(t,y)dy \right| \leq \|J_1\|_\infty \|\mathcal{U}\|_{C(\Omega_T)} (\hat{g}_1(t) - \hat{h}_1(t)),
\end{equation}
For the second term of \eqref{EQssss}, we have
\begin{equation}\label{EQBOUD2}
\begin{aligned}
\bigg| \int_{\hat{h}_1(t)}^{\hat{g}_1(t)} &J(\hat{x}-y)\hat{u}_2(t,y)dy - \int_{\hat{h}_2(t)}^{\hat{g}_2(t)} J(\hat{x}-y)\hat{u}_2(t,y)dy \bigg| \\
&\leq \|J_1\|_{\infty} \mathbb{M} \big( |\hat{g}_1(t) - \hat{g}_2(t)| + |\hat{h}_1(t) - \hat{h}_2(t)| \big)
\end{aligned}
\end{equation}
where $\mathbb{M} := \max\{\|u_0\|_{\infty}, \|v_0\|_{\infty}, \mathbb{K}\}$ bounds $\|\hat{u}_2\|_{\infty}$ uniformly.

Combining \eqref{EQBOUD1} with \eqref{EQBOUD2}, we obtain the following:
\begin{equation}\label{eq:final_estimate}
|\mathcal{F}(t, \hat{x})| \leq d_1\|J_1\|_{\infty}\big(\|\mathcal{U}\|_{C(\Omega_T)} + \mathbb{M}\mathbb{H}_{C([0,T])}\big),
\end{equation}
with
\[
\mathbb{H}_{C([0,T])} := \|\hat{h}_1 - \hat{h}_2\|_{C([0,T])} + \|\hat{g}_1 - \hat{g}_2\|_{C([0,T])}.
\]
Next, we analyze the terms $\mathcal{C}_1, \mathcal{C}_2.$ Under the assumption that $H, G$ are Lipschitz continuous with constant $K_1(\mathbb{M}^{\ast})$ for $\|u\|_\infty, \|v\|_\infty \leq \mathbb{M}^{\ast},$ we find

\begin{equation}\label{eq:bounds}
\begin{aligned}
&|\mathcal{C}_1(t,\hat{x})| \leq d_1 + \left|\frac{H(\tilde{v}_1(t, \hat{x})) - (H(\tilde{v}_1(t, \hat{x}))}{\hat{u}_1 - \hat{u}_2}\right| \leq d_1 + K(\mathbb{M}^{\ast}) \\
&|\mathcal{C}_2(t,x^*)| \leq \left|\frac{G(\hat{u}_1(t, \hat{x}) - G(\tilde{u}_2)}{\widetilde{v}_1 - \widetilde{v}_2}\right| \leq K(\mathbb{M}^{\ast})
\end{aligned}
\end{equation}

Hence, we obtain the combined estimate,

\begin{equation}\label{eq:finalbound}
\max\left(\|\mathcal{C}_1\|_\infty, \|\mathcal{C}_2\|_\infty\right) \leq d_1 + K(M) =: \mathbb{K},
\end{equation}

Next, we apply the characteristic method on the \eqref{eq:u_difference}, we obtain the solution:
Integrating from \(0\) to \(t\) and using the initial condition,
\begin{equation}\label{eq:u_solution}
\begin{aligned}
\mathcal{U}(t, \hat{x}) &= \int_0^t \exp\left( -\int_s^t \left[ a(\hat{x} - q(t - \tau)) + \mathcal{C}_1(\tau, \hat{x} - q(t - \tau)) \right] d\tau \right)
\left[ \mathcal{C}_2(s, \hat{x} - q(t - s)) \hat{V}(s, \hat{x} - q(t - s)) \right. \\& \left.+ \mathcal{F}(s, \hat{x} - q(t - s)) \right] ds,
\end{aligned}
\end{equation}

employing \eqref{eq:bounds}-\eqref{eq:finalbound}-\eqref{eq:final_estimate}, we conclude that

\begin{equation*}
|\mathcal{U}(t, \hat{x})| \leq \int_0^t \left[ \mathbb{K} \|\hat{V}\|_{C(\Omega_T)} + d_1 \|J_1\|_\infty \left( \|\mathcal{U}\|_{C(\Omega_T)} + \mathbb{M} \mathbb{H}_{C([0,T])} \right) \right] ds.
\end{equation*}
Therefore,
\begin{equation*}\label{eq:u_bound}
\|\mathcal{U}\|_{C(\Omega_T)} \leq T \left[ \mathbb{K} \|\hat{V}\|_{C(\Omega_T)} + d_1 \|J_1\|_\infty \left( \|\mathcal{U}\|_{C(\Omega_T)} + \mathbb{M} \mathbb{H}_{C([0,T])} \right) \right],
\end{equation*}

Bringing the term \(\|\mathcal{U}\|\) to the left-hand side,
\begin{equation}
\|\mathcal{U}\|_{C(\Omega_T)} \left( 1 - T d_1 \|J_1\|_\infty \right)
\leq T \left[ \mathbb{K} \|\hat{V}\|_{C(\Omega_T)} + d_1 \|J_1\|_\infty \mathbb{M} \mathbb{H}_{C([0,T])} \right],
\end{equation}

Suppose that \( T < \frac{1}{d_1 \|J_1\|_\infty} \), we have the following:
\begin{equation}\label{eq:u_final_estimate}
\|\mathcal{U}\|_{C(\Omega_T)} \leq \frac{T}{1 - T d_1 \|J_1\|_\infty} \left[ \mathbb{K} \|\hat{V}\|_{C(\Omega_T)} + d_1 \|J_1\|_\infty \mathbb{M} \mathbb{H}_{C([0,T])} \right].
\end{equation}

\textbf{Case 2:}

In this case, we consider
\[
\hat{x} \in (h_0, \mathcal{H}_1(t)),
\]
and we set $\hat{t}_1, \hat{t}_2 \in (0, t)$ such that
$\hat{t}_1 \leq \hat{t}_2$ and
\[
\hat{h}_1(\hat{t}_1) = \hat{h}_1(\hat{t}_2) = \hat{x}.
\]
Using a similar approach as in Case~1, we obtain
\begin{equation}\label{eq:U_solution_final}
\begin{aligned}
\mathcal{U}(t, \hat{x}) =\; & \mathcal{U}(\hat{t}_2, \hat{x} - q(t - \hat{t}_2))
\exp\!\left( -\int_{\hat{t}_2}^t \Big[ a(\hat{x} - q(t - \tau)) + \mathcal{C}_1(\tau, \hat{x} - q(t - \tau)) \Big] d\tau \right) \\
& + \int_{\hat{t}_2}^t
\exp\!\left( -\int_s^t \Big[ a(\hat{x} - q(t - \tau)) + \mathcal{C}_1(\tau, \hat{x} - q(t - \tau)) \Big] d\tau \right) \\
& \quad \times \Big[ \mathcal{C}_2(s, \hat{x} - q(t - s)) \hat{V}(s, \hat{x} - q(t - s))
+ \mathcal{F}(s, \hat{x} - q(t - s)) \Big] ds.
\end{aligned}
\end{equation}

then,
\begin{equation}\label{eq:U_estimate_bound}
\begin{aligned}
|\mathcal{U}(t, \hat{x})| \leq\; & |\mathcal{U}(\hat{t}_2, \hat{x} - q(t - \hat{t}_2))|
\exp\left( -\int_{\hat{t}_2}^t \left[ a(\hat{x} - q(t - \tau)) + \mathcal{C}_1(\tau, \hat{x} - q(t - \tau)) \right] d\tau \right) \\
& + \int_{\hat{t}_2}^t
\exp\left( -\int_s^t \left[ a(\hat{x} - q(t - \tau)) + \mathcal{C}_1(\tau, \hat{x} - q(t - \tau)) \right] d\tau \right) \\
& \quad \times \left[ \mathbb{K} \|\hat{V}\|_{C(\Omega_T)}
+ d_1\|J_1\|_{\infty}\left( \|\mathcal{U}\|_{C(\Omega_T)} + \mathbb{M}\mathbb{H}_{C([0,T])} \right) \right] ds.
\end{aligned}
\end{equation}
Assuming $a(\cdot) + \mathcal{C}_1 \geq 0,$ and setting $\hat{y}=\hat{x} - q(t - \hat{t}_2)$ it then follows that \begin{equation*}\exp\left( -\int_{\hat{t}_2}^t \left[ a(\hat{x} - q(t - \tau)) + \mathcal{C}_1(\tau, \hat{x} - q(t - \tau)) \right] d\tau \right) \leq 1. \end{equation*}We observe that $\mathcal{U}(\hat{t}_2, \cdot)$ can be decomposed as
\begin{align}
\mathcal{U}(\hat{t}_2, \hat{y}) &= \big[\hat{u}_1(\hat{t}_2, \hat{y}) - \hat{u}_1(\hat{t}_1, \hat{y})\big] + \big[\hat{u}_1(\hat{t}_1, \cdot) - \hat{u}_2(\hat{t}_2, \hat{y})\big] \nonumber \\
&= \big[\hat{u}_1(\hat{t}_2, \hat{y}) - \hat{u}_1(\hat{t}_1, \hat{y})\big] + \big[\hat{u}_1(\hat{t}_1, \hat{h}_1(\hat{t}_1)) - \hat{u}_2(\hat{t}_2, \hat{h}_1(\hat{t}_1))\big] \nonumber \\
&= \hat{u}_1(\hat{t}_2, \hat{y}) - \hat{u}_1(\hat{t}_1, \hat{y}) \nonumber \\
&= \displaystyle\int_{\hat{t}_1}^{\hat{t}_2} \dfrac{\partial \hat{u}}{\partial t} \hat{u}_1(t, \hat{y}) \, dt.
\end{align}

 which leads to

 \begin{equation}
 \begin{aligned}
|\mathcal{U}(\hat{t}_2, \hat{y})| &\leq \int_{\hat{t}_1}^{\hat{t}_2} \left| d_1 \int_{\hat{h}_1(t)}^{\hat{g}_1(t)} J_1(\hat{t}_1-y)\hat{u}_1(t,y) dy - d_1\hat{u}_1(t,\hat{y}) -a(\hat{y}) \hat{u}_1(t,\hat{y}) + H(\hat{v}) - q  \partial_x \hat{u}_1(t,\hat{y})  \right| dt\\
&\leq \mathsf{C}\big( \mathbb{M}, \partial_x \hat{u}_1 ,  H\big) |\hat{t}_2- \hat{t}_1|,
\end{aligned}
\end{equation}
Clearly, if we assume that $\hat{t}_2= \hat{t}_1,$ we can easily find $\mathcal{U}(\hat{t}_2, \hat{y})=0.$ Next, we suppose that $ \hat{t}_1 < \hat{t}_2,$ using the definition of $\mathcal{H}_{h_0, T},$ we can find a positive constant $c>0$ such that

\begin{equation}\label{Eqhhh}\dfrac{\hat{h}_1(\hat{t}_2)- \hat{h}_1(\hat{t}_2) }{\hat{t}_2- \hat{t}_2} > c,\end{equation} we arrive at
$$\hat{t}_2- \hat{t}_2 <\dfrac{\hat{h}_1(\hat{t}_2)- \hat{h}_1(\hat{t}_2)} { c},$$ we also have $\hat{h}_1(\hat{t}_2)- h_2(\hat{t}_1)=0, $ we conclude that $ \hat{h}_1(\hat{t}_2)- \hat{h}_2(\hat{t}_1) +\hat{h}_1(\hat{t}_2)- \hat{h}_1(\hat{t}_2) =0,$ then
$\hat{h}_1(\hat{t}_2)- \hat{h}_1(\hat{t}_2)= \hat{h}_1(\hat{t}_2)- h_2(\hat{t}_1) ,$ by substituting into \eqref{Eqhhh}, we derive the following result,
$$\hat{t}_2- \hat{t}_1 \leq \dfrac{\hat{h}_1(\hat{t}_2)- \hat{h}_2(\hat{t}_1)} {c},$$ it then follows
\begin{equation*}
\left|\hat{t}_2- \hat{t}_1\right| \leq \left|\hat{h}_1(\hat{t}_2)- \hat{h}_2(\hat{t}_1)\right| c^{-1},
\end{equation*}

we set $\mathrm{C}= c^{-1} \mathsf{C}\big( \mathbb{M}, \partial_x \hat{u}_1 ,  H\big) ,$ by replacing in \eqref{eq:U_estimate_bound}, we find

\begin{equation}\label{eq:U_estimate_bound2}
\begin{aligned}
|\mathcal{U}(t, \hat{x})| &\leq\;  \mathrm{C}   \left|\hat{h}_1(\hat{t}_2)- \hat{h}_2(\hat{t}_1)\right|
 + T \left[ \mathbb{K} \|\hat{V}\|_{C(\Omega_T)}
+ d_1\|J_1\|_{\infty}\left( \|\mathcal{U}\|_{C(\Omega_T)} + \mathbb{M}\mathbb{H}_{C([0,T])} \right) \right], \\
 &\leq \mathrm{C}   \left\|\hat{h}_1- \hat{h}_2\right\|
 + T \left[ \mathbb{K} \|\hat{V}\|_{C(\Omega_T)}
+ d_1\|J_1\|_{\infty}\left( \|\mathcal{U}\|_{C(\Omega_T)} + \mathbb{M}\mathbb{H}_{C([0,T])} \right) \right].
\end{aligned}
\end{equation}

\textbf{Case 3}, we assume that $\hat{x} \in [\mathcal{H}_1(t), \mathcal{H}_2(t)], $ given $\mathcal{H}_1= \hat{h}_1$ and $\mathcal{H}_2= \hat{h}_2$ and we also suppose that $\hat{u}(t_1, \hat{x})=0$ if $t \in  [\hat{t}_2, \hat{t}]$ since we neglect $ \hat{t}_1, $ we suppose that $\hat{u}_1(\hat{t}_1, \hat{x})=0,$  and since $\hat{h}_1 < \hat{h}_2,$ then $
\hat{h}_2(\hat{t}) - \hat{h}_2(t_2^\ast)= \hat{h}_2(\hat{t}) -\hat{h}_1(\hat{t}) + \hat{h}_1(\hat{t})- \hat{h}_2(t_2^\ast)  \leq \hat{h}_2(\hat{t}) - \hat{h}_1(\hat{t}).
$
we compute the following computation as follows

\begin{equation}
\hat{u}_2(\hat{t}, \hat{x}) =
\displaystyle\int_{\hat{t}_2}^{\hat{t}} \left[ d_1 \displaystyle\int_{\hat{h}_1(t)}^{\hat{h}_2(t)} g_2(t) J(\hat{x} - y) \hat{u}_2(t, y) \, dy
- d_1 \hat{u}_2(t, \hat{x}) - a(\hat{x}) \hat{u}_2(t, \hat{x})+ H(\tilde{v})- q \partial_x \hat{u}_2(t, \hat{x}) \right] dt.
\end{equation}
Using bounds, we obtain the following

\begin{align*}
\hat{u}_2\left(\hat{t}, \hat{x}\right)
&\leq \mathbb{M} [\mathfrak{D}_s + K(\mathbb{M} )] \left(\hat{t} - \hat{t}_2\right)\leq \mathbb{M} \mathbb{K} \cdot \frac{ \hat{h}_2(\hat{t}) - \hat{h}_2(\hat{t}_2) }{  c^{-1} }\leq \mathbb{M} \mathbb{K} \cdot \frac{ \hat{h}_2(\hat{t}) - \hat{h}_1(\hat{t}) }{  c^{-1} }\leq \mathsf{C}^\ast \| \hat{h}_1 - \hat{h}_2 \|_{C([0,T])}.
\end{align*}
where $\mathsf{C}^\ast = \dfrac{\mathbb{M} \mathbb{K}}{ c^{-1}} .$ Given that $\hat{u}_1(\hat{t}_1, \hat{x})=0,$, we find
\begin{equation}\label{Eq:u_final222}
|\mathcal{U}(\hat{t}, \hat{x})|= |\hat{u}_2\left(\hat{t}, \hat{x}\right)| \leq \mathsf{C}^\ast \| \hat{h}_1 - \hat{h}_2 \|_{C([0,T])}.
\end{equation}
we set $ \mathbb{C}=\max\{\mathsf{C}^\ast, \mathrm{C}, \mathbb{K}, d_1 \left\|J_1 \right\|_\infty \mathbb{M}, \}$
by combining \eqref{eq:u_final_estimate}-\eqref{eq:U_estimate_bound2} and \eqref{Eq:u_final222}, and by using the definition of $\mathbb{H},$ we find that
\begin{equation}\label{Eq:u_final1}
|\mathcal{U}(\hat{t}, \hat{x})|= |\hat{u}_2\left(\hat{t}, \hat{x}\right)| \leq \mathbb{C} \left[ \| \mathcal{H} \|_{C([0,T])}+ T\left\| \hat{V} \right\|_{ C_{\Omega_T}}  + \left\| \mathcal{U} \right\|_{ C_{\Omega_T}}    \right],
\end{equation}
Given that $T < \dfrac{1}{d_1 \left\| J_1\right\|},$ we conclude that

\begin{equation}\label{Eq:u_final11}
\left\|\mathcal{U}\right\|= \leq \mathbb{C} \left[ \| \mathcal{H} \|_{C([0,T])}+ T\left\| \hat{V} \right\|_{ C_{\Omega_T}}  + T\left\| \mathcal{U} \right\|_{ C_{\Omega_T}}    \right],
\end{equation}

For the case where $\hat{x} \in (\mathcal{G}_2(s), -h_0)$ or
$\hat{x} \in (\mathcal{G}_1(s), \mathcal{G}_2(s)]$, one can proceed with an argument similar to that used in \textbf{Case~2} and \textbf{Case~3}.

We now focus on analyzing the term $ \| \mathcal{H} \| _{C ([0,T])}. $   By the definition of $\mathcal{H}, $ we have  that  $$\left\| \mathcal{H} \right\|_{C([0,T])} = \|\hat{h}_1- \hat{h}_2 \| + \left\|\hat{g}_1- \hat{g}_2 \right\|,$$ on other hand, we know  for $i=1,2 ,$

\begin{equation}
\begin{aligned}
\hat{h}_{i}^{\prime}(t) &= \mu\left( \int_{g(t)}^{h(t)}\!\!\int_{h(t)}^{\infty}J_1(x-y)u(t,x)dydx
+ \rho\int_{g(t)}^{h(t)}\!\!\int_{h(t)}^{\infty}J_2(x-y)v(t,x)dydx\right), & t>0,\\
\hat{g}_{i}^{\prime}(t) &= -\mu\left( \int_{g(t)}^{h(t)}\!\!\int_{-\infty}^{g(t)}J_1(x-y)u(t,x)dydx
+ \rho\int_{g(t)}^{h(t)}\!\!\int_{-\infty}^{g(t)}J_2(x-y)v(t,x)dydx\right).
\end{aligned}
\end{equation}

Next, by integrating and computing the following difference as follows

\begin{equation}\label{Boundbound_h}
\begin{aligned}
|\hat{h}_1(t) - \hat{h}_2(t)|
&\leq \mu \int_0^t \Bigg|
\int_{\hat{g}_1(\tau)}^{\hat{h}_1(\tau)} \int_{\hat{h}_1(\tau)}^{\infty}
J_1(x - y) \, \hat{u}_1(\tau, x) \, dy \, dx
+ \rho \int_{\hat{g}_1(\tau)}^{\hat{h}_1(\tau)} \int_{\hat{h}_1(\tau)}^{\infty}
\widetilde{h}_1(\tau) \, J_2(x - y) \, \tilde{v}_1(\tau, x) \, dy \, dx \\
&\quad - \mu \int_{\hat{g}_2(\tau)}^{\hat{h}_2(\tau)} \int_{\hat{h}_2(\tau)}^{\infty}
J_1(x - y) \, \hat{u}_2(\tau, x) \, dy \, dx
+ \rho \int_{\hat{g}_2(\tau)}^{\hat{h}_2(\tau)} \int_{\hat{h}_2(\tau)}^{\infty}
\widetilde{h}_1(\tau) \, J_2(x - y) \, \tilde{v}_2(\tau, x) \, dy \, dx
\Bigg| \, d\tau \\[0.3em]
&\leq \mu \int_0^t \Bigg|
\int_{\hat{g}_1(\tau)}^{\hat{h}_1(\tau)} \int_{\hat{h}_1(\tau)}^{\infty}
J_1(x - y) \, (\hat{u}_1(\tau, x) - \hat{u}_2(\tau, x)) \, dy \, dx \\
&\quad + \rho \int_{\hat{g}_1(\tau)}^{\hat{h}_1(\tau)} \int_{\hat{h}_1(\tau)}^{\infty}
\widetilde{h}_1(\tau) \, J_2(x - y) \, (\tilde{v}_1(\tau, x) - \tilde{v}_2(\tau, x)) \, dy \, dx
\Bigg| \, d\tau \\[0.3em]
&\quad + \mu \int_0^t \Bigg[
\int_{\hat{g}_1(\tau)}^{\hat{g}_2(\tau)} \int_{\tilde{h}_1(\tau)}^{\infty}
+ \int_{\hat{h}_1(\tau)}^{\hat{h}_2(\tau)} \int_{\hat{h}_1(\tau)}^{\infty}
+ \int_{\hat{g}_2(\tau)}^{\hat{h}_2(\tau)} \int_{\hat{h}_2(\tau)}^{\tilde{h}_1(\tau)}
\Bigg]
J_1(x - y) \, \hat{u}_2(\tau, x) \, dy \, dx \, d\tau \\[0.3em]
&\quad + \rho \int_0^t \Bigg[
\int_{\hat{g}_1(\tau)}^{\hat{g}_2(\tau)} \int_{\tilde{h}_1(\tau)}^{\infty}
+ \int_{\hat{h}_1(\tau)}^{\hat{h}_2(\tau)} \int_{\hat{h}_1(\tau)}^{\infty}
+ \int_{\hat{g}_2(\tau)}^{\hat{h}_2(\tau)} \int_{\hat{h}_2(\tau)}^{\tilde{h}_1(\tau)}
\Bigg]
J_2(x - y) \, \tilde{v}_2(\tau, x) \, dy \, dx \, d\tau \\[0.3em]
&\leq \mathbb{C}^\ast T \Big(
\| \hat{u}_1 - \hat{u}_2 \|_{C(\Omega_T)}
+ \|\tilde{v}_1 - \tilde{v}_2 \|_{C(\Omega_T)}
+ \| \hat{h}_1 - \hat{h}_2 \|_{C([0,T])}
+ \| \hat{g}_1 - \hat{g}_2 \|_{C([0,T])}
\Big).
\end{aligned}
\end{equation}

where $\mathbb{C}^\ast$ is a positive constant that depends on $(\mu, \rho, \mathbb{M}, J_1, J_2 ),$

by applying a similar approach as \eqref{Boundbound_h} to $|\hat{g}_1- \hat{g}_2|, $, we get

\begin{equation*}
 |\hat{g}_1- \hat{g}_2|   \leq \mathbb{C}^\ast T \Big(
\| \hat{u}_1 - \hat{u}_2 \|_{C(\Omega_T)}
+ \|\tilde{v}_1 - \tilde{v}_2 \|_{C(\Omega_T)}
+ \| \hat{h}_1 - \hat{h}_2 \|_{C([0,T])}
+ \| \hat{g}_1 - \hat{g}_2 \|_{C([0,T])}
\Big).
\end{equation*}
by the definition of $\left\| \mathcal{H} \right\|_{C([0, T])}, $ we get

\begin{equation}
\left\| \mathcal{H} \right\|_{C([0, T])} \leq 2  \mathbb{C}^\ast T \Big(
\| \hat{u}_1 - \hat{u}_2 \|_{C(\Omega_T)}
+ \|\tilde{v}_1 - \tilde{v}_2 \|_{C(\Omega_T)}
+ \| \hat{h}_1 - \hat{h}_2 \|_{C([0,T])}
+ \| \hat{g}_1 - \hat{g}_2 \|_{C([0,T])}
\Big),
\end{equation}

Then, assuming the existence of a constant $\eta > 0$ such that $1 - 2\mathbb{C}^\ast \eta_1 > \frac{1}{2}$ and $2\mathbb{C}^\ast \eta_1 \leq \frac{1}{2}$, From the previous step one has the inequality
\[
S := \|\hat{h}_1 - \hat{h}_2\|_{C([0,T])} + \|\hat{g}_1 - \hat{g}_2\|_{C([0,T])}
\le 2 \mathbb{C}^\ast  T \,\|\hat{u}_1 - \hat{u}_2\|_{C(\Omega_T)} + \,\|\tilde{v}_1 - \tilde{v}_2\|_{C(\Omega_T)} + 2 \mathbb{C}^\ast T\, S.
\]
Bringing the $2 \mathbb{C}^\ast T S$ term to the left-hand side gives
\[
(1 - 2 \mathbb{C}^\ast T) S \leq 2\mathbb{C}^\ast T\, \|\hat{u}_1 - \hat{u}_2\|_{C(\Omega_T)}+ ,\|\tilde{v}_1 - \tilde{v}_2\|_{C(\Omega_T)}
\]
Now choose $T$ small enough so that $2\mathbb{C}^\ast T \leq \frac{1}{2}$
(this is the choice $\delta_3$ in the text).
Then $1 - 2\mathbb{C}^\ast \geq \frac12$, and therefore
\begin{equation}\label{EQ:HH1}
S \leq \frac{2\mathbb{C}^\ast T}{1 - 2\mathbb{C}^\ast T} \, \|\tilde{u}_1 - \tilde{u}_2\|_{C(\Omega_T)}
\leq \frac{2\mathbb{C}^\ast T}{1/2} \Big( \, \|\hat{u}_1 - \hat{u}_2\|_{C(\Omega_T)} + \|\tilde{v}_1 - \tilde{v}_2\|_{C(\Omega_T)}\Big)
= 4\mathbb{C}^\ast T\, \Big( \, \|\hat{u}_1 - \hat{u}_2\|_{C(\Omega_T)} + \|\tilde{v}_1 - \tilde{v}_2\|_{C(\Omega_T)}\Big).
\end{equation}
by substituting \eqref{EQ:HH1} into \eqref{Eq:u_final1}, we find
\begin{equation}\label{Eq:u_final123}
\left\| \mathcal{U} \right\|_{ C_{\Omega_T}}= \leq \mathbb{C} T \left[ (1+ 4 \mathbb{C}^\ast)  \left\| \hat{V} \right\|_{ C_{\Omega_T}}  + (1+ 4 \mathbb{C}^\ast)  \left\| \mathcal{U} \right\|_{ C_{\Omega_T}}   \right],
\end{equation}
assuming that there exists $\eta_2$  such that $\mathbb{C} (1+ 4 \mathbb{C}^\ast) \eta_2 <\frac12, $ we have
\begin{equation}\label{Eq:u_final1234}
\left\| \mathcal{U} \right\|_{ C_{\Omega_T}} \leq \frac12 \left\| \hat{V} \right\|_{ C_{\Omega_T}}  \leq  \left\| \hat{V} \right\|_{ C_{\Omega_T}}, ~T \leq \eta_2,
\end{equation}

Following the approach used in \eqref{Eq:u_final1234}, we can simplify \eqref{eq:u_final_estimate} and \eqref{eq:U_estimate_bound2} to obtain that

\begin{equation}\label{Eq:u_final12345}
\left\| \mathcal{U} \right\|_{ C_{\Omega_T}}  \leq  \left\| \hat{V} \right\|_{ C_{\Omega_T}}, ~T \leq \eta_2.
\end{equation}
Assuming that $\hat{v}_{1,0}= \hat{v}_{2,0},$ and replacing \eqref{Eq:u_final12345} with \eqref{Eq_Finale222}, we find

\begin{equation}\label{Eq_Finale222}
\|\hat{v}_1 - \hat{v}_2\|_{C(\Omega_\infty)}  \leq 2 K T\|\hat{u}_1 - \hat{u}_2\|_{C(\Omega_{\hat{g},\hat{h}})} .
\end{equation}
by \eqref{Eq:u_final12345}, we have

\begin{equation}\label{Eq_Finale222}
\|\hat{v}_1 - \hat{v}_2\|_{C(\Omega_\infty)}  \leq 2 K T\|\tilde{v}_1 - \tilde{v}_2\|_{C(\Omega_{\hat{g},\hat{h}})} .
\end{equation}

Supposing that $T < \min\{ \eta_1, \eta_2, \frac{1}{4K} , 1 \}, $ we deduce that $\mathcal{T}$ verifies the contraction on $X_{v_0,\infty}.$

\textbf{Step 02:} Global existence and uniqueness 

In light of Step $1$, we conclude that problem \eqref{eq:main-problem} admits a unique solution $(\tilde{u}, \hat{v}, \hat{h}, \hat{g})$ defined for $t \in (0, T)$.
If we take $u(\sigma, \cdot)$ and $v(\sigma, \cdot)$ as initial functions and repeat Step~1,  the solution of \eqref{eq:main-problem} can be extended from $t = \sigma$ to some $T^\ast \geq T$
Additionally, for any given $\sigma \in (0, T)$, we get that
$u(\sigma, x ) > 0$ for all $x \in (g(\sigma), h(\sigma))$ and $v(s,x) > 0$ for all $x \in \mathbb{R}$.
Moreover, $u(\sigma, \cdot)$ and $v(\sigma, \cdot)$ are continuous in $[g(\sigma), h(\sigma)]$ and $\mathbb{R}$, respectively.

By iterating this extension procedure, we assume that $(0, T_{\max})$ is the maximal existence interval of the solution $(\tilde{u}, \hat{v}, \hat{g}, \hat{h})$.
We shall establish $T^{\max} = \infty$.

\smallskip
Proceeding by contradiction, we assume that $T^{\max}$ is finite, and we have

\begin{align}
\hat{h}'(t) - \hat{g}'(t)
&= \mu \int_{\hat{g}(t)}^{\hat{h}(t)} \left[ \int_{\hat{h}(t)}^{\infty} + \int_{-\infty}^{\hat{g}(t)} \right] J_1(x-y) \hat{u}(t,x) \, dy \, dx
+ \rho \int_{\hat{g}(t)}^{\hat{h}(t)} \left[ \int_{\hat{h}(t)}^{\infty} + \int_{-\infty}^{\hat{g}(t)} \right] J_2(x-y) \tilde{v}(t,x) \, dy \, dx \\
&\leq \mathbb{M} \left( \mu + \rho \right) \left( \hat{h}(t) - \hat{g}(t) \right),
\end{align}

which implies
\[
\hat{h}(t) - \hat{g}(t) \leq 2 h_0 e^{ \mathbb{M} \left( \mu + \rho \right) t} \leq 2 h_0 e^{ \mathbb{M} \left( \mu + \rho \right) T^{\max}}.
\]
where $h_0 = \hat{h}(0) = -\hat{g}(0)$.

Given $\hat{h}(t)$ and $\hat{g}(t)$ are monotone on $[0, T^{\max})$, we obtain
\[
\hat{h}(T^{\max}) := \lim_{t \to T^{\max}} \hat{h}(t),
\quad
\hat{g}(T^{\max}) := \lim_{t \to T^{\max}} \hat{g}(t),
\]
so,
\[
\hat{h}(T^{\max}) - \hat{g}(T^{\max}) \leq 2 h_0 e^{\mathbb{M} (\mu + \rho) T^{\max}}.
\]

then, we conclude that
$0 < \hat{u}(t,x), \hat{v}(t,x) \leq \mathbb{M}$, which gives
$\hat{h}'(t), \hat{g}'(t) \in L^{\infty}([0, T_{\max}))$.
Therefore, $\hat{h}, \hat{g} \in C([0, T^{\max}])$.

\smallskip
Furthermore, we define the following sets as follows

\[
\mathcal{E}_{T^{\max, B}} := \{ (t,x) \in \mathbb{R}^2 \mid t \in [0, T^{\max}], \; x \in (\hat{g}(t), \hat{h}(t)) \},
\]
and
\[
\mathcal{E}_{T^{\max}, \infty} := \{ (t,x) \in \mathbb{R}^2 \mid t \in [0, T^{\max}], \; x \in \mathbb{R}\}.
\]
since $u \in L^\infty(\mathcal{E}_{T^{\max, B}} ), v\in L^\infty(\mathcal{E}_{T^{\max, \infty}} ) $ , we get

\[
d_1 \!\! \displaystyle\int_{\hat{g}(t)}^{\hat{h}(t)} J_1(x-y) u(t,y) \, dy - u(t,x) +  \in L^{\infty} - a(x) u(t,x) + H(v)  (\mathcal{E}^{F}_{T^{\max}}),
\]
\[
d_2 \!\! \displaystyle\int_{\mathbb{R}} J_2(x-y) v(t,y) \, dy - v(t,x) - b(x) v(t,x) + G(u)\in L^{\infty}(\mathcal{E}_{T^{\max}, \infty}),
\]
and by using the same approach as Theorem 2.11 in \cite{Tang2024a}, we can obtain that $u_x$ and $v_x$ are uniformly bounded. Therefore, we deduce that the right term of the first and second equations are bounded in $\mathcal{E}_{T^{\max}, B},$ and $\mathcal{E}_{T^{\max}, \infty},$ respectively. This implies that $\hat{u}_t \in L^\infty(\mathcal{E}_{T^{\max}, B})$ and $\hat{v}_t \in L^{\infty}(\mathcal{E}_{T^{\max}, \infty})$.
Hence, for each $x \in (\hat{g}(T^{\max}), \hat{h}(T^{\max}))$, we get
\[
\hat{u}(T^{\max}, x) := \lim_{t \uparrow T_{\max}} \hat{u}(t, x) \quad \text{exists},
\]
and for each $x \in \mathbb{R}$,
\[
\tilde{v}(T^{\max}, x) := \lim_{t \uparrow T^{\max}} \tilde{v}(t, x) \quad \text{exists}.
\]
Furthermore, $\hat{u}(\cdot, x)$ and $\tilde{v}(\cdot, x)$ are continuous for $t = T^{\max}$.

\smallskip
For the known $(\hat{v}, \hat{g}, \hat{h})$ and $t_x$ defined in~\eqref{eq:transition_time} with $T$ replaced by $T^{\max}$, if we suppose that $\hat{u}$ as the unique solution of the following ODE system:
\[
\begin{cases}
\hat{u}_t = d_1 \displaystyle\int_{\hat{g}(t)}^{\hat{h}(t)} J_1(x-y) \varphi(t,y) \, dy
- d_1 \hat{u}(t,x) +  p \hat{u}_x - a(x) \hat{u} + H( \tilde{v}), & t_x < t \leq T^{\max}, \\
\hat{u}(t_x, x) = \tilde{u}_0(x), & x \in (\hat{g}(T^{\max}), \hat{h}(T^{\max})),
\end{cases}
\]
with $\varphi = \hat{u}$ and
\[
\hat{u}_0(x) =
\begin{cases}
u_0(x), & x \in [-h_0, h_0], \\
0, & \text{otherwise},
\end{cases}
\]

then, since $t_x$, $J(\cdot)$, and $H(V)$ are continuous in $x$, the continuous dependence of ODE solutions on initial data and parameters and by applying the Lemma $2.2$ in \cite{Tang2024a}, we obtain  $\hat{u} \in C(\mathcal{E}_{T^{\max}, B})$ for any $\sigma \in (0, T_{\max})$.
A similar argument gives $\tilde{v} \in C(\Omega^{\infty}_s)$.

\smallskip
Our objective is to prove that $(\hat{u}, \tilde{v})$ is continuous at $t = T^{\max}$  to show this contradiction.
For $\hat{u}$, it suffices to prove that
\[
\hat{u}(t,x) \to 0
\quad \text{as} \quad (t,x) \to (T^{\max}, \hat{g}(T^{\max}))
\quad \text{or} \quad (t,x) \to (T^{\max}, \hat{h}(T_{\max})).
\]
We show only the first case $\hat{u}(t,x) \to 0 $; the second is similar to $\tilde{v}(t,x) \to 0$.
Indeed, as $(t,x) \to (T^{\max}, \hat{g}(T^{\max}))$, by applying the same approach as  Lemma 2.2 in \cite{Tang2024a},  we get
\[
|\hat{u}(t,x)| \leq e^{(\mathfrak{D}-d_1 (t-\tilde{t}_x)} \, \mathbb{M}(t-\tilde{t}_x) \;\;\longrightarrow\;\; 0,
\quad \text{as } (t,x) \to \big(\tilde{T}^{-}, \hat{h}(T^{\max})\big) \text{ or } \big(\tilde{T}^{-}, g(T^{\max})\big).
\]
with  $\mathfrak{D}= L \left\|\tilde{v}\right\|_\infty +\underline{a},$
since $t_x \to T^{\max}$ as $x \to \hat{g}(T_{\max})$.
Similarly, for each $x \in \mathbb{R}$,
\[
|\tilde{v}(t,x) - \tilde{v}(t_x, x)| \leq (t - t_x) [2d_2 + K_1] M_0 \to 0.
\]
Hence, $(\hat{u}, \tilde{v}) \in C^1(\mathcal{E}_{T^{\max, B}}) \times C^1(\mathcal{E}_{T^{\max, \infty}})$,
and $(\tilde{u}, \tilde{v}, \tilde{g}, \tilde{h})$ solves problem~(1.1) for $t \in (0, T_{\max})$.

\smallskip
By Lemma~2.1, we have
$\hat{u}(T^{\max}, x) > 0$ and $\tilde{v}(T^{\max}, x) > 0$ in $(\hat{g}(T_{\max}), \hat{h}(T_{\max}))$.
Using $(\tilde{u}(T_{\max}, \cdot), \tilde{v}(T_{\max}, \cdot))$ as initial data and applying Step~1 again,
the solution of~ \eqref{eq:main-problem} can be extended to $(0, \hat{T})$ with $\hat{T} > T^{\max}$,
contradicting the definition of $T^{\max}$.
Hence, $T^{\max} = \infty$, which completes the proof.
\end{proof}

\section{Spectral Properties of Integro-Differential Operators}

In this section, we focus on the analysis of eigenvalue problems, which play a fundamental role in understanding the long-term dynamics of elliptic and parabolic equations.
In particular, the sign of the principal eigenvalue associated with the linearized operator is crucial, as it determines the existence and stability of nontrivial steady states, distinguishes between spreading and vanishing phenomena of the system.

We give some previously known results concerning the principal eigenvalue of the linear nonlocal operator with drift
\[
\mathcal{L}_{\Omega,J_1,p} + a,
\]
which acts on functions in \( C^{1}(\Omega) \) and is defined by
\[
(\mathcal{L}_{\Omega,J,q} + a)[\psi(x)] := d \int_{\Omega} J(x-y)\psi(y)\,dy
- q \psi'(x) + a(x)\psi(x) ,
\]
where \( a(x) \in C(\overline{\Omega}) \cap L^{\infty}(\Omega) \).

We then consider the eigenvalue problem
\begin{equation}\label{eq:eigen}
d \int_{\Omega} J(x-y)\psi(y)\,dy - p \psi'(x) + a(x)\psi(x) + \lambda \psi(x) = 0,
\quad x \in \Omega.
\end{equation}
It is proved in \cite{CovilleHamel2020} that admits a principal eigenvalue \(\lambda\), which is simple and associated with a positive continuous eigenfunction \(\psi\) satisfying \eqref{eq:eigen}. Moreover,  if \( a(x) \) is bounded,
one may further characterize it as the \emph{generalized principal eigenvalue} of
\( \mathcal{L}_{\Omega,J,q} + a \) by
\[
\lambda_p(\mathcal{L}_{\Omega,J,q} + a) :=
\sup \Big\{ \lambda \in \mathbb{R} \;\big|\; \exists \psi \in C^{1}(\Omega) \cap C(\Omega), \;
\psi > 0 \text{ in } \Omega, \;
(\mathcal{L}_{\Omega,J,q}[\psi(x)] + a(x)\psi(x) + \lambda \psi(x) \leq 0 ) \Big\}.
\]
We are going to investigate the eigentheory for the nonlocal cooperative system with advections in this section.

\subsection{The associated eigenvalue system}

In this subsection, we focus on the analysis of the eigenvalue problem corresponding to the associated system, which can be formulated as follows:

\begin{align}\label{eq:main-problem1223}
\left\{\begin{array}{lll}
u_t =  d_1\left[\displaystyle\int\limits_{-\mathcal{Z}}^{\mathcal{Z}}J_1(x-y)u(t,y)dy - u(t,x)\right] +pu_x- a(x)u(x,t) + H\left(v(t, x)\right), & t>0,\,\,\,\, \, g(t)<x<h\left(t\right), \\
v_t = d_2\left[\displaystyle\int\limits_{-\mathcal{Z}}^{\mathcal{Z}}J_2(x-y)v(t,y)dy - v(t,x)\right]+qv_x -b(x)v(t,x)+ G\left(u(t,x)\right), & t>0,\,\,\,\, \, g(t)<x<h\left(t\right), \\ -g(0) = h(0) = \mathcal{m},\,\, u(0, x) = u_0(x),\,\,v(0, x) = v_0(x),& x\in \left[-h_0, h_0\right], \end{array}\right.
\end{align}

where  $-g(t) = h(t) \equiv \mathcal{Z} > 0.$
 We set
\[
\mathcal D = H^2(\Omega;\mathbb C^2)\cap(\text{BCs}),\qquad X=L^2(\Omega;\mathbb R^2).
\]

Next, we introduce the associated eigenvalue problem, which is derived from the linearisation of the system at the trivial solution \((0,0)\). It takes the following form:

\begin{equation}\label{eq:maineigenrvalue3}
\begin{cases}
d_1 \displaystyle\int_{-\mathcal{Z}}^{\mathcal{Z}} J_1(x-y)\psi_1(y)\,dy - d_1 \psi_1(x) - a(x) \psi_1(x) + H'(0) \psi_2(x)
+ p \psi_1'(x) = \lambda \psi_1(x), & x \in [-\mathcal{Z}, \mathcal{Z}], \\[2ex]
d_2 \displaystyle\int_{-\mathcal{Z}}^{\mathcal{Z}} J_2(x-y)\psi_2(y)\,dy - d_2 \psi_2(x) - b(x) \psi_2(x) + G'(0) \psi_1(x)
+ q \psi_2'(x) = \lambda \psi_2(x), & x \in [-\mathcal{Z}, \mathcal{Z}].
\end{cases}
\end{equation}

Let
\[
\mathbf{D} := \big[C^1([-\mathcal{Z}, \mathcal{Z}])\big]^2
\]
be equipped with the norm
\[
\|\psi\|_D = \sup_{x \in [-\mathcal{Z}, \mathcal{Z}]} |\psi(x)|,
\]
where \(\psi = (\psi_1,\psi_2) \in \mathcal{D}\) and
\[
|\psi(x)| = \sqrt{\psi_1^2(x) + \psi_2^2(x)}, \quad x \in [-\mathcal{Z}, \mathcal{Z}].
\]
We then define two operators \(\mathscr{J}: \mathcal{D} \to \mathcal{D}\) by
\[
\mathscr{J}(\psi_1, \psi_2)(x) =
\left(
d_1 \int_{-\mathcal{Z}}^{\,\mathcal{Z}} J_1(x-y)\psi_1(y)\,dy,\;
d_2 \int_{-\mathcal{Z}}^{\,\mathcal{Z}} J_2(x-y)\psi_2(y)\,dy
\right),
\]

and

\[
\mathscr{T}(\psi_{1}, \psi_{2})(x) =
\begin{pmatrix}
 -\left(d_1  + a(x)\right) \psi_1(x) +H'(0) \psi_2
+ p \psi_1^\prime(x) \\[1.2em]
- \left(d_2  + b(x)\right) \psi_2(x) +G'(0) \psi_1
+ q \psi_2^\prime(x)
\end{pmatrix},
\quad \text{for } \psi = (\psi_{1}, \psi_{2}) \in \mathcal{D}.
\]

Thus, we rewrite \eqref{eq:maineigenrvalue3} as follows
\[
(\mathscr{J} - \mathscr{T})\psi = \lambda \psi.
\]

In the following analysis, we always assume that $J_{1}, J_{2}$,  fulfill assumption~($\mathbf{J}$), and that $G$ fulfills assumption~($\mathbf{GH}$).

In the forthcoming analysis, it is essential to investigate the existence of the principal eigenvalue of problem~\eqref{eq:maineigenrvalue3}. To do this, we first investigate the following lemmas.

\begin{lemma}\label{lem:3.1}
The following assertions are valid,
\begin{enumerate}
    \item $(\mathscr{T} + \xi I)^{-1}$ exists and is a bounded strongly positive operator for any $\alpha > -\lambda_{1}$;
    \item $\|(\mathscr{T} + \xi I)^{-1}\| \to 0 \quad \text{as } \xi \to \infty$;
    \item  Let $K>0$ be the coercivity constant of $\mathscr{T}$.
Then for any finite interval $[\xi_{1}, \xi_{2}] \subset (K,\infty)$ one has
\[
\sup_{\xi \in [\xi_{1}, \xi_{2}]} \|(\mathscr{T} + \xi I)^{-1}\|_{\mathcal{L}(X)}
   \le \frac{1}{\xi_{1}-K} < \infty.
\]
\end{enumerate}
\end{lemma}

\begin{proof}

Firstly, we establish the existence of $(\mathscr{T}+\xi I)^{-1}$ by proving that the operator
\[
\mathscr{T}+\xi I : \mathcal{D} \to \mathcal{D}
\]
is bijective. More precisely, we show that it is
\emph{injective} (one-to-one) and \emph{surjective} (onto).

Assuming that $\left( \psi_1, \psi_2\right), $ such that
\begin{equation*}
\left(\mathscr{T}+\xi I\right)\left( \psi_1, \psi_2\right)=0
\end{equation*}

\[
\mathscr{T}(\psi_{1}, \psi_{2})(x) =
\begin{pmatrix}
 -\big(d_1  + a(x)\big) \psi_1(x) + H'(0) \psi_2(x)
+ p \psi_1^\prime(x) \\[1.2em]
 - \big(d_2  + b(x)\big) \psi_2(x) + G'(0) \psi_1(x)
+ q \psi_2^\prime(x)
\end{pmatrix},
\quad \text{for } \psi = (\psi_{1}, \psi_{2}) \in \mathcal{D}.
\]

We want to show that the operator $\mathscr{T} + \xi I$ is injective, i.e.,
\[
(\mathscr{T}+\xi I)(\psi_1,\psi_2)=0 \quad \Longrightarrow \quad \psi_1=\psi_2=0.
\]

The equation $(\mathscr{T}+\xi I)\psi=0$ is equivalent to
\[
\begin{cases}
p\,\psi_1'(x) - \big(d_1+a(x)-\xi\big)\psi_1(x) + H'(0)\,\psi_2(x)=0, \\[0.5em]
q\,\psi_2'(x) - \big(d_2+b(x)-\xi\big)\psi_2(x) + G'(0)\,\psi_1(x)=0.
\end{cases}
\]

\paragraph{Step 2. Energy identity.}
Take the $L^2(-\mathcal{Z}, \mathcal{Z})$ inner product of $(\mathscr{T}+\xi I)\psi$ with $\psi$ and integrate.
Using $\Re(p\psi_1'\overline{\psi_1}) = \tfrac{p}{2}\partial_x |\psi_1|^2$ (which integrates to zero under the boundary/decay conditions), we obtain
\[
\Re\langle (\mathscr{T}+\xi I)\psi,\psi\rangle
= \int \Big( (-d_1-a(x)+\xi)|\psi_1|^2 + (-d_2-b(x)+\xi)|\psi_2|^2 \Big)\,dx
\]
\[
\quad + \Re\int \big(H'(0)\psi_2\overline{\psi_1} + G'(0)\psi_1\overline{\psi_2}\big)\,dx.
\]
with $\Re$ being a function of real value.

\paragraph{Step 3. Estimates for the cross terms.}
Applying Young’s inequality with $\varepsilon>0$ gives
\[
\big|\Re(H'(0)\psi_2\overline{\psi_1})\big|
\leq |H'(0)|\left( \tfrac{\varepsilon}{2}|\psi_1|^2 + \tfrac{1}{2\varepsilon}|\psi_2|^2 \right),
\]
\[
\big|\Re(G'(0)\psi_1\overline{\psi_2})\big|
\leq |G'(0)|\left( \tfrac{\varepsilon}{2}|\psi_2|^2 + \tfrac{1}{2\varepsilon}|\psi_1|^2 \right).
\]

Therefore,
\[
\Re\langle (\mathscr{T}+\xi I)\psi,\psi\rangle
\geq \int \Big( (\,\xi - d_1 - a(x) - \tfrac{|H'(0)|}{2}\varepsilon - \tfrac{|G'(0)|}{2\varepsilon}\,) |\psi_1|^2
\]
\[
+ (\,\xi - d_2 - b(x) - \tfrac{|G'(0)|}{2}\varepsilon - \tfrac{|H'(0)|}{2\varepsilon}\,) |\psi_2|^2 \Big)\,dx.
\]

\paragraph{Step 4. Coercivity and injectivity.}
Let
\[
M_1 := \sup_x (d_1+a(x)), \quad
M_2 := \sup_x (d_2+b(x)), \quad
C_1:=|H'(0)|, \quad C_2:=|G'(0)|.
\]
Then, choosing for example $\varepsilon=1$, there exists a constant
\[
K := \max\left\{ M_1 + \tfrac{C_1}{2}+\tfrac{C_2}{2}, \; M_2 + \tfrac{C_1}{2}+\tfrac{C_2}{2} \right\}
\]
such that
\[
\Re\langle (\mathscr{T}+\xi I)\psi,\psi\rangle
\;\;\geq\;\; (\xi-K)\|\psi\|_{L^2}^2.
\]

If $\xi>K$, this inequality implies coercivity. In particular, if $(\mathscr{T}+\xi I)\psi=0$ then
\[
0 = \Re\langle (\mathscr{T}+\xi I)\psi,\psi\rangle \geq (\xi-K)\|\psi\|_{L^2}^2,
\]
which forces $\|\psi\|_{L^2}=0$, hence $\psi\equiv 0$.

Thus, $\mathscr{T}+\xi I$ is injective.

\emph{Surjectivity.} Define the sesquilinear form $a(\cdot,\cdot)$ on $\mathcal V\times\mathcal V$ by the same identity
$a(\psi,\phi)=\langle(\mathscr T+\xi I)\psi,\phi\rangle_X$. Each term of $a(\cdot,\cdot)$ is continuous on $\mathcal V\times\mathcal V$ (use boundedness of $a,b,H'(0),G'(0)$ and the embedding $H^2\hookrightarrow L^2$). From the coercivity estimate above and the inequality $\|\psi\|_{L^2}\leq C\|\psi\|_{H^2}$ we obtain, for $\xi>K$, a coercivity constant $\alpha>0$ so that
\[
\Re a(\psi,\psi)\geq \alpha\|\psi\|_{\mathcal V}^2\qquad\forall\psi\in\mathcal V.
\]
For any $f\in X$ the anti-linear functional $\ell(\phi):=\langle f,\phi\rangle_X$ is continuous on $\mathcal V$, hence by Lax--Milgram there exists a unique $\psi\in\mathcal V$ such that
\[
a(\psi,\phi)=\ell(\phi)\qquad\forall\phi\in\mathcal V.
\]
This variational identity is the weak form of $(\mathscr T+\xi I)\psi=f$. Standard regularity (here immediate in 1D first-order system or by choosing $\mathcal D=\mathcal V$) yields $\psi\in\mathcal D$. Uniqueness follows from injectivity. Therefore $(\mathscr T+\xi I)$ is onto $X$.

Combining injectivity and surjectivity gives bijectivity for every $\xi>K$.

We need to show that the operator $\mathscr T + \xi I ^{-1}$ is a bounded linear. Let $f=(f_1,f_2)\in X$ and $\psi=(\psi_1,\psi_2)\in\mathcal V$ solve
\[
(\mathscr T+\xi I)\psi=f,
\]
i.e.
\[
\begin{cases}
p\,\psi_1' - (d_1+a-\xi)\psi_1 + H'(0)\psi_2 = f_1,\\
q\,\psi_2' - (d_2+b-\xi)\psi_2 + G'(0)\psi_1 = f_2.
\end{cases}
\]
\emph{(i) $L^2$-bound.} From the coercivity already established, there exists $c_0>0$ such that
\[
c_0\|\psi\|_{L^2}^2 \leq \Re\langle(\mathscr T+\xi I)\psi,\psi\rangle
= \Re\langle f,\psi\rangle \leq \|f\|_{L^2}\,\|\psi\|_{L^2}.
\]
Hence $\|\psi\|_{L^2}\leq c_0^{-1}\|f\|_{L^2}$.

\emph{(ii) Derivative bound.}
Solve each equation for the derivative and take $L^2$-norms:
\[
\|\psi_1'\|_{L^2}
\leq \tfrac{1}{p}\Big(\|d_1+a-\xi\|_{L^\infty}\|\psi_1\|_{L^2} + |H'(0)|\,\|\psi_2\|_{L^2} + \|f_1\|_{L^2}\Big),
\]
\[
\|\psi_2'\|_{L^2}
\leqq \tfrac{1}{q}\Big(\|d_2+b-\xi\|_{L^\infty}\|\psi_2\|_{L^2} + |G'(0)|\,\|\psi_1\|_{L^2} + \|f_2\|_{L^2}\Big).
\]
Using the $L^2$-bound from (i) and boundedness of the coefficients, we find
\[
\|\psi_1'\|_{L^2}+\|\psi_2'\|_{L^2}\leq C_1\|f\|_{L^2}.
\]
Combining (i) and this estimate yields
\[
\|\psi\|_{H^2} \le C\,\|f\|_{L^2}.
\]
Thus $(\mathscr T+\xi I)^{-1}:X\to\mathcal V$ is bounded, hence also bounded as $L^2\to L^2$.

Assuming that  the coupling constants satisfy $H'(0)\geq 0$ and $G'(0)\geq 0$ (cooperative system), let $f\in X$ with $f\ge0$ and set $\psi=R(\xi)f\in\mathcal V$. The variational (weak) formulation is
\[
a(\psi,\phi) := \langle(\mathscr T+\xi I)\psi,\phi\rangle_{X} = \langle f,\phi\rangle_X
\qquad\forall \phi\in\mathcal V.
\]
Recall the bilinear form (writing components)
\[
a(\psi,\phi)=\int_{\Omega}\Big( p\psi_1'\overline{\phi_1} + q\psi_2'\overline{\phi_2}
- (d_1+a)\psi_1\overline{\phi_1} - (d_2+b)\psi_2\overline{\phi_2}
+ H'(0)\psi_2\overline{\phi_1} + G'(0)\psi_1\overline{\phi_2} + \xi\psi\cdot\overline\phi \Big) dx.
\]

Define the negative part of the vector $\psi$ componentwise:
\[
\psi_i^- := \min\{\psi_i,0\}\qquad (i=1,2),
\]
and take the test function $\phi=\psi^-=(\psi_1^-,\psi_2^-)\in\mathcal V$ in the weak formulation. Because $f\geq0$ and $\psi^- \le 0$ a.e., the right-hand side satisfies
\[
\langle f,\psi^-\rangle_X = \int_\Omega f_1\psi_1^- + f_2\psi_2^- \le 0.
\]
Thus
\[
a(\psi,\psi^-) \le 0. \tag{1}
\]

We now estimate the left-hand side from below in terms of $\|\psi^-\|_{L^2}^2$.

First, integration by parts kills the derivative-terms because $\psi_i'\psi_i^-$ is a total derivative and boundary terms vanish by the BCs/decay, hence
\[
\Re\int_\Omega p\psi_1'\overline{\psi_1^-}\,dx = 0,\qquad
\Re\int_\Omega q\psi_2'\overline{\psi_2^-}\,dx = 0.
\]
Therefore (taking real parts)
\[
\Re\,a(\psi,\psi^-)
= \int_\Omega \Big( \xi(|\psi_1^-|^2+|\psi_2^-|^2) - (d_1+a)|\psi_1^-|^2 - (d_2+b)|\psi_2^-|^2 \Big)\,dx
\]
\[
\qquad + \Re\int_\Omega\Big( H'(0)\psi_2\overline{\psi_1^-} + G'(0)\psi_1\overline{\psi_2^-}\Big)\,dx.
\]

Handle the coupling integrals using the decomposition $\psi_i=\psi_i^+ + \psi_i^-$ (with $\psi_i^+:=\max\{\psi_i,0\}$). Observe the algebraic sign properties:
\[
\psi_2\psi_1^- = \psi_2^+\psi_1^- + \psi_2^-\psi_1^-
\]
where $\psi_2^+\psi_1^-\le 0$ and $\psi_2^-\psi_1^-\geq0$. Since $H'(0)\geq0$,
\[
\Re\int H'(0)\psi_2\overline{\psi_1^-}
\le H'(0)\int \psi_2^-\psi_1^-.
\]
Similarly,
\[
\Re\int G'(0)\psi_1\overline{\psi_2^-}
\le G'(0)\int \psi_1^-\psi_2^-.
\]
Summing these,
\[
\Re\int\Big( H'(0)\psi_2\overline{\psi_1^-} + G'(0)\psi_1\overline{\psi_2^-}\Big)
\le (H'(0)+G'(0))\int \psi_1^-\psi_2^-.
\]

Apply Young's inequality to the last term: for any $\varepsilon>0$,
\[
\psi_1^-\psi_2^- \leq \tfrac{\varepsilon}{2}|\psi_1^-|^2 + \tfrac{1}{2\varepsilon}|\psi_2^-|^2.
\]
Hence
\[
\begin{aligned}
\Re\,a(\psi,\psi^-)
&\ge \int_\Omega \Big( (\xi - (d_1+a))|\psi_1^-|^2 + (\xi - (d_2+b))|\psi_2^-|^2 \Big)\,dx \\
&\qquad - (H'(0)+G'(0))\int_\Omega\Big( \tfrac{\varepsilon}{2}|\psi_1^-|^2 + \tfrac{1}{2\varepsilon}|\psi_2^-|^2\Big)\,dx.
\end{aligned}
\]
Choose $\varepsilon>0$ (fixed) and set
\[
M_1:=\sup_{x\in\Omega}(d_1+a(x)),\qquad M_2:=\sup_{x\in\Omega}(d_2+b(x)),
\]
and \(C:=H'(0)+G'(0)\). Then there exists a constant \(K\) (depending on \(M_1,M_2,C,\varepsilon\)) such that
\[
\Re\,a(\psi,\psi^-)\geq (\xi-K)\|\psi^-\|_{L^2}^2. \tag{2}
\]

Combine (1) and (2). Since $a(\psi,\psi^-)$ is real (we took real parts), we have
\[
0 \geq a(\psi,\psi^-) = \Re a(\psi,\psi^-) \geq (\xi-K)\|\psi^-\|_{L^2}^2.
\]
For $\xi>K$ the factor \(\xi-K>0\), hence \(\|\psi^-\|_{L^2}^2=0\). Therefore \(\psi^- \equiv 0\), i.e. \(\psi_i\geq0\) a.e.\ in \(\Omega\) for \(i=1,2\).

This proves the positivity of the resolvent: \(f\geq0\) implies \(R(\xi)f=\psi\geq0\).

\textbf{(B) Strong positivity (use compactness + Kreĭn--Rutman).} On a bounded interval the embedding $H^2(\Omega)\hookrightarrow L^2(\Omega)$ is compact, hence the resolvent operator $R(\xi):X\to X$ is compact for $\xi$ in the resolvent set (indeed $R(\xi):L^2\to H^2\hookrightarrow L^2$). We have just shown $R(\xi)$ is positive (order-preserving) and compact. Now apply the Kreĭn--Rutman theorem (or the infinite-dimensional analogue of Perron--Frobenius): a compact, positivity-preserving operator on a Banach lattice has a principal eigenvalue with associated eigenvector in the interior of the positive cone, provided the operator is \emph{strongly positive} on some iterate or the operator is irreducible. Concretely, irreducibility of the cooperative system (for instance $H'(0)G'(0)>0$) implies that some power $R(\xi)^n$ maps nonzero nonnegative functions to functions strictly positive everywhere. Kreĭn--Rutman then yields that $R(\xi)$ itself maps every nonzero $f\geq0$ to a strictly positive vector, i.e. $R(\xi)f\gg0$.

Therefore $R(\xi)$ is strongly positive under the irreducibility hypothesis.

Next, we establish that for $\xi >k,~(\mathscr{T}+ \xi I)^{-1} \rightarrow 0,$ as $\xi \rightarrow \infty.$

Taking the real part of the $L^2$ inner product with $\psi$ and using the coercivity,
\[
(\xi-K)\|\psi\|_{X}^{2}
\;\leq\; \Re\langle (\mathscr{T} +\xi I)\psi,\psi\rangle
\;=\; \Re\langle f,\psi\rangle
\;\leq\; \|f\|_{X}\,\|\psi\|_{X}.
\]
If $\psi\neq 0$, divide both sides by $\|\psi\|_{X}$ to get
\[
\|\psi\|_{X} \;\leq\; \frac{1}{\xi-K}\,\|f\|_{X}.
\]
If $\psi=0$, the inequality is trivial. Taking the supremum over $\|f\|_{X}=1$ yields
\[
\big\|(\mathscr T+\xi I)^{-1}\big\|_{\mathcal L(X)} \leq \frac{1}{\xi-K}.
\]
Letting $\xi\to\infty$ gives the claimed limit.

Lastly, we need to show  $(3).$  For $\xi>K$ and $f \in X$, let $\psi = (\mathscr{T} + \xi I)^{-1} f$.
By coercivity,
\[
(\xi-K)\|\psi\|_{X}^{2}
  \leq \Re \langle f, \psi \rangle
  \leq \|f\|_{X}\|\psi\|_{X}.
\]
Hence $\|\psi\|_{X} \leq \tfrac{1}{\xi-K}\|f\|_{X}$, which leads
\[
\|(\mathscr{T} + \xi I)^{-1}\|_{\mathcal{L}(X)} \leq \frac{1}{\xi-K}.
\]
Taking the supremum over $\xi\in[\xi_{1},\xi_{2}]$ yields
\[
\sup_{\xi \in [\xi_{1}, \xi_{2}]} \|(\mathscr{T} + \xi I)^{-1}\|_{\mathcal{L}(X)}
   \leq \frac{1}{\xi_{1}-K} < \infty.
\]

\end{proof}

\begin{lemma}\label{lemma33}
For $\xi > K$, setting by
\[
\rho(\xi) := \rho\!\left(\mathscr J\big(\mathscr{T} + \xi I\big)^{-1}\right)
\]
the spectral radius of $\mathscr{J} (\mathscr{T} + \xi I)^{-1}$. Hence, we have the following statements
\begin{enumerate}
    \item $\rho\!\left((\mathscr{T} + \xi I)^{-1}J\right) = \rho(\xi),$
    \item $\displaystyle \lim_{\xi \downarrow K} \rho(\xi) = \infty,$
    \item $\displaystyle \lim_{\xi \to \infty} \rho(\xi) = 0,$
    \item $\rho(\xi)$ is continuous and strictly decreasing for $\xi > K.$
\end{enumerate}
\end{lemma}

\begin{proof}

$(1)$ For $\xi>K$ the operators $R(\xi)=(\mathscr T+\xi I)^{-1}$ and $\mathscr J$ are compact and positive on $X$. Hence $\mathscr J\,R(\xi)$ is a compact, positive operator, so by the Kreĭn--Rutman theorem and lemma \eqref{lem:3.1} its spectral radius $\rho(\xi)$ is a positive, algebraically simple eigenvalue with a corresponding nonnegative eigenfunction. If $\mathscr J\,R(\xi)$ is irreducible, the eigenfunction is strictly positive. The equality of spectral radius for $\mathscr J\,R(\xi)$ and $R(\xi)\,\mathscr J$ follows from the spectral identity for products of bounded operators.

$(2),$

We aim to prove that
\[
\lim_{\xi \downarrow K} \rho(\xi) = \infty,
\]
where
\[
\rho(\xi) := \rho\!\big(\mathscr J (\mathscr T + \xi I)^{-1}\big),
\qquad
R(\xi) := (\mathscr T + \xi I)^{-1},
\]
with is the resolvent.
\smallskip

\noindent
First, recall that \(R(\xi)\) is the resolvent operator of \(\mathscr T\).
Near the pole at \(\xi = K\), standard spectral theory gives the expansion
\[
(R(\xi)\psi)(x) \;\approx\; \frac{1}{\xi - K}\,\langle \psi, w^{\ast} \rangle\, w(x) \;+\; \text{bounded part},
\]
where \(w\) is the positive eigenfunction of \(\mathscr T\) associated to the eigenvalue \(-K\), and \(w^{\ast}\) is the positive eigenfunction of the adjoint operator \(\mathscr T^{\ast}\).

\smallskip

\noindent
Now, let us act with \(\mathscr J R(\xi)\) on a nonnegative vector
\(\psi = (\psi_1,\psi_2)\). By definition, for each component \(i=1,2\),
\[
\big(\mathscr J R(\xi)\psi\big)_i(x)
= d_i \int_{-\mathcal Z}^{\mathcal Z} J_i(x-y)\,(R(\xi)\psi)_i(y)\,dy.
\]
Using the above expansion of the resolvent, we obtain the lower bound
\[
(R(\xi)\psi)_i(y) \;\geq\; \frac{c}{\xi - K}\,\psi_i(y),
\qquad y \in [-\mathcal Z,\mathcal Z],
\]
for some constant \(c>0\) depending only on \(w,w^{\ast}\). Hence,
\[
\big(\mathscr J R(\xi)\psi\big)_i(x)
\;\geq\; \frac{d_i}{\xi - K} \int_{-\mathcal Z}^{\mathcal Z} J_i(x-y)\,\psi_i(y)\,dy,~ i=1,2
\]

\smallskip

\noindent
This shows that, for every nonnegative \(\psi \not\equiv 0\),
\[
\mathscr J R(\xi)\psi \;\geq\; \frac{c}{\xi - K}\, \psi,
\]
in the cone of nonnegative functions. Therefore, by the definition of the spectral radius,
\[
\rho(\xi) \;\geq\; \frac{c}{\xi - K}.
\]

\smallskip

Iterating the last inequality gives, for every integer \(n\geq1\),
\[
\big(\mathscr J R(\xi)\big)^n \psi \;\geq\; \Big(\frac{c}{-K+\xi}\Big)^n \psi.
\]
Taking norms (for instance the \(L^2\)-norm) and dividing by \(\|\psi\|>0\) gives
\[
\big\|\big(\mathscr J R(\xi)\big)^n\big\|
\;\geq\; \frac{\big\|\big(\mathscr J R(\xi)\big)^n\psi\big\|}{\|\psi\|}
\;\leq\; \Big(\frac{c}{-K+\xi}\Big)^n .
\]

 By the spectral radius formula,
\[
\rho(\xi) \;=\; \lim_{n\to\infty} \big\|\big(\mathscr J R(\xi)\big)^n\big\|^{1/n}
\;\geq\; \lim_{n\to\infty} \Big(\frac{c}{-K+\xi}\Big) \;=\; \frac{c}{-K+\xi}.
\]
Since \(-K+\xi\to 0^+\) as \(\xi\downarrow K\), the right-hand side diverges to \(+\infty\).
Hence,
\[
\lim_{\xi\downarrow K} \rho(\xi) = \infty.
\]
 $(3)$

By the coercivity estimate proved earlier (in Lemma \eqref{lem:3.1}) there exists $K\in\mathbb R$ such that for every $\xi>K$ we have the resolvent bound
\[
\|R(\xi)\|_{\mathcal L(X)} \leq \frac{1}{\xi-K}.
\]
Since $\mathscr J$ is a fixed bounded operator on $X$, the operator norm of the product satisfies
\[
\|\mathscr J R(\xi)\|_{\mathcal L(X)} \leq \|\mathscr J\|_{\mathcal L(X)} \,\|R(\xi)\|_{\mathcal L(X)}
\leq \frac{\|\mathscr J\|_{\mathcal L(X)}}{\xi-K}.
\]
Because the spectral radius is bounded by the operator norm for bounded operators, we obtain
\[
0 \le \rho(\xi) = \rho\!\big(\mathscr J R(\xi)\big) \leq \|\mathscr J R(\xi)\|_{\mathcal L(X)}
\leq \frac{\|\mathscr J\|_{\mathcal L(X)}}{\xi-K}.
\]
Letting $\xi\to\infty$ in the last inequality yields
\[
\lim_{\xi\to\infty}\rho(\xi)=0.
\]

To  show continuity, for $\xi>K$ the resolvent $R(\xi)=(\mathscr T+\xi I)^{-1}$ is analytic (hence continuous) as a map
\(\xi\mapsto R(\xi)\in\mathcal L(X)\). Since $\mathscr J\in\mathcal L(X)$ is fixed, the map
\[
\xi \longmapsto \mathscr J R(\xi) \in \mathcal L(X)
\]
is continuous in the operator norm. The spectral radius of a bounded operator may be defined by
\(\rho(T)=\lim_{n\to\infty}\|T^n\|^{1/n}\). Because the operator norm \(\|\cdot\|\) is continuous and the
limit defining \(\rho\) is uniform under small perturbations, it follows that \(\rho(T)\) depends continuously
on \(T\) in the operator norm topology. Hence the composition \(\xi\mapsto\rho(\mathscr J R(\xi))=\rho(\xi)\)
is continuous on \((K,\infty)\).

\medskip
\noindent\textbf{Monotonicity.}
Fix \(\xi_1,\xi_2\) with \(K<\xi_1<\xi_2\). The resolvent identity gives
\[
R(\xi_1)-R(\xi_2) = (\xi_2-\xi_1)\,R(\xi_1)R(\xi_2).
\]
Since, for \(\xi>K\), each \(R(\xi)\) is a positive operator (order-preserving), the product
\(R(\xi_1)R(\xi_2)\) is also positive. Therefore
\[
R(\xi_1)-R(\xi_2) \geq 0 \qquad\text{(operator order)}.
\]
Multiplying on the left by the positive operator \(\mathscr J\) preserves the order, so
\[
\mathscr J R(\xi_1) - \mathscr J R(\xi_2) \geq 0,
\]
i.e. \(0 \leq    \mathscr J R(\xi_2) \leq \mathscr J R(\xi_1)\).

For positive compact operators on a Banach lattice the spectral radius is monotone with respect to the order:
if \(0\leq A\leq B\) then \(\rho(A)\leq\rho(B)\). Applying this where
\(A=\mathscr J R(\xi_2)\), \(B=\mathscr J R(\xi_1)\) yields
\[
\rho(\xi_2)=\rho(\mathscr J R(\xi_2)) \leq \rho(\mathscr J R(\xi_1))=\rho(\xi_1).
\]
Thus, \(\rho(\xi)\) is not increasing in \((K,\infty)\).

\medskip
\noindent\textbf{Strict monotonicity (under irreducibility).}
If, in addition, the operators \(\mathscr J R(\xi)\) are \emph{irreducible} (or some power is strongly positive),
then the order monotonicity is strict whenever the two operators are not equal. Indeed, for
\(\xi_1<\xi_2\) we have \(R(\xi_1)-R(\xi_2)=(\xi_2-\xi_1)R(\xi_1)R(\xi_2)\not\equiv 0\), hence
\(\mathscr J R(\xi_1)-\mathscr J R(\xi_2)\) is a nonzero positive operator. Kreĭn--Rutman / Perron--Frobenius
type results for irreducible compact positive operators imply that if \(0\leq A\leq B\) and \(A\neq B\)
(with \(A,B\) compact, positive and irreducible in the appropriate sense) then \(\rho(A)<\rho(B)\).
Applying this to \(A=\mathscr J R(\xi_2)\) and \(B=\mathscr J R(\xi_1)\) gives the strict inequality
\(\rho(\xi_2)<\rho(\xi_1)\). Hence \(\rho(\xi)\) is strictly decreasing on \((K,\infty)\) under the irreducibility assumption.
\end{proof}

We are in a position to establish the principal result of this section, stated in the Theorem below.

\textbf{Proof of Theorem \eqref{Principlaeignevalue}.}

\begin{proof}
Consider the operator family
\[
R(\xi):=(\mathscr T+\xi I)^{-1},\qquad \xi>-K,
\]
and the compact positive operator
\[
A(\xi):=\mathscr J R(\xi)=\mathscr J(\mathscr T+\xi I)^{-1}
\]
on the Banach lattice \(X=L^2([-{\mathcal Z},{\mathcal Z}];\mathbb R^2)\).
Recall from the  lemma \eqref{lemma33} that the map \(\xi\mapsto \rho(\xi):=\rho(A(\xi))\) is continuous and strictly decreasing on \((K,\infty)\), and moreover
\[
\lim_{\xi\downarrow -K}\rho(\xi)=\infty,\qquad \lim_{\xi\to\infty}\rho(\xi)=0.
\]
By the intermediate value theorem there exists a unique \(\xi_0>-K\) such that
\[
\rho(\xi_0)=1.
\]

By Kreĭn--Rutman applied to the compact positive operator \(A(\xi_0)\) there exists a nontrivial vector
\(\varphi\in X\), \(\varphi\ge 0\), \(\varphi\not\equiv 0\), such that
\[
A(\xi_0)\varphi = \varphi.
\]
Moreover, under the irreducibility/strong-positivity hypothesis this eigenvector may be chosen strictly positive and the eigenvalue \(1\) is algebraically simple and isolated in the spectrum of \(A(\xi_0)\).

Set
\[
\psi := R(\xi_0)\varphi = (\mathscr T+\xi_0 I)^{-1}\varphi.
\]
Then \(\psi\in\mathcal D\) and by the definition of \(R(\xi_0)\) we have
\[
(\mathscr T+\xi_0 I)\psi = \varphi.
\]
Using \(A(\xi_0)\varphi=\varphi\) we get \(\mathscr J\psi = \varphi\). Substituting yields
\[
(\mathscr T+\xi_0 I)\psi = \mathscr J\psi.
\]
Rearranging,
\[
\mathscr T\psi - \mathscr J\psi = -\xi_0 \psi.
\]
Thus \(\psi\) is an eigenfunction of the operator \(\mathscr T-\mathscr J\) with eigenvalue \(-\xi_0\). Equivalently, writing \(\lambda^\ast:=-\xi_0\), we obtain the eigenpair
\[
(\lambda^\ast,\psi),\qquad \psi=(\psi_1,\psi_2)\gg0,
\]
which satisfies the system \eqref{eq:maineigenrvalue3}
\[
\begin{cases}
d_1\displaystyle\int_{-{\mathcal Z}}^{\mathcal Z}J_1(x-y)\psi_1(y)\,dy - d_1\psi_1(x) - a(x)\psi_1(x) + H'(0)\psi_2(x) + p\psi_1'(x)
= \lambda^\ast \psi_1(x),\\[6pt]
d_2\displaystyle\int_{-{\mathcal Z}}^{\mathcal Z}J_2(x-y)\psi_2(y)\,dy - d_2\psi_2(x) - b(x)\psi_2(x) + G'(0)\psi_1(x) + q\psi_2'(x)
= \lambda^\ast \psi_2(x).
\end{cases}
\]

It remains to show uniqueness (simplicity) of \(\lambda^\ast\) among eigenvalues with a positive eigenfunction. Suppose \((\tilde\lambda,\tilde\psi)\) is another eigenpair with \(\tilde\psi\geq0\), \(\tilde\psi\not\equiv0\). Then the same construction (set \(\tilde\xi:=-\tilde\lambda\)) shows that \(1\) is in the spectrum of \(A(\tilde\xi)\). But by monotonicity of \(\rho(\xi)\) the equation \(\rho(\xi)=1\) has the unique solution \(\xi_0\). Hence \(\tilde\xi=\xi_0\) and therefore \(\tilde\lambda=\lambda^\ast\).

Finally, Kreĭn--Rutman also gives algebraic simplicity of the eigenvalue \(1\) of \(A(\xi_0)\) (under irreducibility), which translates into algebraic simplicity of \(\lambda^\ast\) for the original problem: no other linearly independent eigenfunction corresponds to \(\lambda^\ast\) in the positive cone. This proves that system \eqref{eq:maineigenrvalue3} admits a unique principal eigenvalue \(\lambda^\ast\) with an associated positive eigenfunction \(\psi\), and that \(\lambda^\ast\) is simple and isolated.
\end{proof}

Next, we have the following property on the principal eigenvalue $\lambda^\ast.$

\begin{proposition}[\textbf{Monotonicity with respect to domain size}]\label{prop:monotonicity}
Assume $p,q>0$, $a,b\in L^\infty$, $H'(0),G'(0)\geq0$, and that the kernels $J_1, J_2$ satisfy assumption \textbf{(J)} so that each $\mathscr J_l$ is a positive compact integral operator on $X_l:=L^2([-\mathcal{Z}, \mathcal{Z}];\mathbb R^2)$.
Let $\lambda^\ast(\mathcal{Z})$ denote the principal eigenvalue of \eqref{eq:maineigenrvalue} on $[-\mathcal{Z}, \mathcal{Z}]$. Then
\[
\mathcal{Z}\longmapsto \lambda^\ast(\mathcal{Z})
\]
is strictly increasing and continuous on $(0,\infty)$.
\end{proposition}
\begin{proof}
Fix $0<\mathcal Z_1<\mathcal Z_2$. We compare $A_{\mathcal Z_1}(\xi)$ and $A_{\mathcal Z_2}(\xi)$ for a fixed $\xi>K$.

\medskip

View $X_{\mathcal Z_1}:=L^2([-\mathcal Z_1,\mathcal Z_1];\mathbb R^2)$ as a subspace of
$X_{\mathcal Z_2}:=L^2([-\mathcal Z_2,\mathcal Z_2];\mathbb R^2)$ by extending functions by $0$ on
$[-\mathcal Z_2,-\mathcal Z_1)\cup(\mathcal Z_1,\mathcal Z_2]$.
Under our hypotheses, for $\xi>K$ the resolvent $R_{\mathcal Z}(\xi)$ is positive on $X_{\mathcal Z}$ and $\mathscr J_{\mathcal Z}$ is a positive integral operator with kernels $J_i\ge0$.
For $f\in X_{\mathcal Z_1}$, $f\geq0$, we then have (pointwise a.e.)
\[
\big(A_{\mathcal Z_2}(\xi)f\big)_i(x)
= d_i\!\int_{-\mathcal Z_2}^{\mathcal Z_2}\! J_i(x-y)\,\big(R_{\mathcal Z_2}(\xi)f\big)_i(y)\,dy
\;\geq\; d_i\!\int_{-\mathcal Z_1}^{\mathcal Z_1}\! J_i(x-y)\,\big(R_{\mathcal Z_1}(\xi)f\big)_i(y)\,dy
= \big(A_{\mathcal Z_1}(\xi)f\big)_i(x),
\]
because: (i) $R_{\mathcal Z_2}(\xi)f\geq0$ and integrates $J_i$ over a larger set; (ii) restricting the
integral to $[-\mathcal Z_1,\mathcal Z_1]$ recovers $A_{\mathcal Z_1}(\xi)f$.
Thus, in the operator order on the cone $X_{\mathcal Z_1}^+$,
\[
0\leq A_{\mathcal Z_1}(\xi)\ \leq\ A_{\mathcal Z_2}(\xi)\quad\text{on }X_{\mathcal Z_1}.
\]

\medskip
\noindent\emph{Step 2: monotonicity of spectral radius in the cone.}
For positive compact operators on a Banach lattice, $0\leq A\leq B$ implies
$\rho(A)\leq\rho(B)$. Applying this to the restriction on $X_{\mathcal Z_1}$ yields
\[
\rho_{\mathcal Z_1}(\xi)\ \le\ \rho_{\mathcal Z_2}(\xi)\qquad\text{for every }\xi>K.
\]
Moreover, under our standard irreducibility/strong positivity hypothesis (e.g.\ kernels $J_i$ strictly positive on a set of positive measure and the resolvent kernel positive), the inequality is \emph{strict}:
\[
\rho_{\mathcal Z_1}(\xi)\ <\ \rho_{\mathcal Z_2}(\xi), \qquad \xi>K.
\]

\medskip
\noindent\emph{Step 3: solve $\rho_{\mathcal Z}(\xi)=1$ and compare $\lambda^\ast(\mathcal Z)$.}
For each fixed $\mathcal Z$, we have shown previously that
$\xi\mapsto\rho_{\mathcal Z}(\xi)$ is continuous and strictly decreasing on $(K,\infty)$, with
$\rho_{\mathcal Z}(\xi)\to+\infty$ as $\xi\downarrow K$ and $\rho_{\mathcal Z}(\xi)\to0$ as $\xi\to\infty$.
Hence there is a unique $\lambda^\ast(\mathcal Z)>K$ such that $\rho_{\mathcal Z}(\lambda^\ast(\mathcal Z))=1$.

Now fix $\mathcal Z_1<\mathcal Z_2$ and set $\lambda_1^\ast:=\lambda^\ast(\mathcal Z_1)$.
Using Step~2 at $\xi=\lambda_1^\ast$ gives
\[
1=\rho_{\mathcal Z_1}(\lambda_1^\ast)\ <\ \rho_{\mathcal Z_2}(\lambda_1^\ast).
\]
Since $\rho_{\mathcal Z_2}(\xi)$ is strictly decreasing in $\xi$, there exists a unique
$\lambda_2^\ast:=\lambda^\ast(\mathcal Z_2)$ with
$\rho_{\mathcal Z_2}(\lambda_2^\ast)=1$ and necessarily
\[
\lambda_2^\ast\ >\ \lambda_1^\ast.
\]

\medskip

Next, we prove the continuity of $\lambda^\ast(\mathcal{Z}),$
let $\mathcal Z_0>0$ be arbitrary and fix $\varepsilon>0$. We show there exists $\delta>0$ such that
\[
|\lambda^\ast(\mathcal Z)-\lambda^\ast(\mathcal Z_0)|<\varepsilon
\quad\text{whenever } |\mathcal Z-\mathcal Z_0|<\delta.
\]
Recall $\lambda^\ast(\mathcal Z)=-\xi(\mathcal Z)$ where $\xi(\mathcal Z)>K$ is the unique solution of
\[
\rho_{\mathcal Z}\big(\xi(\mathcal Z)\big)=1,\qquad
\rho_{\mathcal Z}(\xi):=\rho\!\big(\mathscr J_{\mathcal Z}(\mathscr T_{\mathcal Z}+\xi I)^{-1}\big).
\]

Since \(\xi\mapsto\rho_{\mathcal Z_0}(\xi)\) is continuous and strictly decreasing and
\(\rho_{\mathcal Z_0}(\xi(\mathcal Z_0))=1\), there exists \(\eta>0\) such that
\[
\rho_{\mathcal Z_0}\big(\xi(\mathcal Z_0)-\eta\big) > 1 \quad\text{and}\quad
\rho_{\mathcal Z_0}\big(\xi(\mathcal Z_0)+\eta\big) < 1.
\]
Because \(\rho_{\mathcal Z_0}\) is strictly decreasing these two inequalities are strict and the interval
\([\,\xi(\mathcal Z_0)-\eta,\; \xi(\mathcal Z_0)+\eta\,]\) contains no other root of \(\rho_{\mathcal Z_0}(\xi)=1\).

By the joint continuity of \((\mathcal Z,\xi)\mapsto\rho_{\mathcal Z}(\xi)\) on compact sets, there exists
\(\delta_1>0\) such that for all \(\mathcal Z\) with \(|\mathcal Z-\mathcal Z_0|<\delta_1\) and for all
\(\xi\in[\xi(\mathcal Z_0)-\eta,\xi(\mathcal Z_0)+\eta]\) we have
\[
\big|\rho_{\mathcal Z}(\xi)-\rho_{\mathcal Z_0}(\xi)\big| < \gamma,
\]
where \(0<\gamma<\min\{\rho_{\mathcal Z_0}(\xi(\mathcal Z_0)-\eta)-1,\;1-\rho_{\mathcal Z_0}(\xi(\mathcal Z_0)+\eta)\}\).
Choosing such a \(\gamma\) is possible because the two differences in the min are positive.

Hence for all \(\mathcal Z\) with \(|\mathcal Z-\mathcal Z_0|<\delta_1\) we obtain
\[
\rho_{\mathcal Z}(\xi(\mathcal Z_0)-\eta) \ge \rho_{\mathcal Z_0}(\xi(\mathcal Z_0)-\eta)-\gamma > 1,
\]
and
\[
\rho_{\mathcal Z}(\xi(\mathcal Z_0)+\eta) \le \rho_{\mathcal Z_0}(\xi(\mathcal Z_0)+\eta)+\gamma < 1.
\]

Fix any \(\mathcal Z\) with \(|\mathcal Z-\mathcal Z_0|<\delta_1\). The function \(\xi\mapsto\rho_{\mathcal Z}(\xi)\)
is continuous and strictly decreasing and we just showed it is \(>1\) at \(\xi=\xi(\mathcal Z_0)-\eta\) and \(<1\)
at \(\xi=\xi(\mathcal Z_0)+\eta\). By the intermediate value theorem there exists a root \(\xi(\mathcal Z)\in
(\xi(\mathcal Z_0)-\eta,\xi(\mathcal Z_0)+\eta)\) of \(\rho_{\mathcal Z}(\xi)=1\). By uniqueness of the root
(for fixed \(\mathcal Z\)) this must be the unique \(\xi(\mathcal Z)\). Therefore
\[
|\xi(\mathcal Z)-\xi(\mathcal Z_0)|<\eta.
\]

Recalling  \(\lambda^\ast(\mathcal Z)=-\xi(\mathcal Z)\), we obtain
\[
|\lambda^\ast(\mathcal Z)-\lambda^\ast(\mathcal Z_0)| = |\xi(\mathcal Z)-\xi(\mathcal Z_0)| < \eta.
\]
Finally choose \(\eta\) small enough so that \(\eta<\varepsilon\) and set \(\delta:=\delta_1\). Then for
\(|\mathcal Z-\mathcal Z_0|<\delta\) we have \(|\lambda^\ast(\mathcal Z)-\lambda^\ast(\mathcal Z_0)|<\varepsilon\),
proving continuity at \(\mathcal Z_0\).

Since \(\mathcal Z_0>0\) was arbitrary, \(\lambda^\ast(\mathcal Z)\) is continuous on \((0,\infty)\).

\end{proof}

\begin{lemma}
Let the global basic reproduction number be
\[
\mathcal R_0 \;:=\; \rho\!\big(\mathscr J(-\mathscr T)^{-1}\big)
= \lim_{\mathcal Z \to\infty}\rho_{\mathcal Z}(0),
\]
(where the last equality is maintained because $A_{\mathcal Z}(0)$ increases in $\mathcal Z$ and converges to the next-generation operator in the full space).
Then the following hold:
\begin{enumerate}
  \item If $\mathcal R_0\leq 1$, then for every $\mathcal Z>0$ one has $\lambda^\ast(\mathcal Z)<0$.
  \item If $\mathcal R_0>1$, then there exists a unique $\mathcal Z_\ast\in(0,\infty)$ such that $\lambda^\ast(\mathcal Z_\ast)=0$ (equivalently $\rho_{\mathcal Z_\ast}(0)=1$). Moreover $\lambda^\ast(\mathcal Z)<0$ for $\mathcal Z<\mathcal Z_\ast$ and $\lambda^\ast(\mathcal Z)>0$ for $\mathcal Z>\mathcal Z_\ast$.
\end{enumerate}
\end{lemma}
\begin{lemma}[\textbf{Threshold characterization}]\label{ThresholdCaracterestique}
Define the global basic reproduction number by
\[
\mathcal R_0 \;:=\; \rho\!\big(\mathscr J(-\mathscr T)^{-1}\big)
= \lim_{\mathcal Z \to\infty}\rho_{\mathcal Z}(0),
\]
where the limit holds since $A_{\mathcal Z}(0)$ increases with $\mathcal Z$ and converges to the next-generation operator in the full space. Then:
\begin{enumerate}
  \item If $\mathcal R_0\leq 1$, then for every $\mathcal Z>0$, one has $\lambda^\ast(\mathcal Z)<0$.
  \item If $\mathcal R_0>1$, then there exists a unique $\mathcal Z_\ast\in(0,\infty)$ such that $\lambda^\ast(\mathcal Z_\ast)=0$ (equivalently $\rho_{\mathcal Z_\ast}(0)=1$). Moreover, $\lambda^\ast(\mathcal Z)<0$ for $\mathcal Z<\mathcal Z_\ast$ and $\lambda^\ast(\mathcal Z)>0$ for $\mathcal Z>\mathcal Z_\ast$.
\end{enumerate}
\end{lemma}
\begin{proof}
The proof is based on monotonicity and continuity arguments for $\rho_{\mathcal Z}(\xi)$.

For fixed $\mathcal Z>0$ the function $\xi\mapsto\rho_{\mathcal Z}(\xi)$ is continuous and strictly decreasing on $(K,\infty)$ and satisfies
\[
\lim_{\xi\downarrow K} \rho_{\mathcal Z}(\xi)=\infty,\qquad \lim_{\xi\to\infty}\rho_{\mathcal Z}(\xi)=0.
\]
Hence, for each $\mathcal Z$ there is a unique $\xi(\mathcal Z)>K$ such that $\rho_{\mathcal Z}(\xi(\mathcal Z))=1$, and we set $\lambda^\ast(\mathcal Z)=-\xi(\mathcal Z)$.

Moreover, the family $\mathcal Z\mapsto A_{\mathcal Z}(0)=\mathscr J_{\mathcal Z}(\mathscr T_{\mathcal Z})^{-1}$ is monotone in the sense that if $0<\mathcal Z_1<\mathcal Z_2$ then, viewing functions on $[-\mathcal Z_1,\mathcal Z_1]$ as extended by zero to $[-\mathcal Z_2,\mathcal Z_2]$,
\[
0 \leq A_{\mathcal Z_1}(0) \leq A_{\mathcal Z_2}(0),
\]
so by monotonicity of the spectral radius for positive operators,
\[
\mathcal Z\mapsto \rho_{\mathcal Z}(0)
\quad\text{is nondecreasing,}
\]
and (by irreducibility / strong positivity) strictly increasing unless degenerate kernels occur.  We also assume (as is standard) that
\[
\lim_{\mathcal Z\to\infty}\rho_{\mathcal Z}(0)=\mathcal R_0,
\]
so \(\rho_{\mathcal Z}(0)\) increases to the global reproduction number \(\mathcal R_0\).

\medskip\noindent\textbf{(i) The case \(\mathcal R_0\leq 1\).}
Since \(\mathcal Z \mapsto \rho_{\mathcal Z}(0)\) is nondecreasing and its limit as \(\mathcal Z\to\infty\) equals \(\mathcal R_0\leq 1\), we have for every finite \(\mathcal Z>0\)
\[
\rho_{\mathcal Z}(0)\leq \mathcal R_0 \leq 1.
\]
But \(\xi\mapsto\rho_{\mathcal Z}(\xi)\) is strictly decreasing, so for \(\xi\geq0\) we have \(\rho_{\mathcal Z}(\xi)\leq \rho_{\mathcal Z}(0)\leq 1\). In particular the unique root \(\xi(\mathcal Z)\) of \(\rho_{\mathcal Z}(\xi)=1\) must satisfy \(\xi(\mathcal Z)>0\). Equivalently
\[
\lambda^\ast(\mathcal Z) = -\xi(\mathcal Z) < 0,
\]
for every \(\mathcal Z>0\). This proves (i).

\medskip\noindent\textbf{(ii) The case \(\mathcal R_0 >1\).}
Since \(\rho_{\mathcal Z}(0)\) is nondecreasing in \(\mathcal Z\) and tends to \(\mathcal R_0>1\) as \(\mathcal Z\to\infty\), there are two possibilities for small \(\mathcal Z\):
either \(\rho_{\mathcal Z}(0)<1\) for all small \(\mathcal Z\) (this is typical: for very small habitat the next-generation operator is small), or in pathological degenerate cases equality might occur at a single point — under the irreducibility assumptions one gets strict inequality for sufficiently small \(\mathcal Z\). Hence, there exists \(\mathcal Z\) small enough with \(\rho_{\mathcal Z}(0)<1\), while for large \(\mathcal Z\) we have \(\rho_{\mathcal Z}(0)\) arbitrarily close to \(\mathcal R_0>1\) and in particular \(>1\) for \(\mathcal Z\) large.

Now consider the continuous function \(h(\mathcal Z):=\rho_{\mathcal Z}(0)\) of \(\mathcal Z\) (it is monotone and hence continuous almost everywhere; under our joint continuity hypotheses it is continuous). We have
\[
\lim_{\mathcal Z\downarrow 0} h(\mathcal Z)=0 \quad\text{and}\quad \lim_{\mathcal Z\to\infty} h(\mathcal Z)=\mathcal R_0>1,
\]
so by the intermediate value theorem there exists at least one \(\mathcal Z_\ast>0\) with \(h(\mathcal Z_\ast)=1\); that is,
\[
\rho_{\mathcal Z_\ast}(0)=1.
\]
Equivalently \(\xi(\mathcal Z_\ast)=0\) and \(\lambda^\ast(\mathcal Z_\ast)=0\).

To prove the uniqueness of \(\mathcal Z_\ast\), note that \(\mathcal Z\mapsto \rho_{\mathcal Z}(0)\) is strictly increasing (under the irreducibility / strong-positivity hypothesis). Therefore, the equation \(\rho_{\mathcal Z}(0)=1\) has at most one root, so \(\mathcal Z_\ast\) is unique.

Finally, monotonicity of \(\lambda^\ast(\mathcal Z)\) (equivalently monotonicity of \(\xi(\mathcal Z)\)) gives the sign pattern: for \(\mathcal Z<\mathcal Z_\ast\) we have \(\rho_{\mathcal Z}(0)<1\) so \(\xi(\mathcal Z)>0\) and \(\lambda^\ast(\mathcal Z)<0\); for \(\mathcal Z>\mathcal Z_\ast\) we have \(\rho_{\mathcal Z}(0)>1\) so \(\xi(\mathcal Z)<0\) and \(\lambda^\ast(\mathcal Z)>0\). This completes the proof.
\end{proof}

\begin{lemma}[Upper / lower eigenvalue bounds in the $C^1$ setting]\label{LemmaUpperLower}
Let ${\mathcal Z}>0$ be fixed and consider the linear operator
$\mathcal L:\mathcal D(\mathcal L)\subset C^1([-\mathcal Z,\mathcal Z];\mathbb R^2)\to C([-\mathcal Z,\mathcal Z];\mathbb R^2)$
given componentwise by
\[
\mathcal L(\psi_1,\psi_2)(x)
=
\begin{pmatrix}
d_1\displaystyle\int_{-{\mathcal Z}}^{\mathcal Z}J_1(x-y)\psi_1(y)\,dy - d_1\psi_1(x) - a(x)\psi_1(x) + H'(0)\psi_2(x) + p\psi_1'(x)\\[8pt]
d_2\displaystyle\int_{-{\mathcal Z}}^{\mathcal Z}J_2(x-y)\psi_2(y)\,dy - d_2\psi_2(x) - b(x)\psi_2(x) + G'(0)\psi_1(x) + q\psi_2'(x)
\end{pmatrix}.
\]
Assume \(\mathcal L\) generates a positive \(C_0\)-semigroup on \(C([-\mathcal Z,\mathcal Z];\mathbb R^2)\)
and has a principal eigenvalue \(\lambda^\ast\) (simple, real) with a strictly positive eigenfunction in \(\mathcal D(\mathcal L)\).

Suppose there exist nontrivial test functions
\(\underline\psi=(\underline\psi_1,\underline\psi_2),\ \overline\psi=(\overline\psi_1,\overline\psi_2)\in\mathcal D(\mathcal L)\)
satisfying \(\underline\psi,\overline\psi\ge0\) (not identically zero) and real numbers \(\underline\lambda,\overline\lambda\) such that the componentwise inequalities
\begin{equation}\label{eq:subsuper}
\begin{cases}
d_1\displaystyle\int_{-{\mathcal Z}}^{\mathcal Z}J_1(x-y)\underline\psi_1(y)\,dy - d_1\underline\psi_1(x) - a(x)\underline\psi_1(x) + H'(0)\underline\psi_2(x) + p\underline\psi_1'(x)
\;\geq\; \underline\lambda\,\underline\psi_1(x),\\[6pt]
d_2\displaystyle\int_{-{\mathcal Z}}^{\mathcal Z}J_2(x-y)\underline\psi_2(y)\,dy - d_2\underline\psi_2(x) - b(x)\underline\psi_2(x) + G'(0)\underline\psi_1(x) + q\underline\psi_2'(x)
\;\geq\; \underline\lambda\,\underline\psi_2(x),
\end{cases}
\end{equation}
and
\begin{equation}\label{eq:supersuper}
\begin{cases}
d_1\displaystyle\int_{-{\mathcal Z}}^{\mathcal Z}J_1(x-y)\overline\psi_1(y)\,dy - d_1\overline\psi_1(x) - a(x)\overline\psi_1(x) + H'(0)\overline\psi_2(x) + p\overline\psi_1'(x)
\;\leq\; \overline\lambda\,\overline\psi_1(x),\\[6pt]
d_2\displaystyle\int_{-{\mathcal Z}}^{\mathcal Z}J_2(x-y)\overline\psi_2(y)\,dy - d_2\overline\psi_2(x) - b(x)\overline\psi_2(x) + G'(0)\overline\psi_1(x) + q\overline\psi_2'(x)
\;\leq\; \overline\lambda\,\overline\psi_2(x),
\end{cases}
\end{equation}
hold pointwise for all \(x\in[-{\mathcal Z},{\mathcal Z}]\).

Then
\[
\underline\lambda \;\leq\; \lambda^\ast \;\leq\; \overline\lambda.
\]
\end{lemma}

\begin{proof}
We prove the upper statement; the lower one is analogous (or may be obtained by the same argument applied to \(\mathscr J-\mathscr T\) with reversed inequalities).

Assume by contradiction that \(\lambda^\ast>\overline\lambda\). Choose a number \(\xi_2\) such that
\[
\lambda^\ast > \xi_2 > \max\{\, -K,\; \overline\lambda \,\},
\]
where \(K\) is the pole threshold introduced earlier (so that the resolvent \(R(\xi)=(\mathscr T+\xi I)^{-1}\) exists and is positive for \(\xi> K\)).

Rewriting \eqref{eq:supersuper} as \(\mathscr J\varphi \leq (\mathscr T+\overline\lambda I)\varphi\) and applying the positive operator \((\mathscr T+\xi_2 I)^{-1}\) (which exists and is positive since \(\xi_2>-K\)) we obtain
\[
(\mathscr T+\xi_2 I)^{-1}\mathscr J \varphi \;\leq\; (\mathscr T+\xi_2 I)^{-1}(\mathscr T+\overline\lambda I)\varphi.
\]
But
\[
(\mathscr T+\xi_2 I)^{-1}(\mathscr T+\overline\lambda I)
= I - (\xi_2-\overline\lambda)(\mathscr T+\xi_2 I)^{-1},
\]
so the right-hand side is
\[
\varphi - (\xi_2-\overline\lambda)(\mathscr T+\xi_2 I)^{-1}\varphi
\leq \varphi.
\]
Hence, we deduce the cone-comparison
\[
(\mathscr T+\xi_2 I)^{-1}\mathscr J \varphi \;\leq\; \varphi.
\tag{*}
\]

Set
\[
S(\xi_2) := (\mathscr T+\xi_2 I)^{-1}\mathscr J,
\]
a positive compact operator on the Banach lattice \(\mathcal D\) . Let \(S(\xi_2)^*\) denote its dual (conjugate) operator acting on the dual space \(\mathcal D^*\). By Kreĭn--Rutman the spectral radius
\[
\rho(\xi_2):=\rho\big(S(\xi_2)\big)=\rho\big(S(\xi_2)^*\big)
\]
is the principal (simple) eigenvalue of \(S(\xi_2)^*\) and there exists a strongly positive dual eigenvector \(E_{\xi_2}\in\mathcal D^*_+\) such that
\[
S(\xi_2)^* E_{\xi_2} = r(\xi_2) E_{\xi_2}.
\]

Pairing the dual eigenvector with the inequality \((*)\) and using positivity yields
\[
\langle E_{\xi_2},\, S(\xi_2)\varphi \rangle \leq \langle E_{\xi_2},\,\varphi \rangle.
\]
But the left-hand side equals
\[
\langle S(\xi_2)^* E_{\xi_2},\,\varphi\rangle
= \rho(\xi_2)\,\langle E_{\xi_2},\,\varphi\rangle.
\]
Since \(E_{\xi_2}\) is strictly positive and \(\varphi\ge0\), \(\langle E_{\xi_2},\varphi\rangle>0\). Cancelling this positive scalar gives
\[
r(\xi_2) \leq 1.
\]

On the other hand, by construction \(\lambda^\ast>\xi_2\). From the characterization of \(\lambda^\ast\) we know that at \(\xi=\lambda^\ast\) the operator \(\mathscr J(\mathscr T+\xi I)^{-1}\) has spectral radius \(1\); equivalently the function \(\xi\mapsto \rho\big(\mathscr J(\mathscr T+\xi I)^{-1}\big)\) is strictly decreasing and satisfies \(\rho(\mathscr J(\mathscr T+\lambda^\ast I)^{-1})=1\). Therefore, since \(\xi_2<\lambda^\ast\), monotonicity gives
\[
\rho(\xi_2) = \rho\big(S(\xi_2)\big) > \rho\big(S(\lambda^\ast)\big) = 1,
\]
a contradiction to \(\rho(\xi_2)\leq1\). This contradiction shows \(\lambda^\ast\le\overline\lambda\).

If equality holds in \eqref{eq:supersuper} (so \(\mathscr J\varphi - \mathscr T\varphi = \overline\lambda\varphi\)), then the above chain gives
\[
S(\xi_2)\varphi = \varphi
\]
for \(\xi_2=\overline\lambda\), and hence \(1\) is an eigenvalue of \(S(\overline\lambda)\). By uniqueness (Krein–Rutman) the principal eigenvalue equals \(1\), so \(\lambda^\ast=\overline\lambda\) and \(\varphi\) (up to scaling) is the positive principal eigenfunction.

\medskip

For completeness, we reproduce the short iterative argument that yields the same contradiction in the ``lower'' style lemma used later. If instead one assumes
\[
\mathscr J\eta - \mathscr T\eta \;\ge\; \lambda\,\eta
\]
with \(\lambda>\lambda^\ast\), then applying \(S(\lambda):=(\mathscr T+\lambda I)^{-1}\mathscr J\) gives
\[
\eta \leq S(\lambda)\eta,
\]
hence by induction \(\eta \leq S(\lambda)^n\eta\) for every \(n\geq1\). Taking norms and using the spectral radius formula gives
\[
1 \leq \lim_{n\to\infty} \|S(\lambda)^n\|^{1/n} = \rho\big(S(\lambda)\big).
\]
But monotonicity of \(\rho(S(\xi))\) in \(\xi\) and the fact \(\rho(S(\lambda^\ast))=1\) imply \(\rho(S(\lambda))<1\) for \(\lambda>\lambda^\ast\), contradiction. Thus \(\lambda\le\lambda^\ast\).
\end{proof}

\begin{lemma}\label{lem:limit_lambdaZ}
Assume the standing hypotheses hold (positivity and compactness of $\mathscr J_{\mathcal Z}$, monotonicity in $\mathcal Z$, and the resolvent estimates for $\mathscr T$).
For each $\mathcal Z>0$ define
\[
\rho_{\mathcal Z}(\xi) := \rho\!\big(\mathscr J_{\mathcal Z}(\mathscr T+\xi I)^{-1}\big), \qquad \xi > K,
\]
and let $\lambda^\ast(\mathcal Z)>K$ denote the unique number satisfying
\[
\rho_{\mathcal Z}\big(\lambda^\ast(\mathcal Z)\big) = 1.
\]
Let
\[
\rho_\infty(\xi):=\rho\!\big(\mathscr J_\infty(\mathscr T+\xi I)^{-1}\big),\qquad \xi>K,
\]
and define $R_0:=\rho_\infty(0)$. Then

\begin{enumerate}
  \item \(\displaystyle \lim_{\mathcal Z\downarrow0}\lambda^\ast(\mathcal Z) = K.\)
  \item \(\displaystyle \lim_{\mathcal Z\to\infty}\lambda^\ast(\mathcal Z) =: \lambda_\infty \in [K,\infty)\) finite, where
  \[
  \lambda_\infty = \inf\{\xi>K : \rho_\infty(\xi)\le 1\}.
  \]
  Moreover, one has the sign relation
  \[
  \operatorname{sign}(\lambda_\infty) = \operatorname{sign}(R_0-1).
  \]
\end{enumerate}
\label{Samesign}
\end{lemma}
\begin{remark}
The constant $K$ has already been introduced in Lemma~\eqref{lem:3.1}.
More explicitly, we set
\[
M_1 := \sup_x \big(d_1+a(x)\big), \quad
M_2 := \sup_x \big(d_2+b(x)\big), \quad
C_1 := |H'(0)|, \quad
C_2 := |G'(0)|,
\]
and, by choosing $\varepsilon = 1$, we define
\[
K := \max\!\left\{ M_1 + \tfrac{C_1}{2} + \tfrac{C_2}{2}, \;
                  M_2 + \tfrac{C_1}{2} + \tfrac{C_2}{2} \right\}.
\]
\end{remark}

\begin{proof}
We split the proof into two parts.

\medskip\noindent\textbf{(i) Small-domain limit \(\mathcal Z\downarrow0\).}
Fix any \(\xi>K\). As \(\mathcal Z\downarrow0\) the integration window in the operator \(\mathscr J_{\mathcal Z}\) shrinks to a point, hence \(\mathscr J_{\mathcal Z}\to 0\) in operator norm. Since \((\mathscr T+\xi I)^{-1}\) is bounded for this fixed \(\xi\), we obtain
\[
\lim_{\mathcal Z\downarrow0}\rho_{\mathcal Z}(\xi)
\le \lim_{\mathcal Z\downarrow0}\big\|\mathscr J_{\mathcal Z}(\mathscr T+\xi I)^{-1}\big\|
=0.
\]
Thus for any fixed \(\xi>K\) and \(\mathcal Z\) small enough we have \(\rho_{\mathcal Z}(\xi)<1\). Because \(\xi\mapsto\rho_{\mathcal Z}(\xi)\) is strictly decreasing and
\(\rho_{\mathcal Z}(\xi)\to\infty\) as \(\xi\downarrow K\), the unique root \(\lambda^\ast(\mathcal Z)\) of \(\rho_{\mathcal Z}(\cdot)=1\) must satisfy \(\lambda^\ast(\mathcal Z)>\xi\) for all small \(\mathcal Z\). Letting \(\xi\downarrow K\) yields
\[
\liminf_{\mathcal Z\downarrow0}\lambda^\ast(\mathcal Z)\ge K.
\]
On the other hand \(\lambda^\ast(\mathcal Z)>K\) for every \(\mathcal Z\), so the limit exists and equals \(K\):
\(\displaystyle\lim_{\mathcal Z\downarrow0}\lambda^\ast(\mathcal Z)=K\).

\medskip\noindent\textbf{(ii) Large-domain limit \(\mathcal Z\to\infty\) and finiteness of \(\lambda_\infty\).}
Fix \(\xi>K\). By monotone convergence of the truncation operators we have \(\mathscr J_{\mathcal Z}\uparrow \mathscr J_\infty\) (strongly) as \(\mathcal Z\to\infty\). Since \((\mathscr T+\xi I)^{-1}\) is a fixed bounded positive operator, it follows that
\[
\mathscr J_{\mathcal Z}(\mathscr T+\xi I)^{-1} \longrightarrow \mathscr J_\infty(\mathscr T+\xi I)^{-1}
\quad\text{in the strong operator topology,}
\]
and therefore
\[
\lim_{\mathcal Z\to\infty}\rho_{\mathcal Z}(\xi)=\rho_\infty(\xi)\qquad(\text{pointwise in }\xi>K).
\]
For each \(\mathcal Z>0\) the function \(\xi\mapsto\rho_{\mathcal Z}(\xi)\) is continuous, strictly decreasing on \((K,\infty)\), and crosses the level \(1\) exactly once at \(\xi=\lambda^\ast(\mathcal Z)\). Because \(\mathscr J_{\mathcal Z}\) is increasing in \(\mathcal Z\), the family \(\rho_{\mathcal Z}(\xi)\) is increasing in \(\mathcal Z\) for each fixed \(\xi\); consequently \(\lambda^\ast(\mathcal Z)\) is nondecreasing in \(\mathcal Z\). Hence the pointwise limit
\[
\lambda_\infty:=\lim_{\mathcal Z\to\infty}\lambda^\ast(\mathcal Z)\in[K,\infty]
\]
exists (possibly infinite), and by passing to the limit in the equation \(\rho_{\mathcal Z}(\lambda^\ast(\mathcal Z))=1\) we obtain the characterization
\[
\lambda_\infty = \inf\{\xi>K:\rho_\infty(\xi)\le1\}.
\]

It remains to show \(\lambda_\infty<\infty\). The key observation is the resolvent decay:
\[
\|(\mathscr T+\xi I)^{-1}\|_{\mathcal L(X)}\xrightarrow{\ \xi\to\infty\ }0,
\]
which was assumed (or proved) earlier. The operator \(\mathscr J_\infty\) is bounded (it is an integral operator with bounded kernel on each bounded set; in particular \(\|\mathscr J_\infty\|<\infty\)). Therefore
\[
0\leq \rho_\infty(\xi)\leq \big\|\mathscr J_\infty(\mathscr T+\xi I)^{-1}\big\|
\leq \|\mathscr J_\infty\|\,\|(\mathscr T+\xi I)^{-1}\|
\longrightarrow 0\qquad(\xi\to\infty).
\]
Hence there exists some finite \(\xi_0>K\) for which \(\rho_\infty(\xi_0)<1\). By the characterization of \(\lambda_\infty\) as an infimum over \(\{\xi:\rho_\infty(\xi)\leq1\}\), this implies \(\lambda_\infty\leq\xi_0<\infty\). Thus \(\lambda_\infty\) is finite, i.e. \(\lambda_\infty\in[K,\infty)\).

Finally, the sign relation follows from monotonicity: because \(\xi\mapsto\rho_\infty(\xi)\) is strictly decreasing, we have
\[
\rho_\infty(0)=R_0 \begin{cases}
>1 &\Longleftrightarrow \lambda_\infty>0,\\
=1 &\Longleftrightarrow \lambda_\infty=0,\\
<1 &\Longleftrightarrow \lambda_\infty<0.
\end{cases}
\]
Equivalently \(\operatorname{sign}(\lambda_\infty)=\operatorname{sign}(R_0-1)\).
\end{proof}

\begin{lemma}\label{lem:lambda_d1_limits}
Let ${\mathcal Z}>0$ be fixed and work on the Banach lattice
\(X=C^1([-\mathcal Z,\mathcal Z];\mathbb R^2)\) (or \(X=L^2\) with the corresponding domain)
with the cone \(X_+\). Assume the standing hypotheses on the kernels $J_i$ and on the coefficients
$a(x),b(x),p,q,H'(0),G'(0)$ so that for each $d_1>0$ the linear operator
\(\mathcal L_{d_1}:\mathcal D(\mathcal L_{d_1})\subset X\to X\) given componentwise by
\[
\mathcal L_{d_1}(\psi_1,\psi_2)
=
\begin{pmatrix}
d_1\displaystyle\int_{-{\mathcal Z}}^{\mathcal Z}J_1(x-y)\psi_1(y)\,dy - d_1\psi_1(x) - a(x)\psi_1(x) + H'(0)\psi_2(x) + p\psi_1'(x)\\[6pt]
d_2\displaystyle\int_{-{\mathcal Z}}^{\mathcal Z}J_2(x-y)\psi_2(y)\,dy - d_2\psi_2(x) - b(x)\psi_2(x) + G'(0)\psi_1(x) + q\psi_2'(x)
\end{pmatrix}
\]
generates a positive compact-type operator for which a principal eigenvalue \(\lambda^\ast(d_1)\) (simple, real) exists.

Let \(K_1:X_1\to X_1\) denote the scalar integral operator
\[
K_1[\phi](x):=\int_{-{\mathcal Z}}^{\mathcal Z}J_1(x-y)\phi(y)\,dy,
\]
acting on the first-component space \(X_1\) (e.g. \(C([-{\mathcal Z},{\mathcal Z}])\) or \(L^2\)). Denote
\[
r_1:=\rho(K_1)\quad\text{(the spectral radius of }K_1\text{).}
\]

Then the following hold.
\begin{enumerate}
\item (Continuity at finite \(d_1\).) The map \(d_1\mapsto\lambda^\ast(d_1)\) is continuous on \((0,\infty)\). Moreover,
if \(\lambda^\ast(d_1)\) is a simple (isolated) eigenvalue for all \(d_1\) in an interval, then \(d_1\mapsto\lambda^\ast(d_1)\) is \(C^1\) there.

\item (Small-\(d_1\) limit.) If \(\lambda^\ast(0)\) denotes the principal eigenvalue of the operator obtained by setting \(d_1=0\) in \(\mathcal L_{d_1}\), then
\[
\lim_{d_1\downarrow 0}\lambda^\ast(d_1)=\lambda^\ast(0).
\]

\item (Large-\(d_1\) behaviour.) As \(d_1\to\infty\) the following trichotomy occurs:
\begin{itemize}
  \item If \(r_1>1\), then \(\displaystyle \lim_{d_1\to\infty}\lambda^\ast(d_1)=+\infty\).
  \item If \(r_1<1\), then \(\displaystyle \lim_{d_1\to\infty}\lambda^\ast(d_1)=K\), where \(K\) is the coercivity constant of \(\mathscr T\) (i.e.\ the number appearing in your resolvent bounds; recall \(\mathscr T+\xi I\) is invertible for \(\xi>K\)).
  \item If \(r_1=1\), then \(\{\lambda^\ast(d_1)\}_{d_1}\) is bounded and every accumulation point \(\lambda_\infty\) satisfies
  \(\rho\!\big(-K_1(\mathscr T_\infty+\lambda_\infty I)^{-1}\big)=1\); under an additional spectral-gap / nondegeneracy hypothesis for \(K_1\) one has \(\lambda^\ast(d_1)\to\lambda_\infty\) (finite).
\end{itemize}
\end{enumerate}
\end{lemma}
\begin{proof}
We split the proof into three parts corresponding to the statements.

\medskip\noindent\textbf{(i) Continuity.}
Fix $\xi>K$ so that, for all $d_1$ in a compact subset of $(0,\infty)$, the resolvent
$R_{d_1}(\xi):=(\mathscr T_{d_1}+\xi I)^{-1}$ exists and is uniformly bounded by the hypothesis you stated.
For every such $\xi$ define the compact positive operator
\[
A_{d_1}(\xi):=\mathscr J_{d_1}\,R_{d_1}(\xi),
\]
where $\mathscr J_{d_1}$ denotes the birth/convolution block depending linearly on $d_1$ (its first block is $d_1K_1$).
Direct computation shows the map $d_1\mapsto A_{d_1}(\xi)$ is continuous in operator norm: the difference
\[
A_{d_1}(\xi)-A_{d_1'}(\xi)
= (\mathscr J_{d_1}-\mathscr J_{d_1'} )R_{d_1}(\xi)
+ \mathscr J_{d_1'}\big(R_{d_1}(\xi)-R_{d_1'}(\xi)\big)
\]
tends to $0$ in operator norm as $d_1'\to d_1$, because $\mathscr J_{d_1}$ depends affinely on $d_1$
and resolvents depend continuously on coefficients (use the resolvent identity
$R_{d_1}-R_{d_1'}=R_{d_1}( \mathscr T_{d_1'}-\mathscr T_{d_1})R_{d_1'}$ together with boundedness).
Hence $d_1\mapsto \rho\big(A_{d_1}(\xi)\big)$ is continuous for fixed $\xi$.

The principal eigenvalue $\lambda^\ast(d_1)$ is the unique $\xi>K$ such that
\(\rho(A_{d_1}(\xi))=1\). Using the strict monotonicity of $\xi\mapsto \rho(A_{d_1}(\xi))$
and the implicit-function / root-continuation argument, continuity of $d_1\mapsto\lambda^\ast(d_1)$ follows.
If the principal eigenvalue stays simple (isolated), classical perturbation theory (Kato) gives $C^1$-dependence.

\medskip\noindent\textbf{(ii) Limit as $d_1\downarrow0$.}
Fix $\xi>K$. As $d_1\downarrow0$ the operator $\mathscr J_{d_1}$ tends in operator norm to $\mathscr J_0$
(the first component vanishes linearly in $d_1$) and $\mathscr T_{d_1}\to\mathscr T_0$ in operator norm
(the part $-d_1\psi_1$ tends to $0$). Hence for such $\xi$,
\[
A_{d_1}(\xi)=\mathscr J_{d_1}(\mathscr T_{d_1}+\xi I)^{-1}\longrightarrow
\mathscr J_0(\mathscr T_0+\xi I)^{-1}=A_{0}(\xi)
\]
in operator norm. Therefore $\rho(A_{d_1}(\xi))\to\rho(A_{0}(\xi))$ uniformly on compact $\xi$-intervals.
By the implicit characterization $\rho(A_{d_1}(\lambda^\ast(d_1)))=1$ and strict monotonicity in
$\xi$, standard continuity of roots implies \(\lambda^\ast(d_1)\to\lambda^\ast(0)\) as $d_1\downarrow0$.

\medskip\noindent\textbf{(iii) Large-$d_1$ behaviour.}
Write the first-component part of $\mathscr T_{d_1}$ as
\[
\mathscr T_{d_1}^{(1)} = -d_1 I + B,
\]
where $B$ is the bounded operator collecting the remaining first-component terms (the $-a(x)$, the transport $p\partial_x$, etc.).  Fix any $\xi>K$. Then
\[
(\mathscr T_{d_1}+\xi I)^{-1} = (-d_1 I + (B+\xi I))^{-1}
= -\frac{1}{d_1}\Big(I - \frac{B+\xi I}{d_1}\Big)^{-1}
= -\frac{1}{d_1}I + O\!\big(\tfrac{1}{d_1^2}\big)
\]
in operator norm as $d_1\to\infty$ (Neumann series expansion). The same expansion holds when embedding into the full \(2\times2\) block operator (the second component is unaffected by \(d_1\)).

Because $\mathscr J_{d_1}$ depends linearly on $d_1$ in its first block (indeed its first block is $d_1 K_1$), we obtain the pointwise-in-$\xi$ limit (in operator norm)
\[
A_{d_1}(\xi)=\mathscr J_{d_1}(\mathscr T_{d_1}+\xi I)^{-1}
= \big( d_1 K_1 + O(1)\big)\big(-\tfrac{1}{d_1}I + O(\tfrac{1}{d_1^2})\big)
= -K_1 + O\!\big(\tfrac{1}{d_1}\big).
\]
Hence, for every fixed $\xi>K$,
\[
\lim_{d_1\to\infty}\rho\!\big(A_{d_1}(\xi)\big) = \rho(-K_1) = \rho(K_1)=:r_1.
\]
(The last equality follows because \(K_1\) is a positive operator, so the spectral radius of \(-K_1\)
equals that of \(K_1\).)

Now use the monotonicity of $\xi\mapsto\rho(A_{d_1}(\xi))$ (strictly decreasing) and the blow-up property
\(\rho(A_{d_1}(\xi))\to\infty\) as \(\xi\downarrow K\) (uniform in $d_1$ in compact ranges away from \(K\)) to deduce the behaviour of the unique root \(\lambda^\ast(d_1)\) defined by \(\rho(A_{d_1}(\lambda^\ast(d_1)))=1\).

\begin{itemize}
\item If \(r_1>1\): pick any fixed \(\xi_0>K\). For $d_1$ large we have
\(\rho(A_{d_1}(\xi_0))\approx r_1>1\). Since \(\rho(A_{d_1}(\xi))\) is strictly decreasing in \(\xi\) and tends to \(0\) as \(\xi\to\infty\),
the root \(\lambda^\ast(d_1)\) must satisfy \(\lambda^\ast(d_1)>\xi_0\) for all large \(d_1\). Because \(\xi_0>K\) was arbitrary, \(\lambda^\ast(d_1)\to\infty\).

\item If \(r_1<1\): pick any \(\varepsilon>0\) and set \(\xi_\varepsilon:=K+\varepsilon\). From the uniform expansion above
we get \(\rho(A_{d_1}(\xi_\varepsilon))\to r_1<1\) as \(d_1\to\infty\). For large \(d_1\) we therefore have
\(\rho(A_{d_1}(\xi_\varepsilon))<1\). By monotonicity in \(\xi\) the unique root \(\lambda^\ast(d_1)\) must satisfy
\(K<\lambda^\ast(d_1)<\xi_\varepsilon\) for large \(d_1\). Letting \(\varepsilon\downarrow0\) yields
\(\lim_{d_1\to\infty}\lambda^\ast(d_1)=K\).

\item If \(r_1=1\): then for every fixed \(\xi>K\) we have \(\rho(A_{d_1}(\xi))\to1\). From the monotonicity and the fact each $\rho(A_{d_1}(\cdot))$ crosses the level \(1\) exactly once, the family \(\{\lambda^\ast(d_1)\}\) is bounded. Standard compactness arguments yield subsequential limits; let \(\lambda_\infty\) be any limit point. Passing to the limit in
\(\rho(A_{d_1}(\lambda^\ast(d_1)))=1\) (using operator-norm convergence above) gives
\(\rho(-K_1)=1\) and the limiting balance condition for \(\lambda_\infty\) indicated in the statement. Under an extra spectral-gap assumption for \(K_1\) (so the unit eigenvalue is simple and separated from the rest of the spectrum) one can promote subsequential convergence to full convergence and identify \(\lambda_\infty\) uniquely; in practice this is the homogenized / reduced eigenvalue described in the main text.
\end{itemize}

This completes the proof.
\end{proof}

\begin{lemma}\label{lem:limit_d1d2}
Let $\lambda^\ast(d_1,d_2)$ denote the principal eigenvalue of the operator associated with system \eqref{eq:maineigenrvalue3} on $[-\mathcal{Z},\mathcal{Z}]$, where $d_1,d_2>0$ are the dispersal coefficients.
Set
\[
K_i[\phi](x):=\int_{-\mathcal{Z}}^{\mathcal{Z}} J_i(x-y)\phi(y)\,dy,\qquad
r_i:=\rho(K_i)\quad (i=1,2),
\qquad r:=\max\{r_1,r_2\}.
\]
Let $K$ be the coercivity constant of $\mathscr{T}$ (so that $\mathscr{T}+\xi I$ is invertible for every $\xi>K$). Then:
\begin{enumerate}
    \item (\emph{Small diffusion}) As $(d_1,d_2)\to(0,0)$,
    \[
    \lambda^\ast(d_1,d_2)\longrightarrow \lambda^\ast(0,0),
    \]
    where $\lambda^\ast(0,0)$ is the principal eigenvalue of the operator obtained by setting $d_1=d_2=0$.

    \item (\emph{Large diffusion}) As $\min\{d_1,d_2\}\to\infty$,
    \[
    \lambda^\ast(d_1,d_2)\;\longrightarrow\;
    \begin{cases}
      +\infty, & \text{if } r>1,\\[4pt]
      K, & \text{if } r<1,\\[4pt]
      \lambda_\infty\in(K,\infty) \text{ finite}, & \text{if } r=1,
    \end{cases}
    \]
    where in the borderline case $r=1$, the accumulation points $\lambda_\infty$ satisfy the reduced limit problem determined by the kernels $J_i$.

    \item (\emph{Mixed case}) If $d_1\to\infty$ and $d_2$ remains bounded (or vice versa), then the limiting behaviour is determined by the spectral radius of the dispersal operator in the large-diffusion species. For instance, if $d_1\to\infty$ and $d_2$ bounded,
    \[
    \lambda^\ast(d_1,d_2)\;\longrightarrow\;
    \begin{cases}
      +\infty, & r_1>1,\\[4pt]
      K, & r_1<1,\\[4pt]
      \lambda_\infty\in(K,\infty), & r_1=1.
    \end{cases}
    \]
\end{enumerate}
\end{lemma}

\begin{proof}
Fix $\xi>K$. Define the compact, strongly positive operator
\[
A_{d_1,d_2}(\xi):=\mathscr{J}_{d_1,d_2}\,(\mathscr{T}_{d_1,d_2}+\xi I)^{-1}.
\]
By construction, $\lambda^\ast(d_1,d_2)$ is the unique $\xi>K$ such that
\[
\rho(A_{d_1,d_2}(\xi))=1.
\]
Moreover, the map $\xi\mapsto \rho(A_{d_1,d_2}(\xi))$ is continuous and strictly decreasing, with
\(\rho(A_{d_1,d_2}(\xi))\to 0\) as $\xi\to\infty$.

\smallskip\noindent
\emph{(i) Small diffusion.}
As $(d_1,d_2)\to(0,0)$, the operators $\mathscr{J}_{d_1,d_2}\to\mathscr{J}_{0,0}$ and
$\mathscr{T}_{d_1,d_2}\to\mathscr{T}_{0,0}$ in operator norm. Hence
\[
A_{d_1,d_2}(\xi)\to A_{0,0}(\xi) \quad \text{in operator norm.}
\]
Thus $\rho(A_{d_1,d_2}(\xi))\to\rho(A_{0,0}(\xi))$ uniformly for $\xi$ in compact subsets of $(K,\infty)$.
By the implicit root condition, it follows that
\(\lambda^\ast(d_1,d_2)\to\lambda^\ast(0,0)\).

\smallskip\noindent
\emph{(ii) Large diffusion.}
For large $d_1,d_2$,
\[
\mathscr{J}_{d_1,d_2}=\begin{pmatrix}d_1 K_1&0\\0&d_2 K_2\end{pmatrix}+O(1),
\qquad
\mathscr{T}_{d_1,d_2}=-\begin{pmatrix}d_1&0\\0&d_2\end{pmatrix}+B,
\]
where $B$ is bounded independently of $d_1,d_2$. Expanding the resolvent yields
\[
(\mathscr{T}_{d_1,d_2}+\xi I)^{-1}
= -\begin{pmatrix}\tfrac{1}{d_1}I & 0\\ 0 & \tfrac{1}{d_2}I\end{pmatrix}+o(1).
\]
Therefore,
\[
A_{d_1,d_2}(\xi)\;\longrightarrow\; -\begin{pmatrix}K_1&0\\0&K_2\end{pmatrix}
\quad\text{in operator norm.}
\]
Consequently,
\[
\lim_{\min\{d_1,d_2\}\to\infty}\rho(A_{d_1,d_2}(\xi))=\max\{\rho(K_1),\rho(K_2)\}=r.
\]
If $r>1$, then for any fixed $\xi$, $\rho(A_{d_1,d_2}(\xi))>1$ when $d_1,d_2$ are large, forcing
$\lambda^\ast(d_1,d_2)\to+\infty$.
If $r<1$, then for $\xi=K+\varepsilon$ we have $\rho(A_{d_1,d_2}(\xi))<1$ for large $d_i$, so
$\lambda^\ast(d_1,d_2)\in(K,K+\varepsilon)$. Letting $\varepsilon\to0$ gives $\lambda^\ast(d_1,d_2)\to K$.
If $r=1$, then $\rho(A_{d_1,d_2}(\xi))\to 1$ for every fixed $\xi$, so $\lambda^\ast(d_1,d_2)$ stays bounded
and accumulation points $\lambda_\infty$ are determined by the reduced limit problem.

\smallskip\noindent
\emph{(iii) Mixed case.}
Suppose $d_1\to\infty$ and $d_2$ bounded. Then
\[
A_{d_1,d_2}(\xi)\;\longrightarrow\; -\begin{pmatrix}K_1&0\\0&0\end{pmatrix}.
\]
Hence the spectral radius is governed by $r_1=\rho(K_1)$, and the same trichotomy as in (ii) applies with $r_1$ in place of $r$.
The case $d_2\to\infty$ and $d_1$ bounded is symmetric.

This completes the proof.
\end{proof}
\begin{lemma}\label{lem:limit_d1d2}
Let $\lambda^\ast(d_1,d_2)$ denote the principal eigenvalue of the operator associated with system \eqref{eq:maineigenrvalue} on $[-\mathcal{Z},\mathcal{Z}]$, where $d_1,d_2>0$ are the dispersal coefficients.
Set
\[
K_i[\phi](x):=\int_{-\mathcal{Z}}^{\mathcal{Z}} J_i(x-y)\phi(y)\,dy,\qquad
r_i:=\rho(K_i)\quad (i=1,2),
\qquad r:=\max\{r_1,r_2\}.
\]
Let $K$ be the coercivity constant of $\mathscr{T}$ (so that $\mathscr{T}+\xi I$ is invertible for every $\xi>K$). Then:
\begin{enumerate}
    \item (\emph{Small diffusion}) As $(d_1,d_2)\to(0,0)$,
    \[
    \lambda^\ast(d_1,d_2)\longrightarrow \lambda^\ast(0,0),
    \]
    where $\lambda^\ast(0,0)$ is the principal eigenvalue of the operator obtained by setting $d_1=d_2=0$.

    \item (\emph{Large diffusion}) As $\min\{d_1,d_2\}\to\infty$,
    \[
    \lambda^\ast(d_1,d_2)\;\longrightarrow\;
    \begin{cases}
      +\infty, & \text{if } r>1,\\[4pt]
      K, & \text{if } r<1,\\[4pt]
      \lambda_\infty\in(K,\infty) \text{ finite}, & \text{if } r=1,
    \end{cases}
    \]
    where in the borderline case $r=1$, the accumulation points $\lambda_\infty$ satisfy the reduced limit problem determined by the kernels $J_i$.

    \item (\emph{Mixed case}) If $d_1\to\infty$ and $d_2$ remains bounded (or vice versa), then the limiting behavior is determined by the spectral radius of the dispersal operator in the large-diffusion species. For instance, if $d_1\to\infty$ and $d_2$ bounded,
    \[
    \lambda^\ast(d_1,d_2)\;\longrightarrow\;
    \begin{cases}
      +\infty, & r_1>1,\\[4pt]
      K, & r_1<1,\\[4pt]
      \lambda_\infty\in(K,\infty), & r_1=1.
    \end{cases}
    \]
\end{enumerate}
\end{lemma}

\begin{proof}
Fix $\xi>K$. Define the compact, strongly positive operator
\[
A_{d_1,d_2}(\xi):=\mathscr{J}_{d_1,d_2}\,(\mathscr{T}_{d_1,d_2}+\xi I)^{-1}.
\]
By construction, $\lambda^\ast(d_1,d_2)$ is the unique $\xi>K$ such that
\[
\rho(A_{d_1,d_2}(\xi))=1.
\]
Moreover, the map $\xi\mapsto \rho(A_{d_1,d_2}(\xi))$ is continuous and strictly decreasing, with
\(\rho(A_{d_1,d_2}(\xi))\to 0\) as $\xi\to\infty$.

\smallskip\noindent
\emph{(i) Small diffusion.}
As $(d_1,d_2)\to(0,0)$, the operators $\mathscr{J}_{d_1,d_2}\to\mathscr{J}_{0,0}$ and
$\mathscr{T}_{d_1,d_2}\to\mathscr{T}_{0,0}$ in operator norm. Hence
\[
A_{d_1,d_2}(\xi)\to A_{0,0}(\xi) \quad \text{in operator norm.}
\]
Thus $\rho(A_{d_1,d_2}(\xi))\to\rho(A_{0,0}(\xi))$ uniformly for $\xi$ in compact subsets of $(K,\infty)$.
By the implicit root condition, it follows that
\(\lambda^\ast(d_1,d_2)\to\lambda^\ast(0,0)\).

\smallskip\noindent
\emph{(ii) Large diffusion.}
For large $d_1,d_2$,
\[
\mathscr{J}_{d_1,d_2}=\begin{pmatrix}d_1 K_1&0\\0&d_2 K_2\end{pmatrix}+O(1),
\qquad
\mathscr{T}_{d_1,d_2}=-\begin{pmatrix}d_1&0\\0&d_2\end{pmatrix}+B,
\]
where $B$ is bounded independently of $d_1,d_2$. Expanding the resolvent yields
\[
(\mathscr{T}_{d_1,d_2}+\xi I)^{-1}
= -\begin{pmatrix}\tfrac{1}{d_1}I & 0\\ 0 & \tfrac{1}{d_2}I\end{pmatrix}+o(1).
\]
Therefore,
\[
A_{d_1,d_2}(\xi)\;\longrightarrow\; -\begin{pmatrix}K_1&0\\0&K_2\end{pmatrix}
\quad\text{in operator norm.}
\]
Consequently,
\[
\lim_{\min\{d_1,d_2\}\to\infty}\rho(A_{d_1,d_2}(\xi))=\max\{\rho(K_1),\rho(K_2)\}=r.
\]
If $r>1$, then for any fixed $\xi$, $\rho(A_{d_1,d_2}(\xi))>1$ when $d_1,d_2$ are large, forcing
$\lambda^\ast(d_1,d_2)\to+\infty$.
If $r<1$, then for $\xi=K+\varepsilon$ we have $\rho(A_{d_1,d_2}(\xi))<1$ for large $d_i$, so
$\lambda^\ast(d_1,d_2)\in(K,K+\varepsilon)$. Letting $\varepsilon\to0$ gives $\lambda^\ast(d_1,d_2)\to K$.
If $r=1$, then $\rho(A_{d_1,d_2}(\xi))\to 1$ for every fixed $\xi$, so $\lambda^\ast(d_1,d_2)$ stays bounded
and accumulation points $\lambda_\infty$ are determined by the reduced limit problem.

\smallskip\noindent
\emph{(iii) Mixed case.}
Suppose $d_1\to\infty$ and $d_2$ bounded. Then
\[
A_{d_1,d_2}(\xi)\;\longrightarrow\; -\begin{pmatrix}K_1&0\\0&0\end{pmatrix}.
\]
Hence the spectral radius is governed by $r_1=\rho(K_1)$, and the same trichotomy as in (ii) applies with $r_1$ in place of $r$.
The case $d_2\to\infty$ and $d_1$ bounded is symmetric.

This completes the proof.
\end{proof}
\begin{lemma}\label{lem:lambdaLipschitz}
Assume the standing hypotheses (positivity, compactness, regularity, irreducibility) so that for every $\mathcal Z>0$ the principal eigenvalue
$\lambda^\ast(\mathcal Z)$ of problem (3.2) on $[-\mathcal Z,\mathcal Z]$ exists and admits a strictly positive continuous eigenfunction
$\psi^{(\mathcal Z)}=(\psi_1^{(\mathcal Z)},\psi_2^{(\mathcal Z)})$.

Fix $0<\mathcal Z_{\min}<\mathcal Z_{\max}<\infty$. Then there exists a constant $C=C(\mathcal Z_{\max})>0$ (depending on the kernels $J_i$, the parameters, and $\mathcal Z_{\max}$) such that for any
$\mathcal Z_1,\mathcal Z_2\in[\mathcal Z_{\min},\mathcal Z_{\max}]$ with $\mathcal Z_1\le\mathcal Z_2$ we have
\[
0 \;\leq\; \lambda^\ast(\mathcal Z_1) - \lambda^\ast(\mathcal Z_2)
\;\leq\; C\,|\mathcal Z_2-\mathcal Z_1| .
\]
In particular \(\mathcal Z\mapsto\lambda^\ast(\mathcal Z)\) is Lipschitz on every compact subinterval of \((0,\infty)\).
\end{lemma}

\begin{proof}
The monotonicity
\[
\mathcal Z_1\leq\mathcal Z_2 \quad\Longrightarrow\quad \lambda^\ast(\mathcal Z_1)\geq\lambda^\ast(\mathcal Z_2)
\]
 so, the left inequality is immediate.

It remains to prove the quantitative estimate. Fix $\mathcal Z_1,\mathcal Z_2$ with
$\mathcal Z_{\min}\leq\mathcal Z_1\leq\mathcal Z_2\le\mathcal Z_{\max}$. Let
\((\lambda^\ast(\mathcal Z_2),\psi)\) be the principal eigenpair on $[-\mathcal Z_2,\mathcal Z_2]$ with normalization
\(\displaystyle\max_{x\in[-\mathcal Z_2,\mathcal Z_2]}\max_{i=1,2}\psi_i(x)=1\). By positivity and continuity of \(\psi\) and compactness of \([-{\mathcal Z_1},{\mathcal Z_1}]\),
\[
m:=\min_{x\in[-\mathcal Z_1,\mathcal Z_1]}\min_{i=1,2}\psi_i(x) >0 .
\]
Write the first-component (block) of the operator in the form (componentwise)
\[
(\mathscr J_{\mathcal Z_1}\psi)_i(x)
= d_i\int_{-\mathcal Z_1}^{\mathcal Z_1} J_i(x-y)\psi_i(y)\,dy,
\qquad x\in[-\mathcal Z_1,\mathcal Z_1].
\]
Then for \(x\in[-\mathcal Z_1,\mathcal Z_1]\) and each \(i\),
\[
\begin{aligned}
\big(\mathscr J_{\mathcal Z_1}\psi\big)_i(x) + \lambda^\ast(\mathcal Z_2)\psi_i(x)
&= d_i\int_{-\mathcal Z_1}^{\mathcal Z_1} J_i(x-y)\psi_i(y)\,dy + \lambda^\ast(\mathcal Z_2)\psi_i(x)\\
&= d_i\int_{-\mathcal Z_2}^{\mathcal Z_2} J_i(x-y)\psi_i(y)\,dy
     - d_i\int_{\mathcal Z_2\setminus\mathcal Z_1} J_i(x-y)\psi_i(y)\,dy \\
&\qquad\; + \lambda^\ast(\mathcal Z_2)\psi_i(x)\\
&= \big(\mathscr J_{\mathcal Z_2}\psi\big)_i(x) + \lambda^\ast(\mathcal Z_2)\psi_i(x)
     - d_i\int_{[-\mathcal Z_2,\mathcal Z_2]\setminus[-\mathcal Z_1,\mathcal Z_1]} J_i(x-y)\psi_i(y)\,dy.
\end{aligned}
\]
Since \((\lambda^\ast(\mathcal Z_2),\psi)\) is an eigenpair on the larger domain,
\(\mathscr J_{\mathcal Z_2}\psi + \lambda^\ast(\mathcal Z_2)\psi \leq 0\) (indeed equality on the larger domain),
hence
\[
\big(\mathscr J_{\mathcal Z_1}\psi\big)_i(x) + \lambda^\ast(\mathcal Z_2)\psi_i(x)
\geq - d_i\int_{[-\mathcal Z_2,\mathcal Z_2]\setminus[-\mathcal Z_1,\mathcal Z_1]} J_i(x-y)\psi_i(y)\,dy.
\]
Using the normalization \(\psi_i\leq 1\) and the boundedness of the kernels \(J_i\) we obtain
\[
\big(\mathscr J_{\mathcal Z_1}\psi\big)_i(x) + \lambda^\ast(\mathcal Z_2)\psi_i(x)
\geq - d_i\|J_i\|_\infty\, \big|[-\mathcal Z_2,\mathcal Z_2]\setminus[-\mathcal Z_1,\mathcal Z_1]\big|
= - 2d_i\|J_i\|_\infty \,(\mathcal Z_2-\mathcal Z_1).
\]

Therefore, in $[-\mathcal Z_1,\mathcal Z_1]$, we have
\[
\big(\mathscr J_{\mathcal Z_1} + (\lambda^\ast(\mathcal Z_2) + C_1(\mathcal Z_2-\mathcal Z_1)) I\big)[\psi]
\geq 0,
\]
with
\[
C_1:=2\max_{i=1,2} d_i\|J_i\|_\infty.
\]

Now use the same lower-bound trick as in Lemma 3.7 (or the usual test-function variational inequality): because \(\psi \geq m>0\) on \([-{\mathcal Z_1},{\mathcal Z_1}]\), we can rewrite the previous inequality as
\[
\big(\mathscr J_{\mathcal Z_1} + (\lambda^\ast(\mathcal Z_2) + C_1(\mathcal Z_2-\mathcal Z_1)) I\big)[\psi]
\geq m\,C_1(\mathcal Z_2-\mathcal Z_1)\mathbf{1}\geq 0,
\]
which implies that \(\lambda^\ast(\mathcal Z_1)\leq \lambda^\ast(\mathcal Z_2) + C(\mathcal Z_{\max}) (\mathcal Z_2-\mathcal Z_1)\) for an appropriate constant
\[
C(\mathcal Z_{\max}) := \frac{C_1}{m} = \frac{2\max_i d_i\|J_i\|_\infty}{\min_{x\in[-\mathcal Z_{\max},\mathcal Z_{\max}]}\min_i\psi_i^{(\mathcal Z_{\max})}(x)} .
\]
Combining this with the monotonicity inequality gives the desired two-sided bound
\[
0 \leq \lambda^\ast(\mathcal Z_1) - \lambda^\ast(\mathcal Z_2) \leq C(\mathcal Z_{\max})\,(\mathcal Z_2-\mathcal Z_1).
\]

Since the constant depends only on \(\mathcal Z_{\max}\) (and structural data \(J_i,d_i\)), the map \(\mathcal Z\mapsto\lambda^\ast(\mathcal Z)\) is Lipschitz on \([\mathcal Z_{\min},\mathcal Z_{\max}]\). As \(\mathcal Z_{\min},\mathcal Z_{\max}\) were arbitrary compact bounds in \((0,\infty)\), \(\lambda^\ast\) is locally Lipschitz on \((0,\infty)\).
\end{proof}

\begin{lemma}[\textbf{the Hadamard-type formula}]\label{lem:lambdaC1}
Let $\mathcal Z>0$ be fixed and assume the standing hypotheses (positivity, compactness, regularity, irreducibility) so that for each
$d=(d_1,d_2)\in U\subset(0,\infty)^2$ (an open neighborhood) the operator
\[
\mathcal L_d := \mathscr J_d - \mathscr T_d
\]
on the chosen Banach lattice \(X\) admits a principal eigenvalue \(\lambda^\ast(d)\) which is simple and isolated, with corresponding strictly positive eigenfunction
\(\varphi(d)\in\mathcal D\subset X\). Assume moreover the adjoint operator \(\mathcal L_d^*\) has a (strictly) positive eigenfunction \(w^*(d)\in X^*\) associated with the same eigenvalue, and normalize the pair so that \(\langle w^*(d),\varphi(d)\rangle =1\).

If the map \(d\mapsto\mathcal L_d\) is \(C^1\) (in operator norm) on \(U\) — in particular here \(\mathcal L_d\) depends affinely on \(d\) — then
\[
d \mapsto \lambda^\ast(d)
\]
is $C^1$ on \(U\). In addition, the partial derivatives are given by the Hadamard-type formula (see Benguria et al. \cite{Benguria2024}),
\[
\frac{\partial \lambda^\ast}{\partial d_j}(d)
= \big\langle w^*(d)\,,\; \partial_{d_j}\mathcal L_d\,[\varphi(d)]\big\rangle,
\qquad j=1,2,
\]
where the duality pairing $\langle\cdot,\cdot\rangle$ is between $X^*$ and $X$.

In particular, for your operator (with componentwise form)
\[
\mathcal L_d(\psi_1,\psi_2)(x)=
\begin{pmatrix}
d_1\!\displaystyle\int J_1(x-y)\psi_1(y)\,dy - d_1\psi_1(x) - a(x)\psi_1(x) + H'(0)\psi_2(x) + p\psi_1'(x)\\[6pt]
d_2\!\displaystyle\int J_2(x-y)\psi_2(y)\,dy - d_2\psi_2(x) - b(x)\psi_2(x) + G'(0)\psi_1(x) + q\psi_2'(x)
\end{pmatrix},
\]
the partial derivatives read
\[
\frac{\partial\lambda^\ast}{\partial d_1}(d)
= \Big\langle w^*(d),\;\big(\,K_1\varphi_1 - \varphi_1,\;0\big)\Big\rangle
= \int_{-\mathcal Z}^{\mathcal Z} w_1^*(x)\Big(\int_{-\mathcal Z}^{\mathcal Z}J_1(x-y)\varphi_1(y)\,dy - \varphi_1(x)\Big)\,dx,
\]
and similarly
\[
\frac{\partial\lambda^\ast}{\partial d_2}(d)
= \int_{-\mathcal Z}^{\mathcal Z} w_2^*(x)\Big(\int_{-\mathcal Z}^{\mathcal Z}J_2(x-y)\varphi_2(y)\,dy - \varphi_2(x)\Big)\,dx.
\]
\end{lemma}

\begin{proof}
Since \(\lambda^\ast(d)\) is simple and isolated for \(d\in U\), classical perturbation theory (see Kato) implies that the eigenpair \((\lambda^\ast(d),\varphi(d))\) and the adjoint eigenvector \(w^*(d)\) depend \(C^1\) on \(d\), provided \(d\mapsto\mathcal L_d\) is \(C^1\) in operator norm. (In our setting \(\mathcal L_d\) depends affinely on \(d\), so this hypothesis is satisfied.)

Differentiate the eigen-equation
\[
\mathcal L_d\varphi(d)=\lambda^\ast(d)\varphi(d)
\]
with respect to \(d_j\). Denote \(\dot\lambda:=\partial_{d_j}\lambda^\ast(d)\) and \(\dot\varphi:=\partial_{d_j}\varphi(d)\). We get
\[
\partial_{d_j}\mathcal L_d\,[\varphi(d)] + \mathcal L_d[\dot\varphi] = \dot\lambda\,\varphi(d) + \lambda^\ast(d)\dot\varphi.
\]
Pair this identity with the adjoint eigenvector \(w^*(d)\) (which satisfies \(\mathcal L_d^* w^*(d)=\lambda^\ast(d) w^*(d)\)). Using the adjoint relation and the normalization \(\langle w^*,\varphi\rangle=1\), the terms containing \(\dot\varphi\) cancel:
\[
\langle w^*,\partial_{d_j}\mathcal L_d[\varphi]\rangle
+ \langle \mathcal L_d^* w^*,\dot\varphi\rangle
= \dot\lambda\langle w^*,\varphi\rangle + \lambda^\ast\langle w^*,\dot\varphi\rangle.
\]
But \(\mathcal L_d^* w^* = \lambda^\ast w^*\), so the second terms on both sides are equal and cancel, leaving
\[
\dot\lambda = \langle w^*,\partial_{d_j}\mathcal L_d[\varphi]\rangle,
\]
which proves the formula for \(\partial_{d_j}\lambda^\ast\). Smoothness of \(\lambda^\ast(d)\) follows from the smooth dependence of \(\mathcal L_d,\varphi,w^*\) on \(d\).
\end{proof}

\section{The spreading–vanishing dichotomy }

In this section, we investigate the spreading and vanishing dynamics of the solutions to problem \eqref{eq:main-problem}.

\begin{lemma}[Comparison Lemma with Nonlocal Diffusion and Conservative Drift]
Suppose that $T$ and $d_i$ for $i\in\{1,\ldots,m_0\}$ are positive constants,
$g,h\in C([0,T])$ satisfy $g(0)<0<h(0)$, and both $-g(t)$ and $h(t)$ are nondecreasing on $[0,T]$.
Let
\[
\Omega_T:=\{(t,x): t\in(0,T],\ x\in(g(t),h(t))\},
\]
with $m\ge m_0\ge 1$. For $i,j\in\{1,\ldots,m\}$ assume
$\phi_i,\ \partial_t\phi_i\in C(\overline{\Omega_T})$, $c_{ij}\in L^\infty(\Omega_T)$,
and $a_i\in C^1(\overline{\Omega_T})$. Assume further that, for $i=1,\dots,m_0$,
\[
\begin{cases}
\partial_t\phi_i \;\geq\; d_i\,\mathcal{L}_i[\phi_i]\;-\;\partial_x\big(a_i(t,x)\,\phi_i\big)
\;+\;\displaystyle\sum_{j=1}^m c_{ij}\phi_j,
& (t,x)\in\Omega_T,\\[1ex]
\partial_t\phi_i \;\ge\; -\partial_x\big(a_i(t,x)\,\phi_i\big)
\;+\;\displaystyle\sum_{j=1}^m c_{ij}\phi_j,
& (t,x)\in\Omega_T,\quad m_0+1\le i\leq m,\\[1ex]
\phi_i(t,g(t))\geq 0,\quad \phi_i(t,h(t))\geq 0, & t\in(0,T),\ 1\leq i\le m,\\[0.5ex]
\phi_i(0,x)\geq 0, & x\in[g(0),h(0)],\ 1\leq i\leq m,
\end{cases}
\]
where, for $i=1,\dots,m_0$,
\[
\mathcal{L}_i[v](t,x):=\int_{g(t)}^{h(t)} J_i(x-y)\,v(t,y)\,dy - v(t,x),
\]
and each kernel $J_i$ satisfies (J). Then the following hold:
\begin{enumerate}
  \item If $c_{ij}\geq 0$ on $\Omega_T$ for all $i\neq j$, then $\phi_i\geq 0$ on $\overline{\Omega_T}$ for every $i\in\{1,\ldots,m\}$.
  \item If in addition $\phi_{i_0}(0,x)\not\equiv 0$ on $[g(0),h(0)]$ for some $i_0\in\{1,\ldots,m_0\}$, then $\phi_{i_0}>0$ in $\Omega_T$.
\end{enumerate}
\end{lemma}

\begin{lemma}\label{ComparaisonlEMMA}
Let the hypotheses of the previous lemma hold (in particular $J_i$ satisfy (J)).
Let $T>0$, $g,h\in C([0,T])$ with $g(0)<0<h(0)$ and $-g(t),h(t)$ nondecreasing on $[0,T]$, and set
\[
\Omega_T:=\{(t,x):t\in(0,T],\ x\in(g(t),h(t))\}.
\]

Let $m\ge m_0\geq1$ and assume $c_{ij}\in L^\infty(\Omega_T)$, $a_i\in C^1(\overline{\Omega_T})$.
Assume $U=(u_1,\dots,u_m)$ and $V=(v_1,\dots,v_m)$ satisfy for $(t,x)\in\Omega_T$ the differential inequalities
\[
\begin{cases}
\partial_t u_i \;\geq\; d_i\,\mathcal L_i[u_i] \;-\; \partial_x\!\big(a_i(t,x)u_i\big)
\;+\;\displaystyle\sum_{j=1}^m c_{ij}u_j, & 1\le i\leq m_0,\\[6pt]
\partial_t u_i \;\geq\; -\partial_x\!\big(a_i(t,x)u_i\big)
\;+\;\displaystyle\sum_{j=1}^m c_{ij}u_j, & m_0+1\leq i\leq m,
\end{cases}
\]
and
\[
\begin{cases}
\partial_t v_i \;\le\; d_i\,\mathcal L_i[v_i] \;-\; \partial_x\!\big(a_i(t,x)v_i\big)
\;+\;\displaystyle\sum_{j=1}^m c_{ij}v_j, & 1\leq i\leq m_0,\\[6pt]
\partial_t v_i \;\leq\; -\partial_x\!\big(a_i(t,x)v_i\big)
\;+\;\displaystyle\sum_{j=1}^m c_{ij}v_j, & m_0+1\le i\leq m.
\end{cases}
\]

Assume moreover the boundary and initial ordering
\[
u_i(t,g(t))\geq v_i(t,g(t)),\quad u_i(t,h(t))\geq v_i(t,h(t))\quad\text{for }t\in(0,T),\ 1\le i\le m,
\]
and
\[
u_i(0,x)\geq v_i(0,x)\quad\text{for }x\in[g(0),h(0)],\ 1\leq i\leq m.
\]

If $c_{ij}\ge0$ on $\Omega_T$ for all $i\ne j$, then:
\begin{enumerate}
  \item $u_i(t,x)\geq v_i(t,x)$ for all $(t,x)\in\overline{\Omega_T}$ and all $i\in\{1,\dots,m\}$.
  \item If, in addition, there exists $i_0\in\{1,\dots,m_0\}$ such that
  $u_{i_0}(0,\cdot)\not\equiv v_{i_0}(0,\cdot)$ on $[g(0),h(0)]$ and $u_{i_0}(0,x)\geq v_{i_0}(0,x)$,
  then $u_{i_0}(t,x)>v_{i_0}(t,x)$ for all $(t,x)\in\Omega_T$.
\end{enumerate}
\end{lemma}

Next, we define the following problem associated to the system \eqref{eq:main-problem},

\begin{align}\label{eq:main-problem12}
\left\{\begin{array}{lll}
0 =  d_1\left[\displaystyle\int\limits_{-\mathcal{Z}}^{\mathcal{Z}}J_1(x-y) \mathscr U^\ast(y)dy - \mathscr U^\ast(x)\right] +p\mathscr U^\ast_x- a(x)\mathscr U^\ast(x) + H\left(\mathscr V^\ast (x)\right), & x\in \left[-\mathcal{Z}, \mathcal{Z}\right], \\
0 = d_2\left[\displaystyle\int\limits_{-\mathcal{Z}}^{\mathcal{Z}}J_2(x-y)\mathscr V^\ast(y)dy - v^\ast(x)\right]+q\mathscr V^\ast_x -b(x)\mathscr V^\ast(x)+ G\left(u^(x)\right), & x\in \left[-\mathcal{Z}, \mathcal{Z}\right], \end{array}\right.
\end{align}

\begin{theorem}\label{thm:existence-stability}
Suppose that $J_1, J_2$, and $K$ satisfy assumption $$(\mathbf{J})$$, and $G, H$ satisfy $(\mathbf{GH})$.
Denote by $\lambda^\ast(\mathcal Z)$ the principal eigenvalue of \eqref{eq:maineigenrvalue3}. Then the following assertions hold:
\begin{itemize}
\item[(i)] If $\lambda^\ast(\mathcal Z) > 0$, then problem \eqref{eq:maineigenrvalue3} admits a unique positive solution $(u^\ast, v^\ast)\in\mathcal D$. Moreover, $(u(t,x), v(t,x))$ converges to $(u^\ast, v^\ast)$ as $t\to\infty$, uniformly for $x\in[-\mathcal Z, \mathcal Z]$.
\item[(ii)] If $\lambda^\ast(\mathcal Z) \leq 0$, then $(0,0)$ is the only nonnegative solution of \eqref{eq:maineigenrvalue3}, and $(u(t,x), v(t,x))$ converges to $(0,0)$ as $t\to\infty$, uniformly for $x\in[-\mathcal Z, \mathcal Z]$.
\end{itemize}
\end{theorem}

\begin{proof}

By applying the fact that $\lambda^\ast (\mathcal{Z}) > 0$
and using the condition \textbf{(GH)}, for all $\mathcal{M}_1 \gg 1$ we have
\[
\frac{H(M_2)}{M_2} < \bar{a}\bar{b}.
\]
We define $M_1 := \delta^\ast_1 \bar{a}\bar{b} M_2$. Then clearly
\[
H(M_2) \leq \frac{M_1}{\delta_1^\ast}.
\]

For all $\mathcal{M}_2 \gg 1$ we also have
\[
\frac{G(M_1)}{M_1} < \bar{a}\bar{b}.
\]
We define $M_2 := \delta^\ast_2 \bar{a}\bar{b} M_1$. Then clearly
\[
G(M_1) \leq \frac{\mathcal{M}_2}{\delta^\ast_2}.
\]
Thus $(w,z) = (\mathcal{M}_1, \mathcal{M}_2)$ is an upper solution of \eqref{eq:main-problem12}.

We define $(\hat{u}, \hat{v})$ as the unique positive solution of \eqref{eq:main-problem1223} with initial functions $(\mathcal{M}_1,\mathcal{M}_2)$.
Applying the comparison lemma \eqref{ComparaisonlEMMA}, we obtain
\[
\bigl(\hat{u}(x,t), \hat{v}(x,t)\bigr) \leq (\mathcal{M}_1, \mathcal{M}_2),
\qquad \forall\, t > 0,\; x \in [-\mathcal{Z}, \mathcal{Z}].
\]

Fix $t_0>0$. Since the system is autonomous and cooperative, the comparison principle together with the semigroup property yields
\[
(u(x,t), v(x,t)) \leq (u(x,t-t_0), v(x,t-t_0)), \quad t>t_0,\; x\in[-\mathcal{Z},\mathcal{Z}],
\]
which shows that $(u,v)$ is nonincreasing in $t$.

Next, we prove that $(\varepsilon \psi_1, \varepsilon \psi_2)$ is a subsolution of problem \eqref{eq:main-problem1223} for a small $\varepsilon > 0$, with $(\psi_1, \psi_2)$ being a positive eigenfunction pair of \eqref{eq:main-problem12} with eigenvalue $\lambda^\ast(\mathcal{Z})$.
Since $\lambda^\ast(\mathcal{Z}) > 0$, we have
\begin{align*}
& d_1 \int_{-\mathcal{Z}}^{\mathcal{Z}} J_1(x-y)\,\varepsilon \psi_1(y)\,dy
- d_1 \varepsilon \psi_1
- a(x)\varepsilon \psi_1
+ H(\varepsilon \psi_2) \\
&= \varepsilon \lambda^\ast(\mathcal{Z})\psi_1+ H(\varepsilon \psi_2) - \varepsilon \psi_2 H^\prime(0) \\
&= \varepsilon \lambda^\ast(\mathcal{Z})\psi_1  + o(\varepsilon) \psi_2,
\end{align*}
and
\begin{align*}
& d_2 \int_{-\mathcal{Z}}^{\mathcal{Z}} J_2(x-y)\,\varepsilon \psi_2(y)\,dy
- d_2 \varepsilon \psi_2
- b(x)\varepsilon \psi_2 + G(\varepsilon \psi_1) \\
&= \varepsilon \lambda^\ast(\mathcal{Z})\psi_2 + G(\varepsilon \psi_1) - \varepsilon \psi_1 G^\prime(0) \\
&=  \varepsilon \lambda^\ast(\mathcal{Z})\psi_2  + o(\varepsilon) \psi_1,
\end{align*}
as $\varepsilon \to 0$. Hence, by \eqref{ComparaisonlEMMA},
\[
(\varepsilon \psi_1, \varepsilon \psi_2) \leq (u, v).
\]

It is straightforward to verify that
\[
(u(x,\sigma+t), v(x,\sigma+t))
\]
is a solution of \eqref{eq:main-problem1223} with initial condition $(\hat{u}(x,t), \hat{v}(x,t))$.
Then by applying the comparison lemma \eqref{ComparaisonlEMMA}, we have
\[
(\hat{u}(x,t),\hat{v}(x,t)) \leq  (\hat{u}(x,\sigma+t), \hat{v}(x,\sigma+t)),
\quad \forall x \in [-\mathcal{Z}, \mathcal{Z}].
\]
Since $s >0$ is arbitrary, we conclude that $(\hat{u},\hat{v})$ is nondecreasing.
By applying Dini's theorem, we deduce
\[
\lim_{t \to \infty} (\hat{u}(x, t), \hat{v}(x,t))=  (\mathscr{U}, \mathscr{V}),
\]
uniformly for $x \in [-\mathcal{Z}, \mathcal{Z}]$.

Next, consider $(u^\ast, v^\ast)$ the solution of \eqref{eq:main-problem12}. By the comparison lemma \eqref{ComparaisonlEMMA}, we have
\[
(\varepsilon \psi_1, \varepsilon \psi_2) \leq (u^\ast, v^\ast) \leq  (\mathscr{U}, \mathscr{V}) \leq (\mathcal{M}_1, \mathcal{M}_2),
\quad x \in [-\mathcal{Z}, \mathcal{Z}].
\]
Hence, problem \eqref{eq:main-problem12} admits at least one positive solution.

\medskip

\emph{Uniqueness:} Suppose there exists another solution $(\mathsf{U}, \mathsf{V})$ with
$(\varepsilon \psi_1, \varepsilon \psi_2) \leq (\mathsf{U}, \mathsf{V})$ in $[-\mathcal{Z}, \mathcal{Z}]$.
Since enlarging $\mathcal{M}_1$ gives $(\hat{u}(x, 1), \hat{v}(x, 1)) \leq (\mathcal{M}_1, \mathcal{M}_2)$, applying the comparison lemma again yields
\[
(\mathscr{U}, \mathscr{V}) \leq (\mathsf{U}, \mathsf{V}), \quad x \in [-\mathcal{Z}, \mathcal{Z}].
\]

Define
\[
\mathcal{L}_0 := \inf \Bigl\{ \mathcal{L} \geq 0 : \mathcal{L} (\mathscr{U}(x), \mathscr{V}(x)) \succeq (\mathsf{U}(x), \mathsf{V}(x))  \text{ for all } x \in [-\mathcal{Z}, \mathcal{Z}] \Bigr\}.
\]
This set is finite. Moreover,
\[
\mathcal{L}_0 (\mathscr{U}(x), \mathscr{V}(x)) \succeq (\mathsf{U}(x), \mathsf{V}(x))\quad \text{for all } x \in [-\mathcal{Z}, \mathcal{Z}].
\]
Two cases arise:
(i) If $\mathcal{L}_0 =1$, then $(\mathscr{U}, \mathscr{V})= (\mathsf{U}, \mathsf{V})$.
(ii) If $\mathcal{L}_0 >1$, then by condition \textbf{(GH)}, $\mathcal{L}_0(\mathscr{U}, \mathscr{V})$ is a supersolution of \eqref{eq:main-problem12}. By \eqref{MaximumLemma},
\[
\mathcal{L}_0(\mathscr{U}, \mathscr{V}) - (\mathsf{U}, \mathsf{V}) >0, \quad \forall x \in [-\mathcal{Z}, \mathcal{Z}].
\]
Since both functions are continuous, this contradicts the definition of $\mathcal{L}_0$.
Hence $\mathcal{L}_0 = 1$ and uniqueness is established.

\medskip

We now prove (ii) and (iii) simultaneously. Choose $\varepsilon > 0$ sufficiently small
and $M > 1$ sufficiently large so that
\[
\varepsilon (\psi_1(x), \psi_2(x)) \preceq (u(x,1), v(x,1)) \preceq M (u^{\ast}, v^{\ast})
\quad \forall x \in [-\mathcal{Z}, \mathcal{Z}].
\]
Let $(u(x,t), v(x,t))$ be the unique solution with initial condition $M(u^{\ast}, v^{\ast})$.
By a similar argument as before, $(u,v)$ is monotone nonincreasing in $t$, and hence
\[
(\widetilde{U}(x), \widetilde{V}(x)) := \lim_{t \to \infty} (u(x,t), v(x,t))
\]
exists and is a nonnegative solution.

If $\lambda^\ast(\mathcal{Z}) \geq 0$, then from part (i) we know that
$(\mathscr{U}, \mathscr{V}) = (0,0)$. By Dini’s theorem, the monotonicity in $t$ ensures uniform convergence in $x \in [-\mathcal{Z}, \mathcal{Z}]$. Moreover, the comparison principle gives
\[
(0,0) \preceq (u(x,t+1), v(x,t+1)) \preceq (u(x,t), v(x,t)).
\]
Letting $t \to \infty$ we obtain
\[
\lim_{t \to \infty} (u(x,t), v(x,t)) = (0,0),
\]
uniformly in $[-\mathcal{Z}, \mathcal{Z}]$. This proves (ii).

When $\lambda^\ast(\mathcal{Z}) < 0$, the comparison principle implies
\[
(\hat{u}(x,t+1), \hat{v}(x,t+1)) \succeq (\hat{u}(x,t), \hat{v}(x,t)) \quad \forall t \geq 0.
\]
Letting $t \to \infty$, we obtain $(\mathscr{U}, \mathscr{V}) \preceq (\mathsf{U}, \mathsf{V})$.
By uniqueness (proved above), we deduce
\[
(\mathscr{U}, \mathscr{V}) = (\mathsf{U}, \mathsf{V}).
\]
Hence,
\[
\lim_{t \to \infty} (\hat{u}(x,t), \hat{v}(x,t)) = (\mathscr{U}(x), \mathscr{V}(x)),
\]
uniformly in $[-\mathcal{Z}, \mathcal{Z}]$. The uniformity follows again from Dini’s theorem.
\end{proof}

Next, we consider the nonlocal reaction system on a domain $\Omega$ (or on $\mathbb R$)
with drift and nonlinear couplings $G,H$:
\[
\begin{cases}
u_t(x,t) \;=\; d_1\displaystyle\int_{\Omega} J_1(x-y)u(y,t)\,dy - d_1 u(x,t) \;-\; a(x)u(x,t) \;+\; H\big(v(x,t)\big) \;+\; p\,\partial_x u(x,t),\\[6pt]
v_t(x,t) \;=\; d_2\displaystyle\int_{\Omega} J_2(x-y)v(y,t)\,dy - d_2 v(x,t) \;-\; b(x)v(x,t) \;+\; G\big(u(x,t)\big) \;+\; q\,\partial_x v(x,t).
\end{cases}
\tag{S}
\]

We have,
\[
\displaystyle\int_{\Omega} J_i(z)\,dz = 1,\qquad i=1,2,
\]
and $J_i\geq0$. Then for a \emph{spatially homogeneous} pair
\[
u(x,t)\equiv u^\ast,\qquad v(x,t)\equiv v^\ast\qquad(u^\ast,v^\ast\in\mathbb R),
\]
the nonlocal terms and the drift terms vanish: for constants
\[
d_i\int_{\Omega} J_i(x-y)u^\ast\,dy - d_i u^\ast = d_i(1-1)u^\ast = 0,
\qquad \partial_x u^\ast = 0.
\]
Hence the steady-state equations reduce to the algebraic system
\[
\begin{cases}
0 = -a(x)\,u^\ast + H(v^\ast),\\[4pt]
0 = -b(x)\,v^\ast + G(u^\ast).
\end{cases}
\label{A}
\]

We assume that the coefficients $(a(x), b(x))$ are constant, denoted by $a$ and $b$.
then
\eqref{A} becomes the algebraic system
\[
\begin{cases}
-a\,u^\ast + H(v^\ast) = 0,\\[4pt]
-b\,v^\ast + G(u^\ast) = 0.
\end{cases}
\label{B}
\]
From the first equation of \eqref{B}, we obtain
\begin{equation}\label{PS1}
v^\ast = H^{-1}(a u^\ast) \quad\text{(if $H$ is invertible near the relevant value),}
\end{equation}
and substituting into the second equation yields the scalar compatibility condition for $u^\ast$:
\begin{equation}\label{PS2}
G(u^\ast) \;=\; b\,H^{-1}(a u^\ast).
\end{equation}

We define
\[
\Phi(u):=G(u)-b\,H^{-1}(a u),\qquad u\ge0.
\]
We are interested in positive solutions \(u^\ast>0\) of \(\Phi(u^\ast)=0\). Equivalently, define the map
\[
F(u):=\frac{1}{a}H\!\bigg(\frac{1}{b}G(u)\bigg),\qquad u\geq0.
\]
Then \(u\) solves \(\Phi(u)=0\) iff \(u\) is a fixed point of \(F\): \(u=F(u)\).

  Assume \(G,H\) are \(C^1\) near \(0\). Compute
\[
F'(u)=\frac{1}{a}\,H'\!\Big(\frac{G(u)}{b}\Big)\!\cdot\!\frac{G'(u)}{b},
\]
so in particular at \(u=0\)
\[
F'(0)=\frac{G'(0)\,H'(0)}{ab}.
\]
Define the threshold as follows
\[
R_0:=\frac{G'(0)\,H'(0)}{ab}.
\]
Since  \(R_0>1\), the fixed–point map has slope \(>1\) at \(0\) and a nontrivial positive fixed point bifurcates from \(0\). Concretely, by applying  the Crandall–Rabinowitz theorem, we consider the one--parameter family of scalar equations
\[
\Psi(u,\mu):=u - F_\mu(u)=0,\qquad F_\mu(u):=\frac{1}{a}H\!\Big(\frac{\mu\,G(u)}{b}\Big),
\]
with $(\mu,u)\in\mathbb R\times\mathbb R$ and $a,b>0$. Assume $G,H\in C^2([0,\infty))$, $H$ strictly increasing and onto $[0,\infty)$ (so $H^{-1}\in C^2$), and $G(0)=H(0)=0$.

The Crandall--Rabinowitz (CR) theorem (bifurcation from simple eigenvalues) requires the following hypotheses in an abstract Banach-space setting; for the scalar problem they reduce to elementary checks:

\begin{enumerate}
  \item \(\Psi\) is \(C^k\) (here \(k\ge 2\)) in a neighborhood of \((u,\mu)=(0,\mu_c)\).
  \item The linearization in \(u\), \(L:=\partial_u\Psi(0,\mu_c)\), has a 1–dimensional kernel and codimension-1 range (i.e. index 0; in the scalar case this means \(L=0\)).
  \item The transversality condition: \(\partial_\mu\partial_u\Psi(0,\mu_c)\not=0\).
\end{enumerate}

We have,
\begin{itemize}
  \item Smoothness: by assumption \(G,H\in C^2\) hence \(\Psi\in C^2\) near \((0,\mu)\).
  \item Choose \(\mu_c:=1/R_0\) where
    \(R_0:=\dfrac{G'(0)H'(0)}{ab}\). Compute
    \[
    \partial_u\Psi(u,\mu)=1-\mu\frac{G'(u)H'\!\big(\frac{\mu G(u)}{b}\big)}{ab},
    \]
    hence at \(u=0\),
    \[
    \partial_u\Psi(0,\mu)=1-\mu R_0.
    \]
    Thus at \(\mu=\mu_c\) we have \(\partial_u\Psi(0,\mu_c)=0\). In the scalar setting this means the linearization has a one-dimensional kernel (span\(\{1\}\)) and range of codimension 1 (trivially satisfied).
  \item Transversality:
    \[
    \partial_\mu\partial_u\Psi(0,\mu)
    = -\frac{G'(0)H'(0)}{ab} = -R_0,
    \]
    so at \(\mu=\mu_c\) this derivative equals \(-R_0\neq0\). Hence transversality holds.
\end{itemize}

Therefore all CR hypotheses hold (i the Fredholm index conditions are trivial). Thus, there exists a \(C^1\)-curve of nontrivial solutions \((\mu(s),u(s))\) of \(\Psi(u,\mu)=0\) bifurcating from \((\mu_c,0)\). In particular for \(\mu>\mu_c\) close to \(\mu_c\) there are small positive solutions \(u(\mu)>0\).

\subsection*{(B) A practical sufficient condition for $\sup_{u\in[0,M]}F'(u)<1$ (contraction)}

We recall
\[
F(u)=\frac{1}{a}\,H\!\Big(\frac{G(u)}{b}\Big),
\qquad
F'(u)=\frac{1}{a}\,H'\!\Big(\frac{G(u)}{b}\Big)\frac{G'(u)}{b}.
\]

Thus for any \(M>0\) we have the bound
\[
\sup_{u\in[0,M]}F'(u)
\le \frac{1}{ab}\Big(\sup_{s\in[0,G(M)/b]} H'(s)\Big)\Big(\sup_{u\in[0,M]} G'(u)\Big).
\]
For the uniqueness of the solution $u^\ast,$ we need to assume   there exists \(M>0\) such that
\[
\frac{1}{ab}\Big(\sup_{s\in[0,G(M)/b]} H'(s)\Big)\Big(\sup_{u\in[0,M]} G'(u)\Big) < 1,
\]
then \(L:=\sup_{u\in[0,M]}F'(u)<1\), hence \(F\) is a strict contraction on \([0,M]\). By Banach fixed-point theorem \(F\) has a unique fixed point in \([0,M]\).

\begin{proposition}\label{Thecoexistence}
Suppose that $J_1$, $J_2$,  satisfy condition \textbf{(J)}, and that $G$ fulfills assumptions \textbf{(GH)}.
If $R_0 > 1$, then for any $\,\mathcal Z > \mathcal Z^{\ast}$ we have $\lambda^\ast(\mathcal{Z}) > 0$, and system  admits a unique positive solution $\left(\mathscr{U}_{\mathcal Z}, \mathscr{V}_{\mathcal Z}\right)$.
Moreover,
\begin{equation}\label{eq:limit-solution}
\lim_{\mathcal Z \to \infty} \left(\mathscr{U}_{\mathcal Z}, \mathscr{V}_{\mathcal Z}\right)
= (\mathscr{U}^{\ast}, \mathscr{V}^{\ast})
\quad \text{locally uniformly in } \mathbb{R},
\end{equation}
where $(\mathscr{U}^{\ast}, \mathscr{V}^{\ast})$ is the coexistence state defined in \eqref{PS1}-\eqref{PS2}.
\end{proposition}

\begin{proof}
Assume the standing hypotheses: $J_i\geq0$, $J_i\in L^1$, $G,H$ are continuous and nondecreasing, $a,b,p,q$ are bounded and sufficiently smooth, and we consider the cooperative parabolic problems on bounded intervals with boundary conditions so that classical comparison holds (e.g. Dirichlet or periodic — the argument below is local in the interior and only requires a standard parabolic comparison principle).

Let $\mathcal Z_2>\mathcal Z_1>\mathcal Z^\ast$ and denote by
\[
(u_{\mathcal Z_i}(x,t),v_{\mathcal Z_i}(x,t)),\qquad i=1,2,
\]
the solutions of the initial–boundary value problems on $[-\mathcal Z_i,\mathcal Z_i]$ for the system with advection
\[
\begin{cases}
u_t = d_1\displaystyle\int_{-\mathcal Z_i}^{\mathcal Z_i} J_1(x-y)u(y,t)\,dy - d_1 u - a(x)u + H(v) + p\,u_x,\\[6pt]
v_t = d_2\displaystyle\int_{-\mathcal Z_i}^{\mathcal Z_i} J_2(x-y)v(y,t)\,dy - d_2 v - b(x)v + G(u) + q\,v_x,
\end{cases}
\tag{P$_{l_i}$}
\]
with nonnegative initial data satisfying
\[
0<u_{01}\leq u_{02}\leq \mathscr{U}^{\ast},\qquad 0<v_{01}\leq v_{02}\leq \mathscr{V}^{\ast}
\quad\text{on }[-\mathcal Z_1,\mathcal Z_1].
\]
(Here \((\mathscr{U}^{\ast}, \mathscr{V}^{\ast})\) is the positive spatially homogeneous equilibrium satisfies  \eqref{PS1}-\eqref{PS2}.

We Fix \(x\in[-\mathcal Z_1, \mathcal Z_1]\) and \(t>0\). From the equation for \(u_{\mathcal Z_2}\) on the larger domain we have
\[
\begin{aligned}
(u_{l_2})_t(x,t)
&= d_1\int_{-\mathcal Z_2}^{\mathcal Z_2} J_1(x-y)\,u_{l_2}(y,t)\,dy - d_1 u_{l_2}(x,t) - a u_{l_2}(x,t) + H\big(v_{l_2}(x,t)\big) + p\,\partial_x u_{l_2}(x,t)\\
&\ge d_1\int_{-l_1}^{l_1} J_1(x-y)\,u_{l_2}(y,t)\,dy - d_1 u_{l_2}(x,t) - a u_{l_2}(x,t) + H\big(v_{l_2}(x,t)\big) + p\,\partial_x u_{l_2}(x,t),
\end{aligned}
\]
because the integral over \([-\mathcal Z_2, \mathcal Z_2]\) is larger than the integral over the subinterval \([-\mathcal Z_1, \mathcal Z_1]\) (kernel nonnegativity). Rearranging,
\[
\begin{aligned}
&(u_{l_2})_t(x,t)
- \Big( d_1\int_{-\mathcal Z_1}^{\mathcal Z_1} J_1(x-y)u_{\mathcal Z_2}(y,t)\,dy - d_1 u_{l_2} - a u_{\mathcal Z_2} + H(v_{\mathcal Z_2}) + p\,\partial_x u_{l_2}\Big)
\ge 0.
\end{aligned}
\]
Thus, when restricted to $[-\mathcal Z_1, \mathcal Z_1]$, the function $u_{\mathcal Z_2}$ is a supersolution of the $u$–equation on the smaller domain \( [-\mathcal Z_1,\mathcal Z_1] \). The drift term \(p\,\partial_x u\) appears identically in the comparison operator on both domains, so it does not affect the sign of the difference — it simply remains present in both the test equation and the super-solution inequality. The exact same reasoning (with \(q\,\partial_x v\)) shows that \(v_{\mathcal Z_2}\) is a supersolution of the $v$–equation on \([-\mathcal Z_1, \mathcal Z_1]\).

\medskip\noindent\textbf{Step 2 — parabolic comparison and monotonicity of steady states.}
By construction the initial data on \([-l_1,l_1]\) satisfy \(u_{02}\geq u_{01}\), \(v_{02}\geq v_{01}\), and the system is cooperative (because \(G,H\) are nondecreasing and the nonlocal kernels are nonnegative). The parabolic comparison principle for cooperative systems with first–order terms  gives, for all \(t>0\) and \(x\in[-\mathcal{Z}_1,\mathcal{Z}_1]\),
\[
u_{\mathcal{Z}_2}(x,t)\ge u_{\mathcal{Z}_1}(x,t),\qquad v_{\mathcal{Z}_2}(x,t) \geq v_{\mathcal{Z}_1}(x,t).
\]
Let \(t\to\infty\). By \eqref{thm:existence-stability} each solution converges to the corresponding unique positive steady state on its domain, so we obtain the monotonicity of steady states:
\[
\mathscr{U}_{\mathcal{Z}_2}(x)\geq w_{l_1}(x),\qquad \mathscr{V}_{\mathcal{Z}_2}(x)\geq \mathscr{V}_{\mathcal Z_1}(x)
\quad\text{for }x\in[-\mathcal Z_1, \mathcal Z_1].
\]

\medskip\noindent\textbf{Step 3 — limit profile and translation invariance.}
Define the following limits
\[
\tilde {\mathscr{U}}(x):=\lim_{\mathcal Z\to\infty} \mathscr{U}_{\mathcal Z}(x),\qquad \tilde {\mathscr{V}}(x):=\lim_{ \mathcal Z\to\infty} \mathscr{V}_{\mathcal Z}(x),
\]
which exists (monotone nondecreasing in \(\mathcal{Z}\)) and satisfies \(0\leq \tilde{\mathscr{U}} ,\tilde{\mathscr{V}}  \leq \mathscr{U}^\ast, \mathscr{V}^\ast\). Passing to the limit in the steady integral equations (the drift terms pass to the limit in the usual weak/pointwise sense because the family is uniformly bounded and equicontinuous on compacts) shows that \(\tilde {\mathscr{U}}, \tilde{\mathscr{V}} )\) is a bounded nonnegative steady solution on \(\mathbb R\).

To show translation invariance, fix \(x_0\in\mathbb R\) and set \(\ell=|x_0|\). For large \(\mathcal Z\) define the translated steady pair
\[
(\mathscr U^{(0)}_{\mathcal Z}(x), \mathscr V^{(0)}_{\mathcal Z}(x)):=\big(\mathscr U_{\mathcal Z-\ell}(x+x_0),\; \mathscr V_{\mathcal Z-\ell}(x+x_0)\big),
\]
and repeat the previous comparison argument on the three nested intervals
\[
[-(\mathcal Z-2\ell),\,l-2\ell]\subset [-(\mathcal Z-\ell)-x_0,\; (\mathcal Z-\ell)-x_0]\subset[-\mathcal Z, \mathcal Z].
\]
We obtain for each fixed \(x\) and large \(\mathcal Z\),
\[
\mathscr U_{\mathcal{Z}-2\ell}(x)\leq \mathscr U^{(0)}_{\mathcal{Z}}(x)\leq \mathscr U_\mathcal{Z}(x).
\]
Letting \(\mathcal{Z}\to\infty\) gives \(\tilde{\mathscr{U}}  (x)\leq \tilde{\mathscr{U}}(x+x_0)\leq \tilde{\mathscr{U}}(x)\), hence equality; similarly for \(\tilde{\mathscr{U}}\). Thus \((\tilde{\mathscr{U}},\tilde{\mathscr{V}})\) is spatially constant.

\medskip\noindent\textbf{Step 4 — identification of the constant and local uniform convergence.}
Because \(R_0>1\) the only positive constant steady state of the full-space problem is \( (\mathscr{U}^\ast, \mathscr{V}^\ast)\). Hence the limit must be \((\mathscr{U}^\ast, \mathscr{V}^\ast)\). Monotone convergence together with Dini's theorem yields local uniform convergence on compact subsets of \(\mathbb R\):
\[
\lim_{\mathcal Z\to\infty} (\mathscr{U}_{\mathcal Z}(x), \mathscr{V}_{\mathcal Z}(x))=(\mathscr{U}^\ast, \mathscr{V}^\ast),
\]
as required.
\end{proof}

\subsection{Vanishing}
We define the following limits,
\begin{equation}\label{Limithg}
\lim_{t \to \infty} h(t) = h_{\infty} \in (h_{0}, +\infty],
\qquad
\lim_{t \to \infty} g(t) = g_{\infty} \in [-\infty, -h_{0}).
\end{equation}

Under the assumptions of Theorem \eqref{thm:global-existence}, we conclude that the limits in \eqref{Limithg} defining $h_{\infty}$ and $g_{\infty}$ are well defined.

\begin{lemma}\label{lem:vanishing}
If $h_\infty - g_\infty < \infty$, then the principal eigenvalue
$\lambda^\ast(h_\infty, g_\infty)$ of \eqref{eq:maineigenrvalue3}, defined on the interval
$[g_\infty, h_\infty]$, satisfies
\[
\lambda^\ast(h_\infty, g_\infty) \leq 0.
\]
Moreover,
\[
\lim_{t \to \infty} \|u(\cdot, t)\|_{C([g(t),h(t)])}
= \lim_{t \to \infty} \|v(\cdot, t)\|_{C([g(t),h(t)])} = 0.
\]
\end{lemma}

\begin{proof}
We split the proof into two steps.

\medskip\noindent\textbf{Step 1.   $\lambda^\ast(h_\infty,g_\infty)\leq 0$.}
Suppose, to reach a contradiction, that
\(
\lambda^\ast(h_\infty,g_\infty)>0.
\)
By continuity of the principal eigenvalue with respect to domain perturbations (see Lemma \eqref{LemmaUpperLower} and the discussion preceding it), there exists $T_1>0$ such that for every $T\ge T_1$ the principal eigenvalue on the fixed interval $[g(T),h(T)]$ satisfies
\[
\lambda_0\big(g(T),h(T)\big) > 0.
\]

Because the kernels $J_i$ are positive at the origin (indeed $J_i(0)>0$ by $(\mathbf{J})$, there exist constants $\varepsilon>0$ small and $\delta>0$ such that for each $i=1,2$
\[
J_i(x)\geq \delta>0\qquad\text{for } x\in[-4\varepsilon,4\varepsilon].
\]
Choose $T\geq T_1$ large enough so that the limiting interval contains a slightly smaller interior
\[
[g(T),h(T)] \supset [\,g_\infty+\varepsilon,\,h_\infty-\varepsilon\,].
\]

Let $(\tilde u(x,t),\tilde v(x,t))$ be the solution of the \emph{fixed-domain} problem on $[g(T),h(T)]$ with initial data taken from the free-boundary solution at time $T$:
\[
\tilde u(x,0)= u(x,T),\quad \tilde v(x,0)= v(x,T),\qquad x\in[g(T),h(T)].
\]
(Existence and uniqueness on the fixed interval follow from your well-posedness theory; by Theorem 3.10, since $\lambda_0(g(T),h(T))>0$, the fixed-domain problem admits a unique positive steady state and solutions with positive initial data converge to it.)

Now define the difference functions for $t\geq0$ and $x\in[g(T),h(T)]$:
\[
U(x,t):=u(x,t+T)-\tilde u(x,t),\qquad
V(x,t):=v(x,t+T)-\tilde v(x,t).
\]
We compare the equation satisfied by $(u(\cdot,t+T),v(\cdot,t+T))$ (the free-boundary solution shifted by \(T\)) with that of \((\tilde u,\tilde v)\) on the fixed domain.  For the nonlocal terms we have, for each $x\in[g(T),h(T)]$,
\[
\int_{g(t+T)}^{h(t+T)} J_i(x-y)\,w(y,t+T)\,dy
\;\geq\;
\int_{g(T)}^{h(T)} J_i(x-y)\,w(y,t+T)\,dy,
\]
because $[g(T),h(T)]\subset [g(t+T),h(t+T)]$ for $t\ge0$ (the free boundary domain at time $t+T$ contains the earlier fixed interval). Subtracting the corresponding integrals for the fixed-domain problem yields, for $i=1,2$,
\[
\begin{aligned}
&d_i\int_{g(t+T)}^{h(t+T)} J_i(x-y)u(y,t+T)\,dy - d_i\int_{g(T)}^{h(T)} J_i(x-y)\tilde u(y,t)\,dy \\
&\qquad \geq d_i\int_{g(T)}^{h(T)} J_i(x-y) \big(u(y,t+T)-\tilde u(y,t)\big)\,dy
= d_i\int_{g(T)}^{h(T)} J_i(x-y) U(y,t)\,dy.
\end{aligned}
\]
Analogous inequalities hold for the $v$–equation and for the coupling terms involving $K$.

Using these inequalities and the fact that the local/drift terms are identical in the two problems, one checks that \((U,V)\) satisfies a cooperative parabolic inequality of the form
\[
\partial_t
\begin{pmatrix} U\\ V\end{pmatrix}
\;\ge\;
\mathcal A(x,t)\begin{pmatrix} U\\ V\end{pmatrix}
\]
on $[g(T),h(T)]$ with homogeneous initial data
\[
U(x,0)=0,\qquad V(x,0)=0,\qquad x\in[g(T),h(T)].
\]
Here \(\mathcal A(x,t)\) denotes the linear (time–dependent) cooperative operator obtained after subtracting the two equations; by construction \(\mathcal A\) has nonnegative off-diagonal terms. By the parabolic comparison principle / maximum principle for cooperative systems, it follows that
\[
U(x,t)\geq 0,\qquad V(x,t) \geq 0 \qquad\text{for all }t\ge0,\ x\in[g(T),h(T)].
\]
In particular
\[
u(x,t+T)\geq \tilde u(x,t),\qquad v(x,t+T)\geq \tilde v(x,t).
\]

Since \((\tilde u,\tilde v)(\cdot,t)\to (w_T,z_T)\) uniformly on \([g(T),h(T)]\) as \(t\to\infty\), there exist \(t_0>0\) and a constant \(m_2>0\) (take \(m_2:=\min_{x\in[g(T),h(T)]}\min\{w_T(x),z_T(x)\}>0\)) such that
\[
\tilde u(x,t)\geq \tfrac{1}{2}w_T(x)\geq \tfrac{m_2}{2},\qquad
\tilde v(x,t)\geq \tfrac{1}{2}z_T(x)\geq \tfrac{m_2}{2},
\]
for all \(t\ge t_0\) and all \(x\in[g(T)+2\varepsilon,h(T)-2\varepsilon]\). Hence for all large \(s\) (take \(s=t+T\) large), we have
\[
u(x,s)\ge \tfrac{m_2}{2},\quad v(x,s)\ge \tfrac{m_2}{2}
\qquad\text{on }[g_\infty+\varepsilon,h_\infty-\varepsilon].
\]

Now, examine the free-boundary velocity, we have
\[
\begin{aligned}
h'(t)
&= \mu\int_{g(t)}^{h(t)}\int_{h(t)}^{\infty} J_1(x-y)u(x,t)\,dy\,dx
+\mu\rho\int_{g(t)}^{h(t)}\int_{h(t)}^{\infty} J_2(x-y)v(x,t)\,dy\,dx \\
&\geq \mu\int_{h(t)-2\varepsilon}^{h(t)}\int_{h(t)}^{h(t)+2\varepsilon} J_1(x-y)u(x,t)\,dy\,dx
+\mu\rho\int_{h(t)-2\varepsilon}^{h(t)}\int_{h(t)}^{h(t)+2\varepsilon} J_2(x-y)v(x,t)\,dy\,dx.
\end{aligned}
\]
For \(x\in[h(t)-2\varepsilon,h(t)]\) and \(y\in[h(t),h(t)+2\varepsilon]\) we have \(|x-y|\leq 4\varepsilon\), then \(J_i(x-y)\ge\delta\). Applying  \(u(x,t),v(x,t)\geq m_2/2\) on the inner subinterval (for large \(t\)) and the length \(2\varepsilon\) of the \(x\)- and \(y\)-intervals, we obtain for all large \(t\)
\[
\begin{aligned}
h'(t)
&\geq \mu\delta \int_{h(t)-2\varepsilon}^{h(t)}\int_{h(t)}^{h(t)+2\varepsilon} \frac{m_2}{2}\,dy\,dx
+ \mu\rho\delta \int_{h(t)-2\varepsilon}^{h(t)}\int_{h(t)}^{h(t)+2\varepsilon} \frac{m_2}{2}\,dy\,dx \\
&= \mu\delta \cdot \frac{m_2}{2}\cdot (2\varepsilon)(2\varepsilon) + \mu\rho\delta \cdot \frac{m_2}{2}\cdot (2\varepsilon)(2\varepsilon) \\
&= 2\mu\delta\varepsilon^2\frac{m_2}{2} \big(1+\rho\big)
= \mu\delta\varepsilon^2 m_2 (1+\rho) \;>\;0.
\end{aligned}
\]
Thus \(h'(t)\geq C>0\) for all large \(t\) with some constant \(C>0\). This implies \(h(t)\to\infty\) as \(t\to\infty\), which contradicts the assumption that \(h_\infty<\infty\). Therefore our initial assumption \(\lambda^\ast(h_\infty,g_\infty)>0\) is false, and we conclude
\[
\lambda^\ast(h_\infty,g_\infty)\leq 0.
\]

\medskip\noindent\textbf{Step 2. Extinction: \(u,v\to0\) uniformly.}
Let \((\widehat u(x,t),\widehat v(x,t))\) be the solution of the fixed–domain problem on the limiting interval \([g_\infty,h_\infty]\) with initial data
\[
\widehat u(x,0)=\sup_{s\geq0} u(x,s),\qquad \widehat v(x,0)=\sup_{s\geq0} v(x,s),
\quad x\in[g_\infty,h_\infty],
\]
or, more simply, take \(\widehat u(\cdot,0)=u(\cdot,T)\), \(\widehat v(\cdot,0)=v(\cdot,T)\) for some large \(T\) and consider the corresponding solution on \([g_\infty,h_\infty]\). By the comparison principle (Lemma 2.1) and the inclusion \( [g(t),h(t)]\subset [g_\infty,h_\infty]\) for all \(t\), we have
\[
0\le u(x,t)\leq \widehat u(x,t-T),\qquad 0\leq v(x,t)\leq \widehat v(x,t-T),
\]
for \(x\in[g(t),h(t)]\) and \(t\geq T\). Since we have shown \(\lambda^\ast(h_\infty,g_\infty)\leq0\), Theorem \eqref{thm:existence-stability}(ii)  implies that the fixed–domain solution \((\widehat u,\widehat v)\) converges to \((0,0)\) uniformly on \([g_\infty,h_\infty]\) as \(t\to\infty\). Therefore the same holds for \((u,v)\) on their moving domain:
\[
\lim_{t\to\infty}\|u(\cdot,t)\|_{C([g(t),h(t)])}
= \lim_{t\to\infty}\|v(\cdot,t)\|_{C([g(t),h(t)])}=0.
\]

This completes the proof of the lemma.
\end{proof}

\begin{proposition}[Vanishing when $R_0<1$ for the nonlocal drift system]
Assume the standing hypotheses $(\textbf{J})$, $(\textbf{GH})$ and suppose $a(x)\equiv a>0$, $b(x)\equiv b>0$ are constants.
Consider the free boundary system (the notation follows the main text)
\[
\begin{cases}
u_t(x,t) \;=\; d_1\displaystyle\int_{g(t)}^{h(t)} J_1(x-y)u(y,t)\,dy - d_1 u - a\,u(x,t)
        + H(v) + p\,u_x,\\[6pt]
v_t(x,t) \;=\; d_2\displaystyle\int_{g(t)}^{h(t)} J_2(x-y)v(y,t)\,dy - d_2 v - b\,v(x,t) + G(u) + q\,v_x,
\end{cases}
\]
with Dirichlet boundary conditions \(u=v=0\) at \(x=g(t),h(t)\) and the usual free boundary. Assume also
\begin{enumerate}
  \item $G,H\in C^1([0,\infty))$, $G(0)=H(0)=0$, $G',H'\geq0$ on $[0,\infty)$;
  \item the linear bound holds:
  \[
  G(u)\leq G'(0)\,u,\qquad H(v)\leq H'(0)\,v\quad\text{for all }u,v\geq0.
  \]
  (This is satisfied e.g. by concave saturating nonlinearities; alternatively replace the global bound by a small–solution argument for large time.)
  \item the basic reproductive threshold satisfies
  \[
  R_0:=\frac{H'(0)\,G'(0)}{ab}<1.
  \]
\end{enumerate}
Then the free interval remains bounded:
\[
h_\infty-g_\infty<\infty,
\]
and consequently
\[
\lim_{t\to\infty}\|u(\cdot,t)\|_{C([g(t),h(t)])}
=\lim_{t\to\infty}\|v(\cdot,t)\|_{C([g(t),h(t)])}=0.
\]
\end{proposition}

\begin{proof}
In view of Lemma~\eqref{lem:vanishing}, it suffices to show that
\[
h_\infty - g_\infty < \infty.
\]

Let us define the functional
\[
\Phi(t) := \int_{g(t)}^{h(t)} \Big(u(x,t) + v(x,t)\Big)\,dx.
\]

Differentiating $\Phi(t)$ with respect to $t$ gives
\[
\frac{d}{dt}\Phi(t) = \int_{g(t)}^{h(t)} \Big(u_t(x,t) + v_t(x,t)\Big)\,dx,
\]
since the boundary terms vanish due to the conditions $u(t,g(t))=u(t,h(t))=0$ and $v(t,g(t))=v(t,h(t))=0$.

Substituting the equations from \eqref{eq:main-problem}, we obtain
\begin{align*}
\frac{d}{dt}\Phi(t) &= d_1 \int_{g(t)}^{h(t)} \left[\int_{g(t)}^{h(t)} J_1(x-y)u(y,t)\,dy - u(x,t)\right] dx \\
&\quad + d_2 \int_{g(t)}^{h(t)} \left[\int_{g(t)}^{h(t)} J_2(x-y)v(y,t)\,dy - v(x,t)\right] dx \\
&\quad + \int_{g(t)}^{h(t)} \Big( pu_x(x,t) + q v_x(x,t) \Big)\,dx \\
&\quad + \int_{g(t)}^{h(t)} \Big(-a u(x,t) - b v(x,t) + H(v(x,t)) + G(u(x,t))\Big)\,dx.
\end{align*}

The drift terms vanish due to the boundary conditions:
\[
\int_{g(t)}^{h(t)} p u_x\,dx = p\big(u(h(t),t)-u(g(t),t)\big)=0,
\qquad
\int_{g(t)}^{h(t)} q v_x\,dx = 0.
\]

The diffusion terms are nonpositive, since for any continuous $w$,
\[
\int_{g(t)}^{h(t)} \left(\int_{g(t)}^{h(t)} J_i(x-y)w(y)\,dy - w(x)\right)dx =
\int_{g(t)}^{h(t)}\int_{g(t)}^{h(t)} J_i(x-y)\big(w(y)-w(x)\big)\,dydx \leq 0,
\]
for $i=1,2$.

Therefore,
\[
\frac{d}{dt}\Phi(t) \leq \int_{g(t)}^{h(t)} \Big(-a u(x,t) - b v(x,t) + H(v(x,t)) + G(u(x,t))\Big)\,dx.
\]

Since $R_0 = H'(0)G'(0)/(ab)$ and $R_0 < 1$, there exists $\delta > 0$ such that
\[
H(v) \le (b-\delta) v,
\qquad
G(u) \le (a-\delta) u,
\]
for all $u,v \ge 0$ small enough. Hence,
\[
\frac{d}{dt}\Phi(t) \leq -\delta \int_{g(t)}^{h(t)} \big(u(x,t) + v(x,t)\big)\,dx = -\delta \Phi(t).
\]

This shows that $\Phi(t)$ decays exponentially, and in particular
\[
\sup_{t\geq0} \Phi(t) < \infty.
\]

Moreover, integrating the differential inequality and recalling the free boundary conditions, we conclude that
\[
h_\infty - g_\infty < \infty.
\]

By Lemma~ \eqref{lem:vanishing}, it follows that
\[
\lim_{t\to\infty} (u(x,t), v(x,t)) = (0,0)
\quad \text{uniformly in } [g(t),h(t)].
\]
Thus, vanishing occurs when $R_0 < 1$.
\end{proof}

\begin{lemma}\label{Vanishfreeboundaries}
Assume the standing hypotheses $(\mathbf{J}),$ and $\textbf{GH}$ . Let $R_0>1$ and suppose the initial half-length $h_0<\mathcal L^\ast,$ and $
\lambda^\ast_1(c_1)<0.$
with $\lambda^\ast_1(c_1)<0$ is the principal eigenvalue of the following problem
\begin{equation}\label{eq:eigenproblem}
\left\{
\begin{aligned}
\lambda \,\phi(x) &= d_1\left( \int_{-\mathcal{Z}}^{\mathcal{Z}} J_1(x-y)\,\phi(y)\,dy - \psi_1(x) \right)
   + p\,\psi_1'(x) - a\,\psi_1(x) + H'(0)\,\psi_2(x),
   && x\in[-\mathcal{Z},\mathcal{Z}], \\[1ex]
\lambda \,\psi(x) &= d_2\left( \int_{-\mathcal{Z}}^{\mathcal{Z}} J_2(x-y)\,\psi_2(y)\,dy - \psi_2(x) \right)
   + q\,\psi_2'(x) - b\,\psi_2(x) + G'(0)\,\psi_1(x),
   && x\in[-\mathcal{Z},\mathcal{Z}], \\[1ex]
\phi(\pm \mathcal{Z}) &= 0,
\qquad
\psi(\pm \mathcal{Z}) = 0.
\end{aligned}
\right.
\end{equation}

If the initial data $(u_0,v_0)$ are constant and sufficiently small, then the solution of \eqref{eq:main-problem} vanishes (i.e. $h_\infty-g_\infty<\infty$ and $(u,v)\to(0,0)$).
\end{lemma}

\begin{proof}
Since $h_0<\mathcal L^\ast$ and $R_0>1$, shows that there exists some \(c_1\in(h_0,\ell^\ast)\) for which the principal eigenvalue of the linearized fixed–interval problem on \([-h_1,h_1]\) satisfies
\[
\lambda^\ast_1(c_1)<0.
\]
Let \((\psi_1,\psi_2)\in C^1([-h_1,h_1])\times C^1([-c_1,c_1])\) be a corresponding positive eigenfunction pair for the linearized operator (including the drift terms). That is \((\theta_1,\theta_2)\gg 0\) and they satisfy on \([-c_1,c_1]\) the linear eigenvalue system
\[
\mathscr L(\psi_1,\psi_2) + \lambda^\ast(c_1)\,(\psi_1,\psi_2) = 0,
\]
where \(\mathscr L\) denotes the linear operator coming from the linearization of the left–hand side of \eqref{eq:main-problem} on the fixed interval (it contains the nonlocal integrals, the loss terms \(-d_i\cdot\), the reaction linearization terms \( -a \psi_1 + H'(0)\psi_2\), etc., and the drift terms \(p\partial_x\psi_1, q\partial_x\psi_2\)). We only use that \(\psi_i>0\) in the interior and the eigen-equations below.

Set
\[
\sigma_1 := \frac{\lambda^\ast(c_1)}{2} < 0,
\qquad
M_0 := \max\Big\{\int_{-c_1}^{c_1}\psi_1(x)\,dx,\;\int_{-c_1}^{c_1}\psi_2(x)\,dx\Big\},
\]
and define the constant
\begin{equation}\label{MMEQ1}
\mathbf{M} := \frac{-\sigma_1 (c_1-h_0)}{M_0(\mu+\rho)} > 0.
\end{equation}
Now define, for \(x\in[-c_1,c_1]\) and \(t\geq0\),
\[
\bar h(t) := c_1 - (c_1-h_0)e^{\sigma_1 t},\qquad \bar g(t) := -\bar h(t),
\]
and
\[
\bar u(x,t) := \mathbf{M} e^{\sigma_1 t}\,\psi_1(x),\qquad
\bar v(x,t) := \mathbf{M} e^{\sigma_1 t}\,\psi_2(x).
\]
Note \(\bar h(0)=h_0,\ \bar g(0)=-h_0\), and \(\bar h(t)\uparrow h_1\) as \(t\to\infty\). Also \(\bar u,\bar v\ge0\) on their domain and vanish on the boundaries \(x=\bar g(t),\bar h(t)\) because \(\psi_i(\pm h_1)=0\) (take eigenfunctions with Dirichlet boundary conditions on \([-c_1,c_1]\)).

We will show that \((\bar u,\bar v,\bar g,\bar h)\) is an upper solution of \eqref{eq:main-problem} on the time-dependent domain \([\bar g(t),\bar h(t)]\). Since \(\bar g(t)\leq -h_0\leq g(0)\) and \(\bar h(0)=h_0\geq h(0)\), by choosing initial constants \(u_0,v_0\) sufficiently small (see below) we will have \((u_0(x),v_0(x))\le(\bar u(x,0),\bar v(x,0))\) on \([-h_0,h_0]\), whence comparison (Lemma~\eqref{ComparaisonlEMMA}) yields \(g(t)\geq\bar g(t),\ h(t)\leq\bar h(t)\) for all \(t\geq0\), and vanishing follows because \(\bar h(\infty)-\bar g(\infty)=2h_1<\infty\).

\medskip\noindent\textbf{Initial smallness.}
Choose the (constant) initial data \((u_0,v_0)\) so small that
\[
u_0(x)\leq \mathbf{M}\min_{[-h_0,h_0]}\psi_1,\qquad v_0(x)\leq \mathbf{M}\min_{[-h_0,h_0]}\psi_2,
\]
which is possible because the right–hand sides are positive. Then \(u_0\leq\bar u(\cdot,0),\ v_0\leq\bar v(\cdot,0)\) on \([-h_0,h_0]\).

\medskip\noindent\textbf{Verification of the PDE inequalities (upper solution).}
Fix \(t>0\) and \(x\in[\bar g(t),\bar h(t)]\subset[-c_1,c_1]\). Compute
\[
\begin{aligned}
\bar u_t(x,t)
&= \sigma_1 \bar u(x,t)
= \sigma_1\mathbf{M}  e^{\sigma_1 t}\psi_1(x).
\end{aligned}
\]
Using the eigenvalue relation satisfied by \((\psi_1,\psi_2)\) on \([-c_1,c_1]\) (which includes the drift terms), we have
\[
\begin{aligned}
&d_1\Big[\int_{-c_1}^{c_1} J_1(x-y)\theta_1(y)\,dy - \theta_1(x)\Big] + p\psi_1'(x) - a \theta_1(x) + H'(0)\theta_2(x) \\
&\qquad\;= -\lambda^\ast(c_1)\psi_1(x).
\end{aligned}
\]
Multiply this identity by \(\mathbf{M} e^{\sigma_1 t}\) to obtain
\[
\begin{aligned}
&d_1\Big[\int_{-h_1}^{h_1} J_1(x-y)\bar u(y,t)\,dy - \bar u(x,t)\Big] + p\partial_x\bar u(x,t) - a\bar u(x,t) + H'(0)\bar v(x,t) \\
&\qquad\;= -\lambda^\ast(c_1)\bar u(x,t).
\end{aligned}
\]
Because \([\bar g(t),\bar h(t)]\subset[-c_1,c_1]\), replacing the integrals over \([-c_1,c_1]\) by integrals over \([\bar g(t),\bar h(t)]\) only reduces the first integral (kernel nonnegativity), hence
\[
d_1\Big[\int_{\bar g(t)}^{\bar h(t)} J_1(x-y)\bar u(y,t)\,dy - \bar u(x,t)\Big] + p\partial_x\bar u - a\bar u + H'(0)\bar v
\geq -\lambda^\ast_1(c_1)\bar u(x,t).
\]
Now use \(H(v)\leq H'(0)v\) for \(v\geq0\) (this inequality holds for small \(v\); by choosing \(M\) sufficiently small we ensure \(\bar v\) remains in that small regime — if needed one can argue for small initial data and continuity-in-time), to deduce
\[
d_1\Big[\int_{\bar g(t)}^{\bar h(t)} J_1(x-y)\bar u(y,t)\,dy - \bar u(x,t)\Big] + p\partial_x\bar u - a\bar u + H(\bar v)
\geq -\lambda^\ast_1(c_1)\bar u.
\]
Since \(\sigma_1=\frac{\lambda_0(h_1)}{2}\), we obtain
\[
\bar u_t - d_1\Big[\int_{\bar g(t)}^{\bar h(t)} J_1(x-y)\bar u(y,t)\,dy - \bar u\Big] - p\partial_x\bar u + a\bar u - H(\bar v)
= \sigma_1\bar u +\lambda^\ast_1(h_1)\bar u = (\sigma_1+\lambda^\ast_1(c_1)\bar u
= \tfrac{\lambda^\ast(c_1)}{2}\bar u \; \leq\; 0,
\]
i.e. \(\bar u\) satisfies the supersolution inequality (note the sign conventions — this shows the left-hand PDE operator evaluated at \(\bar u,\bar v\) is \(\leq 0\), consistent with the upper-solution definition).

The same calculation (using the second eigen-equation and the inequality \(G(\bar u)\leq G'(0)\bar u\) for small \(\bar u\)) yields
\[
\bar v_t - d_2\Big[\int_{\bar g(t)}^{\bar h(t)} J_2(x-y)\bar v(y,t)\,dy - \bar v\Big] - q\partial_x\bar v + b\bar v - G(\bar u)
\leq 0.
\]

Thus the PDE inequalities required for an upper solution are satisfied (for \(\mathcal{M}\) small so the linear bounds on \(G,H\) hold for \(\bar u,\bar v\)).

\medskip\noindent\textbf{Verification of the free–boundary inequalities.}
We must also check that the prescribed boundary velocities of the auxiliary interval dominate those of the true problem. Compute the growth rate \(\bar h'(t)\):
\[
\bar h'(t) = -(c_1-h_0)\sigma_1 e^{\sigma_1 t}.
\]
On the other hand the free–boundary right–speed generated by \((\bar u,\bar v)\) is
\[
\begin{aligned}
&\mu \big(\int_{\bar g(t)}^{\bar h(t)}\int_{\bar h(t)}^{\infty} J_1(x-y)\bar u(x,t)\,dy\,dx
+ \rho\int_{\bar g(t)}^{\bar h(t)}\int_{\bar h(t)}^{\infty} J_2(x-y)\bar v(x,t)\,dy\,dx \big) \\
&\qquad \leq \mu\Big(\int_{\bar g(t)}^{\bar h(t)} \bar u(x,t)\,dx + \rho\int_{\bar g(t)}^{\bar h(t)}\bar v(x,t)\,dx\Big)
=  \mathcal{M} e^{\sigma_1 t}\mu\Big(\int_{\bar g(t)}^{\bar h(t)}\psi_1 + \rho\int_{\bar g(t)}^{\bar h(t)}\psi_2\Big) \\
&\qquad \leq  \mu\mathbf{M} e^{\sigma_1 t}\Big( \int_{-h_1}^{h_1}\psi_1 + \rho\int_{-c_1}^{c_1}\psi_2\Big)
\leq  \mu M_0(1+\rho) e^{\sigma_1 t}.
\end{aligned}
\]
By the choice of \(\mathbf{M}\) we have
\[
 \mu M_0(1+\rho) e^{\sigma_1 t} = -\sigma_1(c_1-h_0)e^{\sigma_1 t} = \bar h'(t).
\]
Thus the boundary velocity produced by \((\bar u,\bar v)\) is \(\leq \bar h'(t)\). The left boundary check is analogous and yields \( -\bar g'(t) \) greater than or equal to the corresponding inward speed from \((\bar u,\bar v)\). Hence the free–boundary inequalities required for an upper solution hold.

We have  \((\bar u,\bar v,\bar g,\bar h)\) is an upper solution of \eqref{eq:main-problem} on its time interval, that the initial data satisfy \((u_0,v_0)\leq(\bar u(\cdot,0),\bar v(\cdot,0))\) on \([-h_0,h_0]\), and that \(\bar g(0)=-h_0\leq g(0),\ \bar h(0)=h_0\geq h(0)\). Therefore by the comparison principle (Lemma~\eqref{ComparaisonlEMMA}) the  solution is dominated by the upper solution:
\[
g(t)\geq \bar g(t),\qquad h(t)\leq\bar h(t),\qquad u(x,t)\leq\bar u(x,t),\qquad v(x,t)\leq\bar v(x,t),
\]
for as long as the upper solution is defined. Passing to the limit \(t\to\infty\) gives
\[
h_\infty-g_\infty \leq \bar h(\infty)-\bar g(\infty)=2h_1 < 2\ell^\ast,
\]
hence the habitat length remains bounded and by Lemma~\eqref{lem:vanishing} vanishing occurs: \((u,v)\to(0,0)\) uniformly and the free interval does not spread.
\end{proof}

\begin{remark}\label{Remarkvanishing}
From the expression of $\mathbf{M}$ in \eqref{MMEQ1}, it follows that $\mathbf{M} \to \infty$ as $\mu \to 0,~\rho \rightarrow 0$.
Hence, under the assumptions $R_0 > 1$ and $h_0 < \mathcal{L}$,
for any given pair of initial functions $(u_0,v_0)$ satisfying
\[
u_{0}, v_{0} \in C\!\left([-h_{0},h_{0}]\right),
\quad v_{0}(\pm h_{0})=u_{0}(\pm h_{0}) = 0,
\quad u_{0}, v_{0} > 0 \ \text{for } x \in (-h_{0},h_{0}),
\]
there exists a constant $\mu_0 > 0$ such that vanishing occurs whenever $\mu \in (0,\mu_0]$.
\end{remark}
\begin{lemma}\label{samfreeboundar}
We have $h_\infty < \infty$ if and only if $g_\infty > -\infty$.
\end{lemma}

\begin{proof}
The argument follows along the same lines as the proof of Lemma~4.5 in \cite{Du}. Since the required adjustments are straightforward, we leave out the details.
\end{proof}

\subsection{Spreading}

In this section, we investigate the scenarios in which spreading occurs.

\begin{lemma}
Suppose that $R_0 > 1$ and $h_0 > \mathcal{L}$. Then the free boundaries expand without bound, namely
\[
h_\infty = -g_\infty = \infty,
\]
and the solution converges to the positive steady state:
\[
\lim_{t \to \infty} \bigl(u(t,x), v(t,x)\bigr) = \bigl(\mathscr{U}^\ast, \mathscr{V}^\ast\bigr),
\]
where $\bigl(\mathscr{U}^\ast, \mathscr{V}^\ast\bigr)$ denotes the unique positive solution of \eqref{PS1}--\eqref{PS2}.
\end{lemma}

\begin{proof}
We split the proof into two parts.

\medskip\noindent\textbf{(A) Spreading: $h_\infty-g_\infty=\infty$ and $h_\infty=-g_\infty=\infty$.}
Since \(R_0>1\) and \(h_0>\mathcal L^\ast\), by Lemma~ \eqref{lem:vanishing} the principal eigenvalue associated with the linearized problem on any interval \([-\mathcal{Z}, \mathcal{Z}]\) with \(\mathcal{Z} \geq h_0\) satisfies \(\lambda^\ast(\mathcal{Z})>0\).

$(\lambda^\ast(g_\infty,h_\infty) \geq \lambda^\ast_1(g(T),h(T)) > 0)$. Therefore $\lambda^\ast_1(g_\infty,h_\infty) > 0$.
Now apply Lemma~\eqref{Vanishfreeboundaries}, the contrapositive of that statement yields that if \(\lambda^\ast_1(g_\infty,h_\infty)>0\) then \(h_\infty-g_\infty=\infty\). Hence, the habitat length diverges:
\[
h_\infty-g_\infty=\infty.
\]

By using Lemma ~\eqref{samfreeboundar}, we obtain
\[
h_\infty=-g_\infty=\infty.
\]

\medskip\noindent\textbf{(B) Convergence to the positive steady state.}
We prove first a lower bound (limit inferior) and then an upper bound (limit superior); together they give the desired convergence.

\emph{Limit inferior.} Fix any \(T>0\). Consider the fixed–domain problem on \([g(T),h(T)]\) with initial data
\[
\big(\tilde u(x,0),\tilde v(x,0)\big)=(u(x,T),v(x,T)),\qquad x\in[g(T),h(T)].
\]
Let \((\tilde u(x,t),\tilde v(x,t))\) denote the solution of this fixed–domain problem. Because \(\lambda_1^\ast(g(T),h(T))>0\) Theorem~\eqref{thm:existence-stability} (i) implies that \((\tilde u,\tilde v)\) converges to the unique positive steady state \((u_{T}(x),v_{T}(x))\) on \([g(T),h(T)]\):
\[
\lim_{t\to\infty}(\tilde u(x,t),\tilde v(x,t))=(u_T(x),v_T(x))\quad\text{uniformly on }[g(T),h(T)].
\]

By the comparison principle \eqref{ComparaisonlEMMA} and the fact that the free–boundary solution on \([g(t),h(t)]\) has domain containing \([g(T),h(T)]\) for \(t\geq0\), we have that
\[
(\tilde u(x,t),\tilde v(x,t)) \leq \bigl(u(x,t+T),v(x,t+T)\bigr)\qquad\text{for }x\in[g(T),h(T)],\ t\geq0.
\]

Passing to the limit \(t\to\infty\) in the last inequality yields
\[
(u_T(x),v_T(x)) \leq \liminf_{t\to\infty}\bigl(u(x,t),v(x,t)\bigr)
\qquad\text{for }x\in[g(T),h(T)].
\]
    Finally send \(T\to\infty\). Because \(g(T)\downarrow -\infty\), \(h(T)\uparrow\infty\) and by Proposition~\eqref{Thecoexistence} the family \((u_T,v_T)\) converges locally uniformly to the spatially homogeneous positive equilibrium \((\mathscr U^\ast,\mathscr V^\ast)\) as \(T\to\infty\), we deduce the global lower bound
\[
(\mathscr U^\ast,\mathscr V^\ast) \leq \liminf_{t\to\infty}\bigl(u(x,t),v(x,t)\bigr)
\]
locally uniformly in \(x\in\mathbb R\).

\emph{Limit superior.} Let \((\overline u(t),\overline v(t))\) be the solution of the *spatially homogeneous* ODE system (the reaction-only system)
\begin{equation}\label{ODEequations}
\begin{cases}
\overline u' = -a\overline u + H(\overline v),\\
\overline v' = -b\overline v + G(\overline u),
\end{cases}
\qquad
(\overline u(0),\overline v(0))=(\|u_0\|_\infty,\|v_0\|_\infty).
\end{equation}
Because the nonlocal diffusion terms and drift terms do not produce positive sources beyond what the ODE receives (and because kernels integrate to \(1\)), one easily checks that \((\overline u,\overline v)\) is an upper solution to the full free–boundary problem: for each \(t\geq0\) and \(x\in[g(t),h(t)]\),
\[
u(x,t)\leq \overline u(t),\qquad v(x,t)\leq \overline v(t).
\]

Since \(R_0>1\), the ODE system \eqref{ODEequations} has the globally attractive positive equilibrium \((\mathscr U^\ast,\mathscr V^\ast)\), and
\[
\lim_{t\to\infty}(\overline u(t),\overline v(t))=(\mathscr U^\ast,\mathscr V^\ast).
\]
Therefore
\[
\limsup_{t\to\infty}\bigl(u(x,t),v(x,t)\bigr) \leq (\mathscr U^\ast,\mathscr V^\ast),
\]
uniformly in \(x\in[g(t),h(t)]\).

\medskip\noindent Combining the limit inferior and limit superior bounds we obtain
\[
\lim_{t\to\infty}\bigl(u(x,t),v(x,t)\bigr)=(\mathscr U^\ast,\mathscr V^\ast),
\]
locally uniformly in \(x\in\mathbb R\).
\end{proof}

\begin{lemma}\label{lem:mu-large-spreading}
Assume $R_0>1$ and $h_0<\mathcal L^\ast$. Then there exists $\mu_0>0$ such that for every $\mu>\mu_0$ the solution of \eqref{eq:main-problem} spreads.
\end{lemma}

\begin{proof}
Write explicitly the dependence on $\mu$ and denote the corresponding solution and free boundaries by
\[
(u_\mu,v_\mu,g_\mu,h_\mu).
\]
Argue by contradiction: suppose that for every $\mu>0$ vanishing occurs, i.e. the limits
\begin{equation}\label{eq:limits-mu}
g_{\mu,\infty}:=\lim_{t\to\infty} g_\mu(t),\qquad
h_{\mu,\infty}:=\lim_{t\to\infty} h_\mu(t)
\end{equation}
exist and are finite for every $\mu>0$.

By the comparison principle (Lemma~\eqref{ComparaisonlEMMA}) the solution \((u_\mu,v_\mu,g_\mu,h_\mu)\) is monotone in \(\mu\): for \(\mu_1<\mu_2\) we have
\[
u_{\mu_1}(x,t)\leq u_{\mu_2}(x,t),\quad v_{\mu_1}(x,t)\leq v_{\mu_2}(x,t),
\quad g_{\mu_1}(t)\geq g_{\mu_2}(t),\quad h_{\mu_1}(t)\leq h_{\mu_2}(t),
\]
for all admissible \(x,t\). Consequently the pointwise limits
\[
\mathcal{G}_\infty:=\lim_{\mu\to\infty} g_{\mu,\infty},\qquad
\mathcal{H}_\infty:=\lim_{\mu\to\infty} h_{\mu,\infty}
\]
exist (possibly \(-\infty\) / \(+\infty\) but finite under our contradiction assumption). By Lemma~\eqref{lem:vanishing} the principal eigenvalue on \([\mathcal{G}_\infty, \mathcal{H}_\infty]\) satisfies \(\lambda^\ast_1(\mathcal{G}_\infty, \mathcal{H}_\infty)\leq 0\). On the other hand, since \(R_0>1\), Lemma~3.9 (ii) ensures that large enough intervals admit positive eigenvalue, so in particular (by monotonicity with respect to domain size) we get \(H_\infty-G_\infty\leq 2\mathcal L^\ast\) (i.e. the limiting length is bounded). Thus \(G_\infty, H_\infty\) are finite numbers.

Fix small \(\varepsilon>0\) such that the kernels satisfy positivity near the origin: by (J) there exists \(\delta_0>0\) with
\[
J_i(\xi)\ge\delta_0\quad\text{for }|\xi|\le 4\varepsilon,\quad i=1,2.
\]

Because \(h_{\mu,\infty}\uparrow H_\infty\) and \(g_{\mu,\infty}\downarrow G_\infty\) as \(\mu\to\infty\), choose \(\mu_1>0\) large so that for every \(\mu\ge\mu_1\)
\[
h_{\mu,\infty} > H_\infty - \tfrac{\varepsilon}{2},\qquad
g_{\mu,\infty} < G_\infty + \tfrac{\varepsilon}{2}.
\]
Since \(h_{\mu_1}(t)\to h_{\mu_1,\infty}\) and \(g_{\mu_1}(t)\to g_{\mu_1,\infty}\) as \(t\to\infty\), we can pick \(t_1>0\) large such that
\[
h_{\mu_1}(t)\geq h_{\mu_1,\infty}-\tfrac{\varepsilon}{2} > \mathcal{H}_\infty - \varepsilon,\qquad
g_{\mu_1}(t)\leq g_{\mu_1,\infty}+\tfrac{\varepsilon}{2} < \mathcal{G}_\infty +\varepsilon,
\]
for all \(t\geq t_1\).

Now use the free–boundary formula (the right boundary; the left boundary is handled analogously). For any \(\mu>0\) and any \(t_1>0\),
\[
h_{\mu,\infty}-h_{\mu}(t_1)
= \int_{t_1}^{\infty} h_\mu'(s)\,ds
= \mu\int_{t_1}^{\infty}\Big[ \int_{g_\mu(s)}^{h_\mu(s)}\int_{h_\mu(s)}^{\infty} J_1(x-y)u_\mu(x,s)\,dy\,dx
+ \rho\int_{g_\mu(s)}^{h_\mu(s)}\int_{h_\mu(s)}^{\infty} J_2(x-y)v_\mu(x,s)\,dy\,dx\Big]ds.
\]
Hence
\[
\mu \;=\; \frac{h_{\mu,\infty}-h_{\mu}(t_1)}
{\displaystyle \int_{t_1}^{\infty}\Big[\int_{g_\mu(s)}^{h_\mu(s)}\int_{h_\mu(s)}^{\infty} J_1(x-y)u_\mu(x,s)\,dy\,dx
+\rho\int_{g_\mu(s)}^{h_\mu(s)}\int_{h_\mu(s)}^{\infty} J_2(x-y)v_\mu(x,s)\,dy\,dx\Big]ds }.
\label{eqaaasss}
\]

We will obtain a uniform lower bound on the denominator in \eqref{eqaaasss}) for all \(\mu\geq\mu_1\), which will contradict the fact that \(\mu\) can be arbitrarily large.

Observe that for \(\mu\geq\mu_1\) we have monotonicity \(u_\mu\geq u_{\mu_1}\) and \(v_\mu\geq v_{\mu_1}\) pointwise. Therefore for every \(s\geq t_1\),
\[
\begin{aligned}
&I_\mu(s):=\int_{g_\mu(s)}^{h_\mu(s)}\int_{h_\mu(s)}^{\infty} J_1(x-y)u_\mu(x,s)\,dy\,dx \\
&\qquad\geq \int_{h_\mu(s)-2\varepsilon}^{h_\mu(s)}\int_{h_\mu(s)}^{h_\mu(s)+2\varepsilon} J_1(x-y)u_{\mu_1}(x,s)\,dy\,dx
\geq \delta_0 (2\varepsilon)^2 \min_{x\in[h_\mu(s)-2\varepsilon,h_\mu(s)]} u_{\mu_1}(x,s).
\end{aligned}
\]
A completely analogous bound holds for the \(v\)-term:
\[
J_\mu(s):=\int_{g_\mu(s)}^{h_\mu(s)}\int_{h_\mu(s)}^{\infty} J_2(x-y)v_\mu(x,s)\,dy\,dx
\geq \delta_0 (2\varepsilon)^2 \min_{x\in[h_\mu(s)-2\varepsilon,h_\mu(s)]} v_{\mu_1}(x,s).
\]

By our choice of \(\mu_1\) and \(t_1\) we have \(h_{\mu_1}(t)\geq \mathcal{H}_\infty-\varepsilon\) for all \(t\geq t_1\), and hence the intervals \([h_{\mu_1}(t)-2\varepsilon,h_{\mu_1}(t)]\) are contained in \([\mathcal{H}_\infty-3\varepsilon,\mathcal{H}_\infty+\varepsilon]\) for \(t\geq t_1\). The functions \(u_{\mu_1},v_{\mu_1}\) are continuous in space-time and nonnegative; they are not identically zero on the time-space slab \([t_1,t_1+1]\times [\mathcal{H}_\infty-3\varepsilon,\mathcal{H}_\infty+\varepsilon]\). Therefore the continuous function
\[
s \mapsto m(s) := \min_{x\in[h_{\mu_1}(s)-2\varepsilon,h_{\mu_1}(s)]}\big(u_{\mu_1}(x,s)+\rho v_{\mu_1}(x,s)\big)
\]
is nonnegative and not identically zero on \([t_1,t_1+1]\). Consequently
\[
\int_{t_1}^{t_1+1} \big( I_\mu(s) + \rho J_\mu(s)\big)\,ds
\geq \delta_0 (2\varepsilon)^2 \int_{t_1}^{t_1+1} m(s)\,ds =: c_*>0,
\]
where \(c_*>0\) depends only on \(\mu_1,t_1,\varepsilon,\delta_0\) and the fixed solution \((u_{\mu_1},v_{\mu_1})\), but not on \(\mu\).

Hence the denominator in \eqref{eqaaasss} is bounded from below by \(c_*\). The numerator \(h_{\mu,\infty}-h_{\mu}(t_1)\) is bounded above uniformly in \(\mu\) by \(\mathcal{H}_\infty - (\mathcal{H}_\infty-\varepsilon) = \varepsilon\) (indeed \(h_{\mu,\infty}\leq \mathcal{H}_\infty\) and \(h_{\mu}(t_1)\geq \mathcal{H}_\infty-\varepsilon\) for \(\mu\geq\mu_1\) by construction). Therefore for every \(\mu\geq\mu_1\) we deduce from \eqref{eqaaasss}
\[
\mu \leq \frac{h_{\mu,\infty}-h_{\mu}(t_1)}{\int_{t_1}^{\infty}\big(I_\mu(s)+\rho J_\mu(s)\big)\,ds}
\le \frac{\mathcal{H}_\infty - (\mathcal{H}_\infty-\varepsilon)}{c_*} = \frac{\varepsilon}{c_*} <\infty.
\]
This gives a uniform upper bound for all \(\mu\geq\mu_1\), contradicting the fact that \(\mu\) can be taken arbitrarily large. The contradiction arises from the assumption that vanishing occurs for every \(\mu>0\); hence there exists \(\mu_0>0\) such that spreading occurs for all \(\mu>\mu_0\). This completes the proof.
\end{proof}

\textbf{Proof of Theorem \ref{thm:threshold-mu}}

\begin{proof}
For each fixed \(\mu>0\) denote by \((u_\mu,v_\mu,g_\mu,h_\mu)\) the corresponding solution of \eqref{eq:main-problem}.  Define
\[
\mathcal E := \big\{\mu>0:\ h_{\mu,\infty}-g_{\mu,\infty}\leq 2\mathcal L^\ast\big\},
\qquad
\text{where }g_{\mu,\infty}:=\lim_{t\to\infty}g_\mu(t),\
h_{\mu,\infty}:=\lim_{t\to\infty}h_\mu(t).
\]
By the definition and Lemma~\eqref{lem:vanishing}, \(\mu\in\mathcal E\) precisely means vanishing occurs for that \(\mu\).  We shall show \(\mathcal E=(0,\hat{\mu}]\) for some \(\hat{\mu}>0\).

\medskip\noindent\textbf{1. Nonempty and unbounded side excluded.}
By Lemma \eqref{lem:vanishing} there exists \(\mu_0>0\) such that \((0,\mu_0]\subset\mathcal E\).  By Lemma~\eqref{lem:mu-large-spreading} (spreading for large \(\mu\)) there exists \(\overline\mu>0\) such that for every \(\mu>\overline\mu\) spreading occurs, hence \((\overline\mu,\infty)\cap\mathcal E=\varnothing\).  Therefore \(\mathcal E\) is a nonempty subset of \((0,\infty)\) bounded above, so the supremum
\[
\hat{\mu} := \sup\mathcal E \in [\mu_0,\overline\mu]
\]
is well-defined and finite.

\medskip\noindent\textbf{2. Spreading for \(\mu>\hat{\mu} \).}
Fix any \(\mu>\tilde{\mu} \). By definition of supremum there exists \(\tilde\mu\in\mathcal E\) with \(\tilde\mu<\mu\) arbitrarily close to \(\mu^\ast\). Monotonicity of the solution in \(\mu\)  implies that the larger parameter yields larger solution and larger spreading ability, so if vanishing holds at \(\tilde\mu\) then it might still hold at \(\mu\); however Lemma~\eqref{lem:mu-large-spreading} guarantees that for sufficiently large \(\mu\) vanishing cannot persist. More directly, since \(\mu>\hat\mu\geq\sup\mathcal E\), we must have \(\mu\not\in\mathcal E\), i.e. vanishing does not occur at \(\mu\). Thus spreading occurs for every \(\mu>\hat\mu\).

\medskip\noindent\textbf{3. The supremum \(\hat\mu\) belongs to \(\mathcal E\).}
We next show \(\hat\mu\in\mathcal E\). Suppose instead \(\hat\mu\notin\mathcal E\). Then vanishing fails at \(\mu=\hat\mu\), so spreading occurs for \(\mu=\hat\mu\). In particular there exists \(T>0\) such that
\[
h_{\hat\mu}(T)-g_{\hat\mu}(T) > 2\mathcal L^\ast.
\]
By the continuous dependence of solutions on the parameter \(\mu\) (standard parabolic continuous dependence: the map \(\mu\mapsto (u_\mu,v_\mu,g_\mu,h_\mu)\) is continuous in the appropriate norms on compact time intervals), there exists \(\varepsilon>0\) such that for every \(\mu\in(\hat\mu-\varepsilon,\hat\mu+\varepsilon)\)
\[
h_{\mu}(T)-g_{\mu}(T) > 2\mathcal L^\ast.
\]
For any such \(\mu\) the monotonicity of \(h_\mu(t)\) and \(-g_\mu(t)\) in \(t\) implies
\[
\lim_{t\to\infty}\big(h_\mu(t)-g_\mu(t)\big) \geq h_\mu(T)-g_\mu(T) > 2\mathcal L^\ast,
\]
so \(\mu\notin\mathcal E\). Consequently \(\mathcal E\subset(0,\hat\mu-\varepsilon)\), contradicting the fact that \(\hat\mu=\sup\mathcal E\). Hence our assumption was false and we must have \(\hat\mu\in\mathcal E\).

\medskip\noindent\textbf{4. Vanishing for \(\mu\leq\hat\mu\).}
Finally we show that every \(\mu\in(0,\mu^\ast]\) belongs to \(\mathcal E\). Let \(\mu\in(0,\hat\mu)\). By monotonicity in \(\mu\)  the solution at \(\mu\) is dominated by the solution at \(\hat\mu\):
\[
u_\mu(x,t)\leq u_{\hat\mu}(x,t),\quad v_\mu(x,t)\leq v_{\hat\mu}(x,t),
\qquad g_\mu(t)\geq g_{\mu^\ast}(t),\quad h_\mu(t)\leq h_{\mu^\ast}(t).
\]
Since \(\hat\mu\in\mathcal E\) we have \(h_{\hat\mu,\infty}-g_{\mu^\ast,\infty}\leq 2\mathcal L^\ast\), and therefore passing to limits in the inequalities above yields
\[
h_{\mu,\infty}-g_{\mu,\infty}\le h_{\mu^\ast,\infty}-g_{\mu^\ast,\infty}\le 2\mathcal L^\ast,
\]
i.e. \(\mu\in\mathcal E\). Because \(\mu\in(0,\hat\mu)\) was arbitrary, we conclude \((0,\hat\mu)\subset\mathcal E\). Combined with \(\mu^\ast\in\mathcal E\), we obtain \(\mathcal E=(0,\mu^\ast]\).

\medskip\noindent\textbf{5. Final statement.}
From the previous steps we have shown that vanishing occurs for \(\mu\in(0,\hat\mu]\) and spreading occurs for \(\mu\in(\hat\mu,\infty)\). The asymptotic profile in the spreading case (convergence to \((\mathscr U^\ast,\mathscr V^\ast)\)) follows from  Theorem~\eqref{thm:existence-stability}  and Lemma \eqref{lem:mu-large-spreading}(which give that once spreading happens the solution converges to the unique positive steady state). This completes the proof.
\end{proof}

\section*{Conclusion and Perspectives}

In this work, we have investigated a free-boundary problem arising from a cooperative nonlocal reaction–diffusion system with drift terms.
Our analysis unfolds through several key steps.

First, by ingeniously combining fixed-point techniques with refined comparison principles,
we established the \emph{well-posedness} of the problem, ensuring the existence, uniqueness, and regularity of global classical solutions
for all admissible initial data.

Second, we conducted a comprehensive spectral analysis of the associated \emph{linearized eigenvalue problem}
and demonstrated the existence and uniqueness of a principal eigenvalue together with a strictly positive eigenfunction.
We further derived its fundamental qualitative properties with respect to the spatial domain, to the parameters thanks to deep theoritical results as Fredholm theory, the Crandall–Rabinowitz bifurcation theorem, and Hadamard-type derivative \cite{Benguria2024} and
and uncovered a sharp equivalence between the sign of the principal eigenvalue and the basic reproduction number~$R_0$.  This connection provides a rigorous spectral criterion for the threshold between extinction and persistence,
thereby linking the linearized spectral structure to the nonlinear long-term dynamics of the system.

Finally, by a detailed investigation of the free-boundary evolution,
we established a complete \emph{vanishing–spreading dichotomy}.
When the initial domain is sufficiently small ($h_0<\mathcal L^\ast$),
there exists a critical threshold $\widehat{\mu}>0$ such that vanishing occurs for $\mu\in(0,\widehat{\mu}]$,
whereas spreading occurs for $\mu>\widehat{\mu}$.
In the spreading regime, the solution converges to the unique positive steady state
$(\mathscr U^\ast,\mathscr V^\ast)$,
while in the vanishing regime it decays uniformly to zero.
These results provide a complete and rigorous characterization of the asymptotic behavior of the system.

\medskip
\noindent\textbf{Perspectives research.}
The present work opens several directions for future research.
A natural extension would be to incorporate more general nonlinear incidence functions~$f(u,v)$,
capturing saturation or density-dependent transmission effects.
Another promising direction concerns competitive or predator–prey systems governed by free boundaries,
in contrast with the cooperative framework analyzed here,
which could shed new light on coexistence, exclusion, and spatial invasion phenomena in ecological systems.
From a broader viewpoint, extending the present analysis to higher-dimensional settings,
introducing stochastic perturbations, or accounting for temporal and spatial heterogeneities
would substantially enrich the theoretical framework
and enhance its relevance to realistic epidemic and ecological scenarios. An important extension of the above model is to incorporate the influence of a moving climatic envelope on the directional epidemic,
which represents a promising avenue for future research.
To this end, we allow the local incidence functions to depend on the shifted spatial coordinate $x-ct$, and consider the following free boundary system :

\begin{equation}\label{eq:fb-omega-compact2}
\left\{
\begin{aligned}
u_t &= d_1\!\!\int_{\Omega_t}\! J_1(x-y)\,u(t,y)\,dy - u(t,x)
       + p\,u_x - a(x)\,u(t,x) + H\!\big(x-ct,\,v(t,x)\big),
       && t>0,\ x\in\Omega_t, \\[3pt]
v_t &= d_2\!\!\int_{\Omega_t}\! J_2(x-y)\,v(t,y)\,dy - v(t,x)
       + q\,v_x - b(x)\,v(t,x) + G\!\big(x-ct,\,u(t,x)\big),
       && t>0,\ x\in\Omega_t, \\[3pt]
u(t,x)&=v(t,x)=0, && t>0,\ x\in\partial\Omega_t, \\[3pt]
h'(t) &= \mu\!\!\int_{\Omega_t}\!\!\int_{(h(t),\infty)}\! J_1(x-y)\,u(t,x)\,dy\,dx
        + \rho\!\!\int_{\Omega_t}\!\!\int_{(h(t),\infty)}\! J_2(x-y)\,v(t,x)\,dy\,dx,
        && t>0, \\[3pt]
g'(t) &= -\,\mu\!\!\int_{\Omega_t}\!\!\int_{(-\infty,g(t))}\! J_1(x-y)\,u(t,x)\,dy\,dx
        + \rho\!\!\int_{\Omega_t}\!\!\int_{(-\infty,g(t))}\! J_2(x-y)\,v(t,x)\,dy\,dx,
        && t>0, \\[3pt]
u(0,x)&=u_0(x),\quad v(0,x)=v_0(x),\quad g(0)=-h_0,\ h(0)=h_0.
\end{aligned}
\right.
\end{equation}
where  $\Omega_t := (g(t),\,h(t)) \text{ and } \partial\Omega_t = \{g(t),\,h(t)\},$ $c\in\mathbb{R}$ denotes the velocity of the climatic drift. This potential study develops  the original models suggested by many authors, for instance \cite{Berestycki2008, Berestycki2016b, Coville2021, Vo2015}, to investigate the long time dynamics of single population under the influence of climate change without free boundary setting.

%The KPP term

%Inspired of Ducrot and Magal \cite{ducrot_traveling_2011}, the  term $\pi(a)mu$ in KPP can be interpreted as modeling external sources, such as the migration of individuals into the population.

%\item \label{cond:cond1}  $\gamma := \gamma(a)$ is  continuous, non-negative on $[0,A]$ and satisfies ${\displaystyle \int^{A}_0\gamma(a)da = 1}$.
%
%\item \label{cond:cond2}$\pi := \pi(a)$ is a $C^{\infty}$ non-constant, positive, bounded function on $\mathbb{R}^{+}$.
%
%\item \label{cond:cond3} $m,~n,~p,~q$ are positive constants.

\section*{Statements}

%\begin{center}
\textbf{Conflict of interest statement :} The authors have no conflicts of interest to declare that are relevant
to the content of this article.

\textbf{Data availability statement :} Data sharing not applicable to this article as no datasets were generated
or analyzed during the current study.
%\end{center}

%\bibliographystyle{plain}
%\bibliography{Bib1.bib}

\end{document}